\documentclass[11pt,letterpaper]{amsart}
\usepackage{mathtools}
\usepackage[T1]{fontenc}
\usepackage[latin9]{inputenc}
\usepackage{geometry}
\usepackage{wrapfig}
\usepackage{multicol}
\usepackage{graphicx}
\usepackage{soul}
\usepackage{xcolor}
\usepackage{amssymb}

\newtheorem{theorem}{Theorem}[section]
\newtheorem{lemma}[theorem]{Lemma}
\newtheorem{proposition}[theorem]{Proposition}

{ \theoremstyle{definition}
\newtheorem{definition}[theorem]{Definition}}
{ \theoremstyle{remark}
\newtheorem{remark}[theorem]{Remark}}
\usepackage{placeins}
\usepackage{cite}
\usepackage{caption}
\usepackage{enumerate}
\usepackage{afterpage}
\usepackage{enumitem}
\usepackage{bmpsize}
\usepackage{mathtools}% Loads amsmath

\usepackage{hyperref}
\usepackage{tabu}
\usepackage{enumitem}   
\numberwithin{equation}{section}
\usepackage{stmaryrd}

\def\i{\mathsf{i}}

\def\rr{\mathsf{r}}
\def\cc{\mathsf{c}}
\def\tt{\tau_1}

\def\Re{ \mathsf{Re}}
\def\Im{ \mathsf{Im}}
\def\Ehf {\mathbb{E}^{\operatorname{h-fl}}}
\def\Phf {\mathbb{P}^{\operatorname{h-fl}}}

\usepackage{tikz}
\usetikzlibrary{matrix,graphs,arrows,positioning,calc,decorations.markings,shapes.symbols}

\title{One-point asymptotics for half-flat ASEP}
\date{\today}
\author{Evgeni Dimitrov and Anushka Murthy}

\begin{document}

\maketitle

\begin{abstract} We consider the asymmetric simple exclusion process (ASEP) with half-flat initial condition. We show that the one-point marginals of the ASEP height function are described by those of the $\mbox{Airy}_{2 \rightarrow 1}$ process, introduced by Borodin-Ferrari-Sasamoto in (Commun. Pure Appl. Math., 61, 1603-1629, 2008). This result was conjectured by Ortmann-Quastel-Remenik (Ann. Appl. Probab., 26, 507-548), based on an informal asymptotic analysis of exact formulas for generating functions of the half-flat ASEP height function at one spatial point. Our present work provides a fully rigorous derivation and asymptotic analysis of the same generating functions, under certain parameter restrictions of the model.
\end{abstract}

\tableofcontents

%-------------------------------------------------------------------------------------------------------------------------------------------------------------------------------------------------
% Section 1
%
%-------------------------------------------------------------------------------------------------------------------------------------------------------------------------------------------------
\section{Introduction and main results}\label{Section1}

%-------------------------------------------------------------------------------------------------------------------------------------------------------------------------------------------------
% Section 1.1
%
%-------------------------------------------------------------------------------------------------------------------------------------------------------------------------------------------------
\subsection{Preface}\label{Section1.1} The {\em asymmetric simple exclusion process} (ASEP) is a continuous time Markov process, introduced to the mathematical community by Spitzer \cite{Spitzer70} in 1970 (and also appearing two years earlier in the biology work of MacDonald, Gibbs and Pipkin \cite{MGP68}). 

The state space of the process is $\{0,1\}^{\mathbb{Z}}$ and we interpret an element $\eta \in \{0,1\}^{\mathbb{Z}}$ as a particle configuration, where $\eta(x) = 1$ indicates the presence of a particle at location $x$, and $\eta(x) = 0$ indicates the presence of a hole at location $x$. The dynamics of the model depend on a parameter $p \in [0,1]$, and can be described as follows. Each particle carries an independent exponential clock, which rings at rate $1$. When the clock rings the particle attempts to jump one site to the right with probability $p$, and with probability $q = 1-p$ it attempts to jump one site to the left. The jump is successful if the site to which the particle attempts to jump is a hole; otherwise, the jump is suppressed and the particle stays put. As mentioned in \cite{Spitzer70}, there is some delicacy in formally constructing a Markov process corresponding to the latter dynamics when the number of particles is infinite; however, a number of papers have been published confirming the existence of such a Markov process -- see the works of Harris \cite{Harris72}, Holley \cite{Holley70}, and Liggett \cite{Liggett72}. We denote this process by $\{ \eta_t: t \geq 0 \}$. If $q = 1, p = 0$ (or $q= 0, p = 1$) the process $\{ \eta_t: t \geq 0 \}$ is called the {\em totally} asymmetric simple exclusion process (TASEP), if $q = p = 1/2$ it is called the {\em symmetric} simple exclusion process (SSEP), and if $ p, q > 0$ and $p \neq q$ it is called the (partially) asymmetric simple exclusion process (ASEP). Throughout the paper we will assume that $q > p > 0$, so that the particles have a drift to the left.

Given $\{ \eta_t: t \geq 0 \}$, we define $\{ \hat{\eta}_t: t \geq 0 \}$ through $\hat{\eta}_t(x) = 2 \eta_t(x) - 1$ so that $\hat{\eta}_t \in \{-1, 1\}^{\mathbb{Z}}$, and then we define the {\em height function} of ASEP to be
\begin{equation}\label{HGen}
h(t,x) = \begin{cases} 2N_0^{\operatorname{flux}}(t) + \sum_{y = 1}^x \hat{\eta}_t(y)  &\mbox{ if } x\geq 1 \\ 2 N_0^{\operatorname{flux}}(t), &\mbox{ if } x = 0 \\ 2N_0^{\operatorname{flux}}(t) - \sum_{y = x + 1}^0 \hat{\eta}_t(y) &\mbox{ if } x \leq -1,\end{cases}
\end{equation}
where $N_0^{\operatorname{flux}}(t)$ is the net number of particles that crossed from site $1$ to $0$ up to time $t$, i.e. particles that jump from $1$ to $0$ are counted as $+1$ and those that jumped from $0$ to $1$ are counted as $-1$. 

ASEP is an important member of the one-dimensional Kardar-Parisi-Zhang (KPZ) universality class. For more on the KPZ universality class we refer to the surveys and books \cite{CU2,HT,QS} and the references therein. In particular, it is expected that the scaled ASEP height function $\epsilon^{1/2} h(\epsilon^{-3/2} T,  \epsilon^{-1}x)$, appropriately shifted, converges as $\epsilon \rightarrow 0+$ to the fixed time $T$ distribution of a certain Markov process known as the {\em KPZ fixed point}, started from an initial state that depends on the initial condition for ASEP. We mention that the KPZ fixed point was constructed by Matetski-Quastel-Remenik \cite{MQR}, and the convergence of the ASEP height function to the KPZ fixed point was proved for a class of initial conditions by Quastel-Sarkar \cite{QuaSar}.\\

Despite the remarkable progress in understanding the asymptotic behavior of ASEP, the results in \cite{QuaSar} only apply when the initial condition asymptotically is either continuous with moderate growth, or consists of multiple narrow wedges. There are still many natural and important initial conditions that are not handled by \cite{QuaSar}, and for which precise statements are not available. One such example is the case of {\em half-flat} initial condition, which corresponds to starting ASEP from $\eta^{\operatorname{h-fl}}_0 = {\bf 1}\{ x \in 2\mathbb{N}\}$ , where $\mathbb{N} = \{1, 2,3, \dots \}$. The half-flat initial condition asymptotically becomes the function that is $0$ for positive and $-\infty$ for negative values, so that it is a natural mixture of the two classes of initial conditions handled in \cite{QuaSar}, without belonging to either. 

The goal of the present paper is to investigate the large time limit of the ASEP height function $h(t,x)$, started from $\eta^{\operatorname{h-fl}}_0$, and show that its one-point marginals are asymptotically described by those of the Airy$_{2 \rightarrow 1}$ process, introduced by Borodin-Ferrari-Sasamoto \cite{BFS08}. This convergence result was conjectured by Ortmann-Quastel-Remenik \cite{OQR}, based on an informal analysis of exact formulas for generating functions of the half-flat ASEP height function at one spatial point. Our present work provides a fully rigorous derivation and asymptotic analysis of the same generating functions, under certain parameter restrictions of the model.

The rest of the introduction is structured as follows. In Section \ref{Section1.2} we introduce some relevant notation and state a certain moment formula derived in \cite{OQR} -- this is Proposition \ref{S1Prop1}. Within the same section we state the exact formulas for the generating functions of the half-flat ASEP height function as Theorem \ref{S1Thm1}. We mention that Theorem \ref{S1Thm1} already appeared in \cite{OQR}, but as explained in Remarks \ref{Issue2} and \ref{CompOQR}, there are aspects of the derivation of that theorem that are not justified in \cite{OQR}. In fact, we can only prove the well-posedness of our formulas in Theorem \ref{S1Thm1} when $p$ is assumed to be sufficiently close to zero. In Section \ref{Section1.3} we introduce the $\mbox{Airy}_{2 \rightarrow 1}$ process, and in Section \ref{Section1.4} we present our main asymptotic result.

%-------------------------------------------------------------------------------------------------------------------------------------------------------------------------------------------------
% Section 1.2
%
%-------------------------------------------------------------------------------------------------------------------------------------------------------------------------------------------------
\subsection{Exact formulas for half-flat ASEP}\label{Section1.2} We begin by summarizing some of the notation we require for ASEP in the following definition.
\begin{definition}\label{DefASEP} We fix $p \in (0, 1/2)$ and let $q = 1-p$, $\tau = p/q$, $\gamma = q- p$, so that $p,q,\tau,\gamma \in (0,1)$. Let $\{ \eta_t: t \geq 0 \}$ denote the ASEP started from half-flat initial condition $\eta^{\operatorname{h-fl}}_0 = {\bf 1}\{ x \in 2\mathbb{N}\}$ for the parameters $p,q$ as in Section \ref{Section1.1}. We denote the distribution of $\{ \eta_t: t \geq 0 \}$ from this initial condition by $\Phf$ and write $\Ehf$ for the expectation with respect to this measure. For $x \in \mathbb{Z}$ and $t > 0$ we define
\begin{equation}\label{DefN}
N_x(t) = \sum_{y = -\infty}^x \eta_t(y),
\end{equation}
to be the total number of particles at or to the left of location $x$ at time $t$, and note that from (\ref{HGen}) we have the following relationship between $N_x(t)$ and the height function
\begin{equation}\label{DefH}
h(t,x) = 2N_x(t) - x.
\end{equation}
\end{definition}
In the remainder of this section we present the exact formulas for certain observables involving $N_x(t)$ that were derived in \cite{OQR}. We mention that in view of (\ref{DefH}) all of the observables below can alternatively be expressed in terms of $h(t,x)$; however, we will formulate them with $N_x(t)$, following \cite{OQR}. Both the observables and the formulas for them depend on various special functions, which we present next.

For $q \in [0,1)$ and $a \in \mathbb{C}$ we let $(a;q)_{\infty}$ be the $q$-Pochhammer symbol
\begin{equation}\label{S1Poch}
(a;q)_{\infty} = \prod_{n = 0}^\infty (1 - a q^n).
\end{equation}
We also define the $q$-factorial for $k \in \mathbb{Z}_{\geq 0}$
\begin{equation}\label{S1QFac}
k_q! = \frac{\prod_{a = 1}^k (1 - q^a)}{(1- q)^k},
\end{equation}
and the $q$-exponential function
\begin{equation}\label{S1QExp}
e_q(x) = \frac{1}{((1-q)x;q)_{\infty}} = \sum_{k = 0}^{\infty} \frac{x^k}{k_q!},
\end{equation}
where the first identity is the definition of $e_q(x)$ and is valid for $x \neq (1-q)^{-1} q^{-m}$ for $m \in \mathbb{Z}_{\geq 0}$ and the second holds when $|x| < 1$ -- see \cite[Corollary 10.2.2a]{Andrews}.

For the next several functions we fix $p,q,\tau$ as in Definition \ref{DefASEP}, $x \in \mathbb{Z}$, and $t \geq 0$. We define
\begin{equation}\label{S1BasicFun}
\begin{split}
\mathfrak{f}(w; n) =& \hspace{2mm} (1- \tau)^n \exp \left( \frac{(q-p)t}{1 + w} - \frac{(q-p)t}{1 + \tau^n w}  \right) \cdot \left( \frac{1 + \tau^n w}{1 + w} \right)^{x-1}, \\
\mathfrak{g}(w; n) = & \hspace{2mm} \frac{(-w;\tau)_{\infty} (\tau^{2n}w^2;\tau)_{\infty}}{(-\tau^nw;\tau)_{\infty} (\tau^nw^2;\tau)_{\infty}},\\
\mathfrak{h}(w_1,w_2; n_1, n_2) = &\hspace{2mm} \frac{(w_1w_2;\tau)_{\infty} (\tau^{n_1 + n_2} w_1 w_2;\tau)_{\infty}}{(\tau^{n_1} w_1w_2 ;\tau)_{\infty} (\tau^{n_2} w_1 w_2;\tau)_{\infty}}
\end{split}
\end{equation}
Throughout the paper we denote by $\vec{w}$ the vector $\vec{w} = (w_1, \dots, w_k)$ and by $d\vec{w}$ the complex integration form $dw_1 \cdots dw_k$. We mention that the dimension $k$ and the contours over which the integration is performed will be clear from the context. 

For complex vectors $\vec{w}, \vec{s} \in \mathbb{C}^k$ and $\zeta \in \mathbb{C} \setminus [0, \infty)$ we write
\begin{equation}\label{S1SpecFun}
\begin{split}
&F(\vec{s}, \vec{w}) =  \det \left[ \frac{-1}{w_a \tau^{s_a} - w_b}  \right]_{a,b = 1}^k  \prod_{a = 1}^k  \mathfrak{f}(w_a;s_a) \mathfrak{g}(w_a; s_a) \cdot \prod_{1 \leq a < b \leq k} \mathfrak{h}(w_a, w_b; s_a, s_b) ,\\
&F(\zeta; \vec{s}, \vec{w})  = \prod_{a = 1}^k \frac{\pi (-\zeta)^{s_a}}{\sin(-\pi s_a)}  \cdot F(\vec{s}, \vec{w}).
\end{split}
\end{equation}
In equation (\ref{S1SpecFun}) and throughout the paper we will always take the principal branch of the logarithm.\\

With the above notation in place, we can state the exact formulas we use.
\begin{proposition}\label{S1Prop1}\cite[Theorem 1.3]{OQR}. Assume the same notation as in Definition \ref{DefASEP}. For any $m \in \mathbb{Z}_{\geq 0}$ we have
\begin{equation}\label{S1MomentFormula}
\Ehf \left[ \tau^{m N_x(t)}\right] = m_{\tau}! \sum_{k = 0}^m \nu^{\operatorname{h-hl}}_{k,m}(t,x),
\end{equation}
where $\gamma_{-1,0}$ is the positively oriented circle of radius $\tau^{-1/8}$, centered at the origin and
\begin{equation}\label{S1DefNu}
\begin{split}
\nu^{\operatorname{h-hl}}_{k,m}(t,x) = \hspace{2mm} & \frac{1}{k!} \sum_{\substack{ n_1, \dots, n_k \in \mathbb{N} \\ n_1 + \cdots + n_k = m}}\frac{1}{(2\pi \i)^{k}} \oint_{\gamma_{-1,0}^k}d\vec{w} F(\vec{n}, \vec{w}).
\end{split}
\end{equation}
\end{proposition}
\begin{remark} \cite[Theorem 1.3]{OQR} is formulated for general contours $\gamma_{-1,0}$, and as explained in that theorem $\gamma_{-1,0}$ needs to be positively oriented, enclose $0$, $-1$, and exclude all other poles of $F(\vec{n}, \vec{w})$. Our choice for $\gamma_{-1,0}$  to be the positively oriented circle of radius $\tau^{-1/8}$, centered at the origin, clearly satisfies all of these conditions. 
\end{remark}
\begin{remark} Let us briefly explain how Proposition \ref{S1Prop1} is proved in \cite{OQR}. If one defines 
$$\tilde{Q}_{x}(t) = \frac{\tau^{N_x(t)} - \tau^{N_{x-1}(t)}}{\tau -1},$$
it was shown in \cite{BCS14} that $\Ehf \left[ \tilde{Q}_{x_1}(t) \cdots \tilde{Q}_{x_k}(t) \right]$ is the unique solution to a certain system of differential equations. In \cite{OQR} the authors discovered a certain $k$-fold contour integral that solved this system, and thus provided formulas for the observables $\Ehf \left[ \tilde{Q}_{x_1}(t) \cdots \tilde{Q}_{x_k}(t) \right]$. These formulas can be found in \cite[Theorem 1.2]{OQR}, and the derivation is done in Section 2 of that paper. From \cite[Lemma 4.18]{BCS14} there is a way to express $\tau^{m N_x(t)}$ as a linear combination over $\tilde{Q}_{x_1}(t) \cdots \tilde{Q}_{x_k}(t)$ with $0 \leq k \leq m$ and $x_1 < \cdots < x_k \leq x$, which allows one to express $\Ehf \left[ \tau^{m N_x(t)}\right]$ as a certain $m$-fold contour integral over {\em different} contours. This is done in \cite[Proposition 3.2]{OQR}. Deforming all of these contours to the {\em same} one, using a nested contour integral ansatz that is similar to \cite[Proposition 3.8]{BorCor} and whose proof is inspired by \cite{HO97}, one arrives at (\ref{S1MomentFormula}).
\end{remark}

Using the formulas for the moments of $\tau^{N_x(t)}$, one can obtain a formula for the generating series of the {\em $\tau$-Laplace transform} of $\tau^{N_x(t)}$. This formula is given in the following theorem, and forms the basis of our asymptotic analysis.
\begin{theorem}\label{S1Thm1} Assume the same notation as in Definition \ref{DefASEP}, and let $\zeta \in \mathbb{C} \setminus [0, \infty)$. We further suppose that $\tau \in (0,1)$ is sufficiently small, so that 
\begin{equation}\label{S1TauSmall}
\frac{(\tau^{1/2} + \tau^{1/4})}{(1 - \tau^{1/2})^2} \cdot \frac{(-\tau^{3/4};\tau)_{\infty} (-\tau^{3/4};\tau)_{\infty}}{(\tau^{1/4}  ;\tau)_{\infty} (\tau^{1/4};\tau)_{\infty}}  < 1.
\end{equation} Then, for $e_{\tau}$ as in (\ref{S1QExp}) we have 
\begin{equation}\label{S1QLT}
\begin{split}
\Ehf \left[ e_{\tau} \left( \zeta \tau^{N_x(t)} \right) \right] = \hspace{2mm} &1 + \sum_{k = 1}^{\infty}H_{k}(\zeta), \mbox{ where }
\end{split} 
\end{equation} 
\begin{equation}\label{S1KthSum}
\begin{split}
&H_{k}(\zeta) =  \frac{1}{k! (2\pi \i)^{2k}} \int_{(1/2 + \i \mathbb{R})^k} d\vec{s} \oint_{\gamma_{-1,0}^k} d\vec{w} F(\zeta; \vec{s}, \vec{w}),
\end{split}
\end{equation}
$\gamma_{-1,0}$ is as in Proposition \ref{S1Prop1}, and $F(\zeta; \vec{s}, \vec{w})$ is as in (\ref{S1SpecFun}).
\end{theorem}
\begin{remark}\label{RemMT1} Part of the statement of the theorem is that the integral in (\ref{S1KthSum}) is well-defined and finite, and the sum on the right side of (\ref{S1QLT}) is absolutely convergent. We establish these statements in Lemma \ref{S2AnalyticFD}.
\end{remark}
\begin{remark}\label{Issue1} The way that the formula in (\ref{S1QLT}) is derived is by first showing that for $|\zeta| < 1$ 
\begin{equation}\label{AQW1}
\Ehf \left[ e_{\tau} \left( \zeta \tau^{N_x(t)} \right) \right]  = \lim_{N \rightarrow \infty} \sum_{m =0}^N \frac{1}{m_{\tau}!} \cdot \Ehf \left[ \tau^{m N_x(t)}\right],
\end{equation}
which is a direct consequence of (\ref{S1QExp}). Afterwards, we use the formulas for the moments $\Ehf \left[ \tau^{m N_x(t)}\right] $ from (\ref{S1MomentFormula}), take the limit $N \rightarrow \infty$ and rearrange the resulting sum. In this way one obtains 
\begin{equation}\label{AQW2}
\Ehf \left[ e_{\tau} \left( \zeta \tau^{N_x(t)} \right) \right]  =1 + \sum_{k = 1}^{\infty} \sum_{n_1 = 1}^{\infty} \cdots \sum_{ n_k = 1}^{\infty}  \frac{\zeta^{n_1 + \cdots + n_k}}{k! (2\pi \i)^k} \oint_{\gamma_{-1,0}^k}d\vec{w} F(\vec{n}, \vec{w}).
\end{equation}
One can express the $k$-fold sums over $n_1, \dots, n_k$ as $k$-fold contour integrals using a Mellin-Barnes type integral representation from \cite[Lemma 3.20]{BorCor}. This proves (\ref{S1QLT}) for $|\zeta| < 1$, and then one shows that the equality holds for all $\zeta \in \mathbb{C} \setminus [0, \infty)$ by analytic continuation. We mention that in taking the $N \rightarrow \infty$ limit, as well as in applying the analytic continuation argument, we require absolute summability of the series in (\ref{AQW2}), which we can ensure only when $\tau$ is sufficiently small. This is the origin of the assumption (\ref{S1TauSmall}). For general $\tau \in (0,1)$ the integrals in (\ref{S1KthSum}) are still well-defined; however, we do not know how to show that the series in (\ref{S1QLT}) is convergent. The proof of Theorem \ref{S1Thm1} is the content of Section \ref{Section2}.
\end{remark}
\begin{remark}\label{Issue2} Theorem \ref{S1Thm1} appears as Theorem 1.4 in \cite{OQR}, although there is a missing term $\prod_{a = 1}^k \frac{\pi}{\sin(-\pi s_a)}$ in the integrand in (\ref{S1KthSum}) and the result is formulated without the assumption (\ref{S1TauSmall}). The authors derive their result following the outline we gave in Remark \ref{Issue1} and it is when the authors apply this Mellin-Barnes integral representation that the $\prod_{a = 1}^k \frac{\pi}{\sin(-\pi s_a)}$ term is dropped in \cite{OQR}, see \cite[Equation (4.13)]{OQR}. This is a small typo in the paper. A more serious problem, which we discuss in Remark \ref{CompOQR}, is that the authors  do not provide sufficient justification to demonstrate that the series in (\ref{AQW2}) is absolutely convergent for $|\zeta|<1$, and that one can analytically extend the equality in (\ref{S1QLT}) to $\zeta \in \mathbb{C} \setminus [0, \infty)$. In fact, it is not even clear from the arguments in \cite{OQR} that the right side of (\ref{S1QLT}) is convergent and hence well-defined.
\end{remark}

%-------------------------------------------------------------------------------------------------------------------------------------------------------------------------------------------------
% Section 1.3
%
%-------------------------------------------------------------------------------------------------------------------------------------------------------------------------------------------------
\subsection{The $\mbox{Airy}_{2 \rightarrow 1}$ process}\label{Section1.3} In this section we define the $\mbox{Airy}_{2 \rightarrow 1}$ process, which is the object that describes the large time limit of the height function $h$ from (\ref{DefH}) under the measure $\Phf$ from Definition \ref{DefASEP}. The $\mbox{Airy}_{2 \rightarrow 1}$ process was introduced by Borodin-Ferrari-Sasamoto \cite{BFS08}, where it arises from the limit of the TASEP with half-flat initial condition.

Let us first introduce certain complex contours that arise in our discussion.
\begin{definition}\label{S1ContV}
For $a \in \mathbb{C}$ and $\phi \in (0, \pi)$ we define the contour $C_{a,\phi}$ to be the union of $\{a + ye^{-\i\phi)} \}_{y \in \mathbb{R}^+}$ and $\{a + ye^{\i \phi} \}_{y \in \mathbb{R}^+}$ oriented to have increasing imaginary part.
\end{definition}
The following is Definition 2.1 from \cite{BFS08}.
\begin{definition}\label{Airy21} (The $\mbox{Airy}_{2 \rightarrow 1}$ process) Let us set 
\begin{equation}\label{QW1}
\tilde{x} = x - s^2 \cdot {\bf 1}\{s \leq 0\}, \hspace{2mm} \tilde{y} = y - t^2 \cdot {\bf 1}\{t \leq 0\},
\end{equation}
and define the kernel 
\begin{equation}\label{QW2}
K_{\infty}(s, x; t , y) = - \frac{{\bf 1}\{ t > s\}  e^{ - \frac{(\tilde{y} - \tilde{x})^2}{4 (t -s)} }  }{\sqrt{4\pi (t -s )}}+ \frac{1}{(2\pi \i)^2} \int_{C_{1,\pi/4}} \hspace{-2mm} dw \int_{C_{0,3\pi/4}} \hspace{-2mm}  dz \frac{e^{w^3/3 + t w^2 - \tilde{y} w}}{e^{z^3/3 + s z^2 - \tilde{x}z}} \cdot \frac{-2w}{z^2 - w^2}, 
\end{equation}
where $C_{1,\pi/4}, C_{0,3\pi/4}$ are as in Definition \ref{S1ContV}. The {\em Airy$_{2 \rightarrow 1}$ process}, denoted by $\mathcal{A}_{2 \rightarrow 1}$, is the process on $\mathbb{R}$ with $m$-point joint distribution at $t_1 < t_2 < \cdots < t_m$ given by the Fredholm determinant 
\begin{equation}\label{QW3}
\mathbb{P} \left( \cap_{k = 1}^m \{ \mathcal{A}_{2\rightarrow 1}(t_k) \leq y_k \} \right) = \det \left( I - \chi_y K_{\infty} \chi_{y} \right)_{L^2(\{t_1, \dots, t_m\}\times \mathbb{R})}.
\end{equation}
In (\ref{QW3}) we have that the measure on $\{t_1, \dots, t_m\}\times \mathbb{R}$ is the product measure of the counting measure on $\{t_1, \dots, t_m\}$ and the Lebesgue measure on $\mathbb{R}$. The function $\chi_y$ is defined on $\{t_1, \dots, t_m\}\times \mathbb{R}$ through $\chi_y(t_k,x) = {\bf 1}\{ x > y_k\}$. We mention that in \cite[Appendix B]{BFS08} it was shown that there is a trace class kernel on $L^2(\{t_1, \dots, t_m\}\times \mathbb{R})$ that is conjugate to $\chi_y K_{\infty} \chi_{y}$, so that the Fredholm determinant in (\ref{QW3}) is well-defined.
\end{definition}
\begin{remark} In \cite{BFS08} the authors chose a different set of contours in the definition of $K_{\infty}(s, x; t , y)$ in (\ref{QW2}), denoted by $\gamma_+, \gamma_-$ in that paper. Starting from \cite[(2.7)]{BFS08}, we can deform $\gamma_+,\gamma_-$ to $C_{1,\pi/4}, C_{0,3\pi/4}$, respectively, without crossing any poles of the integrand and thus without affecting the value of $K_{\infty}(s, x; t , y)$ by Cauchy's theorem. The decay necessary to deform the contours near infinity comes from the cubic terms in the exponential functions. We also mention that there is an extra minus sign in the integrand in (\ref{QW2}), compared to \cite[(2.7)]{BFS08}, which comes from the fact that $\gamma_+$ in that paper is oriented in the direction if decreasing imaginary part, whereas $C_{1,\pi/4}$ is oriented in the direction of increasing imaginary part.
\end{remark}
\begin{remark}\label{LimitA21} As explained in \cite{BFS08}, we have that the finite dimensional distributions of $\mathcal{A}_{2\rightarrow 1}(t + m)$ converge to those of $2^{1/3}\mathcal{A}_1(2^{-2/3} t)$ as $m \rightarrow \infty$, where $\mathcal{A}_1$ is the Airy$_1$ process, introduced by Sasamoto \cite{Sas05}. In addition, $\mathcal{A}_{2\rightarrow 1}(t + m)$ converges to $\mathcal{A}_2(t)$ as $m \rightarrow -\infty$, where $\mathcal{A}_2$ is the Airy$_2$ process, introduced by Pr{\" a}hofer and Spohn \cite{Spohn}. We refer the interested reader to \cite{QR13, QR14} for more background on the Airy$_{2 \rightarrow 1}$ and other Airy processes.  
\end{remark}

The way that the formula (\ref{QW3}) was derived in \cite{BFS08} was by constructing a certain process $X_{t}$, related to the TASEP with half-flat initial condition, and showing that 
\begin{equation}\label{QW4}
\lim_{t \rightarrow \infty} \mathbb{P} \left( \cap_{k = 1}^m \{ X_t(t_k) \leq y_k \} \right) = \det \left( I - \chi_y K_{\infty} \chi_{y} \right)_{L^2(\{t_1, \dots, t_m\}\times \mathbb{R})}.
\end{equation}
The fact that the right side of (\ref{QW4}) is the pointwise limit of distribution functions on $\mathbb{R}^m$ does not a priori imply that it is itself a distribution function. As we could not find a place in the literature where it is shown that the right side of (\ref{QW4}) is a distribution function, we show it in the following lemma, whose proof is given in Section \ref{Section6}. In the same lemma we also establish some properties of the finite-dimensional distribution functions of $ \mathcal{A}_{2\rightarrow 1}$, which will be required in our arguments.

\begin{lemma}\label{S1LWD} The following all hold.
\begin{enumerate}
\item There is a probability space $(\Omega, \mathcal{F}, \mathbb{P})$ and a real-valued process $\{ \mathcal{A}_{2 \rightarrow 1}(t): t \in \mathbb{R}\}$ on that space, whose finite-dimensional distribution is given by (\ref{QW3}).
\item For each $m$-tuple $t_1 < \cdots < t_m$ the right side of (\ref{QW3}) is continuous in $(y_1, \dots, y_m) \in \mathbb{R}^m$.
\item For each $y_1, t_1 \in \mathbb{R}$ we have the following identity
\begin{equation}\label{CrossFD2}
\begin{split}
&\det \left( I - \chi_y K_{\infty} \chi_{y} \right)_{L^2(\{t_1\}\times \mathbb{R})} = 1 + \sum_{k = 1}^{\infty} I_k, \mbox{ where } \\
& I_k = \frac{1}{(2\pi \i)^{2k} k!}  \int_{C_{1, \pi/4}^k}  \hspace{-2mm}d\vec{w} \int_{C_{0, 3\pi/4}^k}\hspace{-2mm} d\vec{z} \det \left[ \frac{2w_a e^{w_a^3/3 + t_1 w_a^2 +{\bf 1}\{t_1 \leq 0\}  t_1^2 w_a -y_1 w_a  }}{( z_a^2 - w_b^2)(w_a - z_a) e^{z_a^3/3 + t_1 z_a^2 +{\bf 1}\{t_1 \leq 0\}  t_1^2 z_a  - y_1 z_a}} \right]_{a,b = 1}^k     ,
\end{split}
\end{equation}
$C_{1,\pi/4}, C_{0,3\pi/4}$ are as in Definition \ref{S1ContV}. Part of the statement is that the integrals in (\ref{CrossFD2}) are all finite, and the series is absolutely convergent.
\end{enumerate}
\end{lemma}

%-------------------------------------------------------------------------------------------------------------------------------------------------------------------------------------------------
% Section 1.4
%
%-------------------------------------------------------------------------------------------------------------------------------------------------------------------------------------------------
\subsection{Main result}\label{Section1.4} Here we present the main asymptotic result of the paper, which describes the large time limit of the height function $h$ as in (\ref{DefH}), under the measure $\Phf$ from Definition \ref{DefASEP}.
\begin{theorem}\label{mainThm} Assume the same notation as in Definition \ref{DefASEP}. There exists $\tilde{\tau} \in (0,1)$ sufficiently small, so that for all $\tau \in (0, \tilde{\tau}]$ the following holds. For each $\alpha, y \in \mathbb{R}$ we have 
\begin{equation}\label{mainLE1}
\begin{split}
&\lim_{t \rightarrow \infty} \Phf \left( t^{-1/3} \cdot \left(t/2 + t^{1/3} \cdot (\alpha^2/2) \cdot {\bf 1}\{ \alpha \leq 0\}  - h( t/\gamma, \lfloor  t^{2/3} \alpha \rfloor)  \right) \leq y \right) \\
&= \mathbb{P}\left(  2^{-1/3} \mathcal{A}_{2 \rightarrow 1} (2^{-1/3} \alpha) \leq y\right),
\end{split}
\end{equation}
where $\mathcal{A}_{2 \rightarrow 1} $ is as in Definition \ref{Airy21}.
\end{theorem}
\begin{remark}\label{Issue3} We mention that Theorem \ref{mainThm} agrees with the analogous result for TASEP (corresponding to $q= 1$ and $p = 0$) from \cite{BFS08} as we explain here. If we let $Y_n(t)$ denote the location of the $n$-th leftmost particle of TASEP at time $t$, we have from \cite[Theorem 2.2]{BFS08} that for each $\alpha, {y} \in \mathbb{R}$, and $n(\alpha, t) = \lfloor t/4 + (1/2) \alpha t^{2/3} \rfloor$
\begin{equation}\label{YU1}
\lim_{t \rightarrow \infty} \Phf \left( \frac{Y_{n(\alpha, t)}(t)  -\alpha t^{2/3} + (\alpha^2/2){\bf 1}\{ \alpha \leq 0\} t^{1/3}    }{t^{1/3} } \leq {y} \right) = \mathbb{P}\left( 2^{-1/3} \mathcal{A}_{2 \rightarrow 1} (2^{-1/3} \alpha) \leq {y}\right).
\end{equation}
We mention that we have set $\tau = 2^{-1/3} \alpha$ in \cite[Theorem 2.2]{BFS08} and also $Y_n(t) = - X_n(t)$ as in that paper, since we have started from the initial condition $\eta_0 = {\bf 1}\{ x \in 2\mathbb{N}\}$ and $q =1 , p =0$, while in \cite{BFS08} the convention $\eta_0 = {\bf 1}\{ -x \in 2\mathbb{N}\}$ and $q =0 , p =1$ is used. Let us set 
$$\alpha_t = \alpha + t^{-1/3} \cdot \left(y- (1/2) \alpha^2 {\bf 1}\{ \alpha \leq 0\} \right),$$
and observe that we have the equality of events
$$\left\{ \frac{Y_{n(\alpha, t)}(t)  -\alpha t^{2/3} + (\alpha^2/2){\bf 1}\{ \alpha \leq 0\}  t^{1/3}    }{t^{1/3} } \leq {y} \right\} = \{ Y_{n(\alpha, t)}(t) \leq t^{2/3} \alpha_t \} = \{ N_{\lfloor t^{2/3} \alpha_t \rfloor}(t) \geq  n(\alpha, t)\},$$
where we recall that $N_x(t)$ is as in (\ref{DefN}). In particular, from (\ref{DefH}) and (\ref{YU1}) we get 
\begin{equation}\label{YU2}
\begin{split}
&\lim_{t \rightarrow \infty} \hspace{-1mm} \Phf \left( \hspace{-1mm}\frac{h(t,\lfloor t^{2/3} \alpha_t \rfloor) +\lfloor t^{2/3} \alpha_t \rfloor }{2}  \geq t/4 +  t^{2/3} \cdot \frac{\alpha_t - t^{-1/3} y + t^{-1/3} \cdot (\alpha^2/2) \cdot {\bf 1}\{ \alpha \leq 0\}  }{2} \right) \\
&=\lim_{t \rightarrow \infty}  \Phf \left( h(t,\lfloor t^{2/3} \alpha_t \rfloor)  \rfloor  \geq t/2  - t^{1/3} y + t^{1/3} \cdot (\alpha^2/2) \cdot {\bf 1}\{ \alpha \leq 0\} \right) \\
&= \mathbb{P}\left( 2^{-1/3} \mathcal{A}_{2 \rightarrow 1} (2^{-1/3} \alpha) \leq {y}\right).
\end{split}
\end{equation}
Equation (\ref{YU2}) agrees with (\ref{mainLE1}) as $\gamma = q - p = 1$. We mention that (\ref{YU2}) does not follow from Theorem \ref{mainThm}, as $p = 0$ is not allowed, and the point of this computation is to provide a sanity check for the scaling we perform in Theorem \ref{mainThm}.
\end{remark}

%-------------------------------------------------------------------------------------------------------------------------------------------------------------------------------------------------
% Section 1.5
%
%-------------------------------------------------------------------------------------------------------------------------------------------------------------------------------------------------
\subsection*{Outline}\label{Section1.5} The rest of the paper is organized as follows. In Section \ref{Section2} we prove Theorem \ref{S1Thm1}, which is the starting point of our asymptotic analysis. In Section \ref{Section3} we present the proof of Theorem \ref{mainThm}, which relies on taking the limit of the identity (\ref{S1QLT}), when $\zeta$ is scaled appropriately with $t$ and $t \rightarrow \infty$. In order to obtain the limit of (\ref{S1QLT}), we need to show that each summand converges, which is the statement of Proposition \ref{S3MainP1}, and that we can exchange the order of the series and the limit, which is ensured by Proposition \ref{S3MainP2}. Proposition \ref{S3MainP1} is proved in Section \ref{Section4} and relies on the method of steepest descent, as well as various estimates of the functions in (\ref{S1BasicFun}). Proposition \ref{S3MainP2} is proved in Section \ref{Section5} and is based on estimating the functions in (\ref{S1SpecFun}) along two classes of contours, depending on whether $k$ is large or small relative to $t$. In Section \ref{Section6} we prove Lemma \ref{S1LWD}.

%-------------------------------------------------------------------------------------------------------------------------------------------------------------------------------------------------
% Section 1.6
%
%-------------------------------------------------------------------------------------------------------------------------------------------------------------------------------------------------
\subsection*{Acknowledgments}\label{Section1.6} We are grateful to Alexei Borodin, Jeremy Quastel and Daniel Remenik for their useful comments on earlier drafts of the paper. ED is partially supported by NSF grant DMS:2054703.

%----------------------------------------------------------------------------------------------------------------------------------------------------------------------------------------------------------------------------
%
%     Section 2
%
%----------------------------------------------------------------------------------------------------------------------------------------------------------------------------------------------------------------------------
\section{Prelimit formula}\label{Section2} In this section we present the proof of Theorem \ref{S1Thm1}, which is the starting point of our asymptotic analysis in the next sections. The proof of Theorem \ref{S1Thm1} is given in Section \ref{Section2.2} and in Section \ref{Section2.1} below we establish several statements, which will be required. We continue with the same notation as in Section \ref{Section1}.

%----------------------------------------------------------------------------------------------------------------------------------------------------------------------------------------------------------------------------
%
%     Section 2.1
%
%----------------------------------------------------------------------------------------------------------------------------------------------------------------------------------------------------------------------------
\subsection{Definitions and notation}\label{Section2.1} In this section we present two key lemmas, which will be required in the proof of Theorem \ref{S1Thm1} in the next section, and whose proofs are postponed until Section \ref{Section2.3}. After stating the lemmas we summarize various estimates, which will also be of future use.

The first key result we require is as follows.
\begin{lemma}\label{S2AnalyticFD} Fix $k \in \mathbb{N}$ and $\zeta \in \mathbb{C} \setminus [0, \infty)$. Then, the integral $H_{k}(\zeta) $ in (\ref{S1KthSum}) is well-defined and finite. Moreover,  as a function of $\zeta$ we have that $H_{k}(\zeta)$ is analytic in $\mathbb{C} \setminus [0, \infty)$. If $\tau$ satisfies (\ref{S1TauSmall}), then the series on the right side of (\ref{S1QLT}) is absolutely convergent for each $\zeta \in \mathbb{C} \setminus [0, \infty)$ and defines an analytic function in $\mathbb{C} \setminus [0, \infty)$.
\end{lemma}

The second key lemma we need is the following. It is a special case of a Mellin-Barnes type integral representation from \cite[Lemma 3.20]{BorCor}. 
\begin{lemma}\label{S2MellinBarnes} Let $N,k \in \mathbb{N}$ and $\zeta \in \mathbb{C} \setminus [0,\infty)$ be such that $|\zeta| < 1$. Then, the series 
\begin{equation}\label{S1Sum}
\begin{split}
&\sum_{n_1 = 1}^{\infty} \cdots \sum_{ n_k = 1}^{\infty} \frac{\zeta^{n_1 + \cdots + n_k} }{k! (2\pi \i)^k} \oint_{\gamma_{-1,0}^k}d\vec{w}  F(\vec{n}, \vec{w})
\end{split}
\end{equation}
is absolutely convergent and equals $H_k(\zeta)$ from (\ref{S1KthSum}). Here, $F(\vec{s}, \vec{w})$ is as in (\ref{S1SpecFun}) and $\gamma_{-1,0}$ is as in Proposition \ref{S1Prop1}.
\end{lemma}

In the remainder of this section we summarize several basic estimates for the various functions that appear in (\ref{S1SpecFun}). We observe that given $c > 0$ we can find $c' > 0$ such that if $x,y \in \mathbb{R}$ and $d(x + \i y, \mathbb{Z}) \geq c$, then 
\begin{equation}\label{S2BoundSine}
\frac{1}{| \sin (\pi x + \i \pi y)|} \leq c' e^{-\pi |y|}.
\end{equation}

Recall the Cauchy determinant formula, see e.g. \cite[1.3]{Prasolov},
\begin{equation}\label{S2CauchyDet}
\det \left[ \frac{1}{x_i - y_j}\right]_{i,j = 1}^N  = \frac{\prod_{1 \leq i < j \leq N} (x_i - x_j) (y_j - y_i)}{\prod_{i,j = 1}^N (x_i - y_j)}.
\end{equation}
The following statement summarizes two different bounds on the above Cauchy determinant.
\begin{lemma}\label{DetBounds}\cite[Lemma 3.12]{ED2020} Let $N \in \mathbb{N}$.
\begin{enumerate}
\item Hadamard's inequality: If $A$ is an $N \times N$ matrix and $v_1, \dots ,v_N$ denote the column vectors of $A$, then $|\det A| \leq \prod_{i = 1}^N \|v_i\|$ where $\|x\| = (x_1^2 + \cdots + x_N^2)^{1/2}$ for $x = (x_1, \dots, x_N)$. 
\item Fix $r, R \in (0,\infty)$ with $R > r$. Let $z_i, w_i \in \mathbb{C}$ be such that $|w_i| = R$ and $|z_i|  \leq r$ for $i = 1, \dots, N$. Then, 
\begin{equation}\label{CDetGood}
\left|\det \left[ \frac{1}{z_i - w_j}\right]_{i,j = 1}^N \right| \leq R^{-N} \cdot \frac{N^N \cdot (r/R)^{\binom{N}{2}}}{(1-r/R)^{N^2}}.
\end{equation}
\end{enumerate}
\end{lemma}

Let $\mathfrak{f}, \mathfrak{g}, \mathfrak{h}$ be as in (\ref{S1BasicFun}). Then, we can find $A_1 > 0$, depending on $\tau,t,p,q,x$, such that for all $w \in \gamma_{-1,0}$ (recall this was the positively oriented circle of radius $\tau^{-1/8}$, centered at the origin) and $s \in \mathbb{C}$ such that $\Re(s) \geq 1/2$ we have
\begin{equation}\label{S1S1E1}
\left| \mathfrak{f}(w; s) \mathfrak{g}(w; s)  \right| \leq A_1.
\end{equation}
Also if $w_1, w_2 \in \gamma_{-1,0}$ and $s_1, s_2 \in \mathbb{C}$ with $\Re(s_1),\Re(s_2)  \geq 1/2$ we have
\begin{equation}\label{S1S1E2}
\left| \mathfrak{h}(w_1,w_2; s_1, s_2) \right| \leq (1 + \tau^{-1/4}) \cdot \frac{(-\tau^{3/4};\tau)_{\infty} (-\tau^{3/4};\tau)_{\infty}}{(\tau^{1/4}  ;\tau)_{\infty} (\tau^{1/4};\tau)_{\infty}},
\end{equation}
where we used that for $|z| = r \in [0, 1)$ the function $|(z;\tau)_{\infty}|$ is maximized when $z = -r$ and minimized when $z = r$. 

Finally, from Lemma \ref{DetBounds}(ii) applied to $r = \tau^{3/8}$ and $R = \tau^{-1/8}$ we have for all $w_i \in \gamma_{-1,0}$ and $s_i \in \mathbb{C}$ with $\Re(s_i) \geq 1/2$ for $i = 1, \dots, k$ that
\begin{equation}\label{S1S1E3}
\left| \det \left[ \frac{-1}{w_a \tau^{s_a} - w_b}  \right]_{a,b = 1}^k  \right| \leq  \tau^{k/8} \frac{k^k \cdot \tau^{\frac{1}{2} \binom{k}{2} }}{(1 - \tau^{1/2})^{k^2} }.
\end{equation}
Combining (\ref{S1S1E1}), (\ref{S1S1E2}) and (\ref{S1S1E3}) we conclude that for each $k \in \mathbb{N}$, $w_i \in \gamma_{-1,0}$ and $s_i \in \mathbb{C}$ with $\Re(s_i) \geq 1/2$ for $i = 1, \dots, k$ we have
\begin{equation}\label{S1S1E4}
\begin{split}
&\frac{1}{k!} \left| \det \left[ \frac{-1}{w_a \tau^{s_a} - w_b}  \right]_{a,b = 1}^k   \mathfrak{f}(w_a;s_a) \mathfrak{g}(w_a; s_a) \cdot \prod_{1 \leq a < b \leq k} \mathfrak{h}(w_a, w_b; s_a, s_b) \right| \leq  A^k \cdot \rho^{\binom{k}{2}}, \mbox{ where }\\
& A =  \frac{eA_1 \tau^{1/8} }{(1 - \tau^{1/2})} \mbox{ and } \rho = \frac{(\tau^{1/2} + \tau^{1/4})}{(1 - \tau^{1/2})^{2}} \cdot   \frac{(-\tau^{3/4};\tau)_{\infty} (-\tau^{3/4};\tau)_{\infty}}{(\tau^{1/4}  ;\tau)_{\infty} (\tau^{1/4};\tau)_{\infty}}.
\end{split}
\end{equation}
In deriving the last inequality we used that $k^k \leq k! e^k$, which can be deduced from 
\begin{equation}\label{S2Rob}
n! = \sqrt{2\pi} n^{n+1/2} e^{-n} e^{r_n} \mbox{ for } n \in \mathbb{N}, \mbox{ where } \frac{1}{12n + 1} < r_n < \frac{1}{12n},
\end{equation}
see \cite[Equation (1)]{Rob}.

%----------------------------------------------------------------------------------------------------------------------------------------------------------------------------------------------------------------------------
%
%     Section 2.2
%
%----------------------------------------------------------------------------------------------------------------------------------------------------------------------------------------------------------------------------
\subsection{Proof of Theorem \ref{S1Thm1}}\label{Section2.2} We continue with the same notation as in the statement of the theorem. For clarity we split the proof into two steps. \\

{\bf \raggedleft Step 1.} Note that by Lemma \ref{S2AnalyticFD} each summand on the right side of (\ref{S1QLT}) is well-defined and finite, and moreover by our assumption on $\tau$ in (\ref{S1TauSmall}) we have that the series on the right side is absolutely convergent and defines an analytic function in $\zeta$ on $\mathbb{C} \setminus [0, \infty)$. 

We claim that (\ref{S1QLT}) holds provided that $\zeta \in \mathbb{C} \setminus [0, \infty)$ is such that $|\zeta| < 1$. We will prove this statement in the next step. Here we assume its validity and proceed to prove that (\ref{S1QLT}) holds for all $\zeta \in \mathbb{C} \setminus [0, \infty)$.\\

From (\ref{S1QExp}) we have
\begin{equation}\label{S2S2E1}
\Ehf \left[ e_{\tau} \left( \zeta \tau^{N_x(t)} \right) \right]  = \sum_{n = 0}^{\infty} \frac{1}{((1-\tau) \zeta \tau^n; \tau)_{\infty}} \cdot \Phf (N_x(t) = n ).
\end{equation}
If $K \subset \mathbb{C} \setminus [0, \infty)$ is compact we observe that so is the set $\hat{K} = \{0\} \cup \cup_{n = 0}^{\infty} \tau^n \cdot K$, and the latter is separated from the zeros of the function $((1-\tau) z ; \tau)_{\infty}$. In particular, we conclude that there exists a constant $C$, depending on $K$ and $\tau$, such that for all $\zeta \in K$ and $n \in \mathbb{Z}_{\geq 0}$ we have
$$\left| \frac{1}{((1-\tau) \zeta \tau^n; \tau)_{\infty}} \right| \leq C.$$
The latter implies that the right side of (\ref{S2S2E1}) is the uniform over $\zeta \in K$ limit of 
$$\sum_{n = 0}^{N} \frac{1}{((1-\tau) \zeta \tau^n; \tau)_{\infty}} \cdot \Phf (N_x(t) = n )$$
as $N \rightarrow \infty$. Since each of the above functions are analytic in $\zeta$ on $\mathbb{C} \setminus [0, \infty)$, we conclude the same is true for the expressions in (\ref{S2S2E1}), cf. \cite[Chapter 2, Theorem 5.2]{Stein}.

Our work in the above paragraph shows that the left side of (\ref{S1QLT}) is analytic in $\zeta$ on $\mathbb{C} \setminus [0, \infty)$ and from Lemma \ref{S2AnalyticFD} we know the same is true for the right side. As these two functions agree when $\zeta\in \mathbb{C} \setminus [0, \infty)$ is such that $|\zeta| < 1$ by assumption, we conclude that they agree for all $\zeta \in \mathbb{C} \setminus [0, \infty)$, cf. \cite[Chapter 2, Theorem 4.8]{Stein}. This concludes the proof of the theorem.\\

{\bf \raggedleft Step 2.} In this step we fix $\zeta \in \mathbb{C} \setminus [0, \infty)$ such that $|\zeta| < 1$ and proceed to prove (\ref{S1QLT}). From (\ref{S1QExp})
\begin{equation}\label{S2S2E2}
\Ehf \left[ e_{\tau} \left( \zeta \tau^{N_x(t)} \right) \right]  = \Ehf \left[ \sum_{m = 0}^{\infty} \frac{\zeta^m \tau^{m N_x(t)}}{m_{\tau}!}  \right] = \lim_{N \rightarrow \infty} \sum_{m = 0}^{N}  \Ehf \left[ \frac{\zeta^m  \tau^{m N_x(t)} }{m_{\tau}!} \right],
\end{equation}
where we used Fubini's theorem to exchange the order of the sum and the expectation. Note that the application of Fubini's theorem is justified since 
$$\Ehf \left[ \sum_{m = 0}^{\infty} \left| \frac{\zeta^m  \tau^{m N_x(t)}}{m_{\tau}!} \right| \right] \leq \Ehf \left[ \sum_{m = 0}^{\infty} \frac{|\zeta|^m}{m_{\tau}!}\right] =  e_{\tau} \left( |\zeta| \right)< \infty,$$
where we used that $N_x(t) \geq 0$ almost surely, $\tau \in (0,1)$ and $|\zeta| < 1$ by assumption.\\

We further have from Proposition \ref{S1Prop1} for each $N \in \mathbb{N}$ that 
\begin{equation}\label{S2S2E3}
\begin{split}
&\sum_{m = 0}^{N}  \Ehf \left[ \frac{\zeta^m  \tau^{m N_x(t)} }{m_{\tau}!} \right] =1 + \sum_{k = 1}^{\infty} \sum_{n_1 = 1}^{\infty} \cdots \sum_{ n_k = 1}^{\infty} H_k^N(\zeta; n_1, \dots, n_k)  \mbox{, where }\\
&H_k^N(\zeta; n_1, \dots, n_k)  = {\bf 1}\{ n_1 + \cdots + n_k \leq N \} \cdot  \frac{\zeta^{n_1+\cdots + n_k}}{k! (2\pi \i)^k} \oint_{\gamma_{-1,0}^k}d\vec{w} F(\vec{n}, \vec{s}).
\end{split}
\end{equation}

From (\ref{S1S1E4}) and the fact that $\gamma_{-1,0}$ has length $2\pi \tau^{-1/8}$ we have
\begin{equation}\label{S2S2E4}
\left| H_k^N(\zeta; n_1, \dots, n_k) \right| \leq A^k \cdot \rho^{\binom{k}{2}}  \cdot \tau^{-k/8}  |\zeta|^{n_1 + \cdots + n_k}.
\end{equation}
In addition, from (\ref{S2S2E4}) and the dominated convergence theorem we have
\begin{equation*}
\begin{split}
&\lim_{N \rightarrow \infty}\sum_{n_1 = 1}^{\infty} \cdots \sum_{ n_k = 1}^{\infty} H_k^N(\zeta; n_1, \dots, n_k)  =  \sum_{n_1 = 1}^{\infty} \cdots \sum_{ n_k = 1}^{\infty}  \frac{\zeta^{n_1 + \cdots + n_k}}{k! (2\pi \i)^k} \oint_{\gamma_{-1,0}^k}d\vec{w} F(\vec{n}, \vec{w}).
\end{split}
\end{equation*}
The last equation and Lemma \ref{S2MellinBarnes} together imply
\begin{equation}\label{S2S2E5}
\lim_{N \rightarrow \infty}\sum_{n_1 = 1}^{\infty} \cdots \sum_{ n_k = 1}^{\infty} H_k^N(\zeta; n_1, \dots, n_k) = H_k(\zeta).
\end{equation}

Combining (\ref{S2S2E2}) and (\ref{S2S2E3}) we conclude 
\begin{equation}\label{S2S2E6}
\begin{split}
&\Ehf \left[ e_{\tau} \left( \zeta \tau^{N_x(t)} \right) \right]  = \lim_{N \rightarrow \infty} 1 + \sum_{k = 1}^{\infty} \frac{1}{k!} \sum_{n_1 = 1}^{\infty} \cdots \sum_{ n_k = 1}^{\infty} H_k^N(\zeta; n_1, \dots, n_k)  = \\
& 1 + \sum_{k = 1}^{\infty}\lim_{N \rightarrow \infty}  \frac{1}{k!} \sum_{n_1 = 1}^{\infty} \cdots \sum_{ n_k = 1}^{\infty} H_k^N(\zeta; n_1, \dots, n_k) = 1 + \sum_{k = 1}^{\infty}H_k(\zeta),
\end{split}
\end{equation}
where in the last equality we used (\ref{S2S2E5}). We mention that in exchanging the order of the sum and the limit above we used the dominated convergence theorem as from (\ref{S2S2E4}) we have
$$ \sum_{n_1 = 1}^{\infty} \cdots \sum_{ n_k = 1}^{\infty} \left| H_k^N(\zeta; n_1, \dots, n_k)  \right| \leq \left( \frac{A \tau^{-1/8}}{1 - |\zeta|} \right)^k \rho^{\binom{k}{2}},$$
 and the latter is summable over $k \in \mathbb{N}$. From (\ref{S2S2E6}) we get (\ref{S1QLT}) when $|\zeta| < 1$, as desired.

\begin{remark}\label{CompOQR} Now that we have presented our proof of Theorem \ref{S1Thm1}, let us explain how our approach compares with that in \cite{OQR}. The overall strategy of both proofs is quite similar, in that (\ref{S1QLT}) is first established when $\zeta \in \mathbb{C} \setminus [0, \infty)$ is such that $|\zeta| < 1$, and then extended to $\zeta \in \mathbb{C} \setminus [0, \infty)$ by analytic continuation. In the case of $|\zeta| < 1$ the main difficulty is in taking the $N \rightarrow \infty$ limit of (\ref{S2S2E3}) and justifying exchanging the order of the series on the right side of that equation and the limit. To analytically extend (\ref{S1QLT}) to $\zeta \in \mathbb{C} \setminus [0, \infty)$, the main difficulty is in showing that the series on the right side of (\ref{S1QLT}) is absolutely convergent and hence analytic in $\zeta$ (as the absolutely convergent series of analytic in $\zeta$ functions).

In order to overcome the two challenges above, one needs to find good bounds on the summands in the two series in (\ref{S1QLT}) and (\ref{S2S2E3}). One runs into trouble, since over $\gamma_{-1,0}^k$ the best pointwise estimate for $|\prod_{1 \leq a < b \leq k} \mathfrak{h}(w_a, w_b; s_a,s_b) |$ (which appears in both $ H_k^N(\zeta; n_1, \dots, n_k)$ and $H_k(\zeta)$) is $e^{\tilde{c} k^2}$ for some $\tilde{c} > 0$. This bound behaves very poorly in $k$, and cannot be compensated by the $k!$ in the denominator of $H_k(\zeta)$ in (\ref{S1KthSum}). Faced with this difficulty, the authors \cite{OQR} suggested to deform $\gamma_{-1,0}^k$ to certain $k$-dependent contours $\overline{\gamma}_k^k$, and for general $\zeta \in \mathbb{C} \setminus [0, \infty)$ also deform $(1/2 + \i \mathbb{R})^k$ to certain $k$-dependent contours $\overline{D}_k^k$. Over the contours $\overline{\gamma}^k_k$ and $\overline{D}_k^k$ one can obtain a favorable estimate for $|\prod_{1 \leq a < b \leq k} \mathfrak{h}(w_a, w_b; s_a,s_b) |$. Unfortunately, what was not realized in \cite{OQR} is that over these deformed contours other parts of the integrands in the definitions of $ H_k^N(\zeta; n_1, \dots, n_k)$ and $H_k(\zeta)$ behave badly and become difficult to control.

The way we overcome the growth of $|\prod_{1 \leq a < b \leq k} \mathfrak{h}(w_a, w_b; n_a,n_b) |$ as a function of $k$, is to recognize that one has some decay built into the Cauchy determinant in the definitions of $ H_k^N(\zeta; n_1, \dots, n_k)$ and $H_k(\zeta)$. This decay is sufficient to control the series in (\ref{S1QLT}) and (\ref{S2S2E3}), only when $\tau$ is sufficiently small -- see (\ref{S1TauSmall}). In fact, it still remains a difficulty for us to prove that the right side of (\ref{S1QLT}) is convergent and hence well-defined for general $\tau \in (0,1)$. 
\end{remark}

%----------------------------------------------------------------------------------------------------------------------------------------------------------------------------------------------------------------------------
%
%     Section 2.3
%
%----------------------------------------------------------------------------------------------------------------------------------------------------------------------------------------------------------------------------
\subsection{Proof of Lemmas \ref{S2AnalyticFD} and \ref{S2MellinBarnes}}\label{Section2.3} In this section we give the proofs of the two key lemmas from Section \ref{Section2.1}.

\begin{proof}[Proof of Lemma \ref{S2AnalyticFD}] Let us fix a compact set $K \subset \mathbb{C} \setminus [0, \infty)$ and $\zeta \in K$. It follows from (\ref{S2BoundSine}) that there are constants $C_1,c_1 > 0$, depending on $K$ alone, such that for all $s = 1/2 + \i y  \in 1/2 + \i \mathbb{R}$ 
\begin{equation}\label{S2S3E1}
\left|\frac{\pi (-\zeta)^{s}}{\sin(-\pi s)} \right| \leq C_1 e^{-c_1 |y|}.
\end{equation}
Combining (\ref{S2S3E1}) with (\ref{S1S1E4}) we conclude that the integral in (\ref{S1KthSum}) is well-defined and finite. Moreover, we see that $H_k(\zeta)$ is the uniform over $\zeta \in K$ limit of
\begin{equation*}
\begin{split}
&G^N_k(\zeta):= \frac{1}{k! (2\pi \i)^{2k}} \cdot \int_{(1/2 + \i [-N, N])^k} d\vec{s} \oint_{\gamma_{-1,0}^k} d\vec{w} F(\zeta; \vec{s}, \vec{w}) \mbox{ as $N \rightarrow \infty$.}
\end{split}
\end{equation*}

The integrand in $G_k^N(\zeta)$ is analytic in $\zeta$ on $\mathbb{C} \setminus [0,\infty)$ and is jointly continuous in $\zeta$ and $\vec{w}, \vec{s}$, while the contours we are integrating over are compact. The latter implies that $G_k^N(\zeta)$ is analytic in $\zeta$ for each $N$, see e.g. \cite[Theorem 5.4]{Stein}. We thus conclude that $H_k(\zeta)$ is analytic as the uniform over compact sets limit of analytic functions, cf. \cite[Chapter 2, Theorem 5.2]{Stein}. \\

Using that for $\lambda > 0$ we have $\int_{\mathbb{R}} e^{-\lambda |x|} dx = 2\lambda^{-1}$, that the length of $\gamma_{-1, 0}$ is $2\pi \tau^{-1/8}$ as well as (\ref{S1S1E4}) and (\ref{S2S3E1}) we conclude that 
$$\left| H_k(\zeta) \right|  \leq \left( \frac{C_1 A \tau^{-1/8}}{ \pi c_1} \right)^k \cdot \rho^{\binom{k}{2}}.$$
The right side above is summable over $k$, since $\rho \in (0,1)$ in view of (\ref{S1TauSmall}). The latter implies that the series in (\ref{S1QLT}) is absolutely convergent. Moreover, since each summand is analytic in $\zeta$ on $\mathbb{C} \setminus [0,\infty)$, we conclude the same is true for the series by \cite[Chapter 2, Theorem 5.2]{Stein}. 
\end{proof}

\begin{proof}[Proof of Lemma \ref{S2MellinBarnes}] From (\ref{S1S1E4}) and the fact that the length of $\gamma_{-1, 0}$ is $2\pi \tau^{-1/8}$ we conclude that the $(n_1, \dots, n_k)$-th summand in (\ref{S1Sum}) is bounded in absolute value by 
$$|\zeta|^{n_1 + \cdots + n_k} \cdot A^k \cdot \rho^{\binom{k}{2}} \cdot \tau^{-k/8} \cdot (2\pi)^{-k},$$
and the latter is summable over $(n_1, \dots, n_k) \in \mathbb{N}^k$ (since $|\zeta| < 1$ by assumption), proving the absolute convergence of the series in (\ref{S1Sum}).

To conclude the proof of the lemma, it suffices to show
\begin{equation}\label{S2S3E2}
\begin{split}
&\lim_{N \rightarrow \infty} \sum_{n_1 = 1}^{N} \cdots \sum_{ n_k = 1}^{N} \frac{\zeta^{n_1 + \cdots + n_k}}{k! (2\pi \i)^k} \oint_{\gamma_{-1,0}^k}d\vec{w} F(\vec{n}, \vec{w}) = H_k(\zeta).
\end{split}
\end{equation}

Let $R_N = 1/2 + N$ and set $A^1_N = 1/2 - \i R_N$, $A^2_N = 1/2 + \i R_N$, $A^3_N = R_N + \i R_N$, $A^4_N = R_N - \i R_N$. Denote by $\gamma_N^1$ the contour, which goes from $A_N^1$ vertically up to $A^2_N$, by $\gamma_2^N$ the contour, which goes from $A^2_N$ horizontally to $A_N^3$, by $\gamma_N^3$ the contour, which goes from $A_N^3$ vertically down to $A_N^4$ and by $\gamma_N^4$ the contour, which goes from $A^4_N$ horizontally to $A_N^1$. Also let $\gamma_N = \cup_{i = 1}^4 \gamma^i_N$ traversed in order, see Figure \ref{S2_1}.
\begin{figure}[h]
\centering
\scalebox{0.6}{\includegraphics{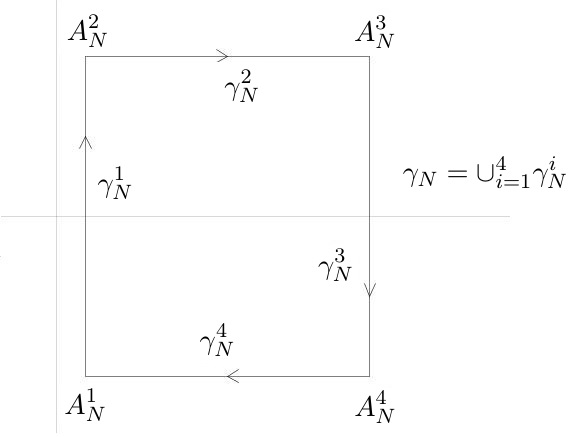}}
\caption{The contours $\gamma_N^i$ for $i = 1,\dots,4$.}
\label{S2_1}
\end{figure}

We observe by the Residue Theorem that for each $N \in \mathbb{N}$
\begin{equation}\label{S2S3E3}
\begin{split}
& \sum_{n_1 = 1}^{N} \cdots \sum_{ n_k = 1}^{N} \frac{\zeta^{n_1 + \cdots + n_k}}{k! (2\pi \i)^k} \oint_{\gamma_{-1,0}^k}d\vec{w} F(\vec{n}, \vec{w}) =  \frac{1}{k! (2\pi \i)^{2k} } \oint_{\gamma_{-1,0}^k} d\vec{w}  \int_{(\gamma_N)^k} d\vec{s} F(\zeta; \vec{s}, \vec{w}),
\end{split}
\end{equation}
where we recall that $F(\zeta; \vec{s}, \vec{w})$ was defined in (\ref{S1SpecFun}). In deriving the last expression we used that in each variable $s_a$ for $a = 1, \dots, k$ the function $F(\zeta; \vec{s}, \vec{w})$ is analytic in the region enclosed by $\gamma_N$ except at the points $s_a \in \{1, \dots, N \}$ where the function has a simple pole coming from $\frac{\pi }{\sin(-\pi s_a)}$, the fact that 
$$\mathsf{Res}_{z = m} \frac{\pi }{\sin(-\pi s_a)} = (-1)^{m+1},$$
and also that $\gamma_N$ is {\em negatively} oriented.

We next note by the dominated convergence theorem that 
\begin{equation}\label{S2S3E4}
\begin{split}
&\lim_{N \rightarrow \infty} \frac{1}{k! (2\pi \i)^{2k} } \oint_{\gamma_{-1,0}^k} d\vec{w}  \int_{(\gamma^1_N)^k} d\vec{s} F(\zeta; \vec{s}, \vec{w})  = H_{k}(\zeta).
\end{split}
\end{equation}
In deriving the last statement we used (\ref{S1S1E4}) and (\ref{S2S3E1}), which justify the dominated convergence theorem with dominating function $(C_1 A)^k \rho^{\binom{k}{2}} \prod_{a = 1}^k e^{-c_1 |y_a|}$ (here $C_1, c_1$ are as in (\ref{S2S3E1}) and we have written $s_a = x_a + \i y_a$ for $ a = 1, \dots, k$).

In addition, we have from (\ref{S2BoundSine}) and (\ref{S1S1E4}), the fact that the length of $\gamma_{0,-1}$ is $2\pi \tau^{-1/8}$, the length of $\gamma_N^{r}$ is $N$ for $r = 2,4$ and $|\zeta| < 1$ that for $r=2,4$
\begin{equation}\label{S2S3E5}
\begin{split}
& \left| \frac{1}{k! (2\pi \i)^{2k} } \oint_{\gamma_{-1,0}^k} d\vec{w}  \int_{(\gamma^r_N)^k} d\vec{s} F(\zeta; \vec{s}, \vec{w})  \right| \leq \frac{A^k \rho^{\binom{k}{2}}  N^k \tau^{-k/8} \cdot   }{(2\pi)^k} \cdot c'^ke^{-k\pi N}.
\end{split}
\end{equation}
Using also that the length of $\gamma_N^3$ is $2N + 1$, and the same statements as above we get
\begin{equation}\label{S2S3E6}
\begin{split}
& \left| \frac{1}{k! (2\pi \i)^{2k} } \oint_{\gamma_{-1,0}^k} d\vec{w}  \int_{(\gamma^3_N)^k} d\vec{s} F(\zeta; \vec{s}, \vec{w})  \right| \leq  |\zeta|^{k(N+1/2)}  \frac{A^k \rho^{\binom{k}{2}}  (2N + 1)^k \tau^{-k/8}}{(2\pi)^k} \cdot c'^k.
\end{split}
\end{equation}
We mention that in (\ref{S2S3E5}) and (\ref{S2S3E6}) the constant $c'$ is as in (\ref{S2BoundSine})  for $c = 1/2$, and we used that $\gamma_N^i$ are at least distance $1/2$ from $\mathbb{Z}$ by construction. 

Since $|\zeta| < 1$ by assumption, we have that the right sides of (\ref{S2S3E5}) and (\ref{S2S3E6}) both converge to zero as $N \rightarrow \infty$. Combining (\ref{S2S3E3}), (\ref{S2S3E4}), (\ref{S2S3E5}) and (\ref{S2S3E6}) we conclude (\ref{S2S3E2}), which concludes the proof of the lemma.

\end{proof}

%-------------------------------------------------------------------------------------------------------------------------------------------------------------------------------------------------
%
% Section 3
%
%-------------------------------------------------------------------------------------------------------------------------------------------------------------------------------------------------
\section{Weak convergence}\label{Section3} The goal of this section is to prove Theorem \ref{mainThm}, for which we need to study equation (\ref{S1QLT}) in Theorem \ref{S1Thm1} as $t \rightarrow \infty$. In Seciton \ref{Section3.1} we explain how we need to scale the parameters in (\ref{S1QLT}), and formulate two key asymptotic statements about the summands $H_{k}(\zeta)$ -- see Propositions \ref{S3MainP1} and \ref{S3MainP2}. In Section \ref{Section3.2} we use these two propositions to complete the proof of Theorem \ref{mainThm}. In Section \ref{Section3.3} we present a useful way to rewrite $H_{k}(\zeta)$, which will help us establish Propositions \ref{S3MainP1} and \ref{S3MainP2}, whose proofs are given in Sections \ref{Section4} and \ref{Section5}, respectively. Throughout this section we continue with the same notation as in Sections \ref{Section1} and \ref{Section2}.

%-------------------------------------------------------------------------------------------------------------------------------------------------------------------------------------------------
%
% Section 3.1
%
%-------------------------------------------------------------------------------------------------------------------------------------------------------------------------------------------------
\subsection{Two key propositions}\label{Section3.1} In this section we formulate the key results we require in the proof of Theorem \ref{mainThm} in Section \ref{Section3.2}. We begin by stating our assumptions on parameters and their scaling.
\begin{definition} \label{DefScale} We assume the same notation as in Definition \ref{DefASEP}. We further fix $\tilde{r}, \alpha \in \mathbb{R}$, and for $t > 0$ define
\begin{equation}\label{S3ZetaScale}
\zeta = - (1-\tau)^{-1}\tau^{- (1/4)t - (1/2)( \lfloor t^{2/3}\alpha \rfloor - 1 )  + t^{1/3} \tilde{r} }.
\end{equation}
With this data we let $I(k,t) = H_{k}(\zeta)$, where $H_{k}(\zeta)$ is as in (\ref{S1KthSum}) with $p,q,\tau, \zeta$ as just specified, $x = \lfloor  t^{2/3} \alpha \rfloor$ and $t$ replaced with $t/\gamma$. We mention that $I(k,t) $ is well-defined and finite in view of Lemma \ref{S2AnalyticFD} and the fact that $\zeta$ from (\ref{S3ZetaScale}) lies in $(-\infty, 0)$.
\end{definition}

The first key proposition we require is as follows.
\begin{proposition}\label{S3MainP1} Assume the same notation as in Definition \ref{DefScale}. For $k \in \mathbb{N}$
\begin{equation}\label{S3KE1}
\lim_{t \rightarrow \infty} I(k,t) =  \frac{1}{k!(2\pi \i)^{2k}} \int_{C_{0,\pi/4}^k}\hspace{-3mm} d\vec{u}  \int_{C_{-1,3\pi/4}^k}  \hspace{-3mm} d\vec{v} \det \left[ \frac{2v_a e^{u_a^3/48 - u_a^2 \alpha/ 8 - u_a \tilde{r}} }{(u_a-v_a)(v_b^2 - u_a^2)e^{v_a^3/48 - v_a^2 \alpha / 8 - v_a \tilde{r}}}\right]_{a,b = 1}^k,
\end{equation}
where $C_{a,\phi}$ is as in Definition \ref{S1ContV}.
\end{proposition}

The second key proposition we require is as follows.
\begin{proposition}\label{S3MainP2} Assume the same notation as in Definition \ref{DefScale}. There exists $\tau_0 \in (0,1)$ sufficiently small, so that the following holds. If $\tau \in (0, \tau_0]$, we can find constants $A, T > 0$, depending on $\tau, \tilde{r},\alpha$, such that for $t \geq T$ and $k \in \mathbb{N}$
\begin{equation}\label{S3KE2}
\left| I(k,t) \right| \leq A^k \cdot k^{-k/2}.
\end{equation}
\end{proposition}

We end this section with the following elementary probability lemma from \cite{BorCor}, which will also be required in our arguments. We mention that analogues of the below lemma have been known for a while in the physics literature, see e.g. \cite[Equation (14)]{CDR10}. 
\begin{lemma}\label{prob}
\cite[Lemma 4.39]{BorCor}. Suppose that $f_n$ is a sequence of functions $f_n: \mathbb{R} \rightarrow [0,1]$, such that for each $n$, $f_n(y)$ is strictly decreasing in $y$ with a limit of $1$ at $y = -\infty$ and $0$ at $y = \infty$. Assume that for each $\delta > 0$ one has on $\mathbb{R}\backslash [-\delta,\delta]$, $f_n \rightarrow {\bf 1}_{\{y < 0\}}$ uniformly. Let $X_n$ be a sequence of random variables such that for each $x \in \mathbb{R}$ 
$$\mathbb{E}[f_n(X_n - x)] \rightarrow p(x),$$
and assume that $p(x)$ is a continuous probability distribution function. Then $X_n$ converges in distribution to a random variable $X$, such that $\mathbb{P}(X \leq x) = p(x)$.
\end{lemma}

%-------------------------------------------------------------------------------------------------------------------------------------------------------------------------------------------------
%
% Section 3.2
%
%-------------------------------------------------------------------------------------------------------------------------------------------------------------------------------------------------
\subsection{Proof of Theorem \ref{mainThm}}\label{Section3.2} In this section we present the proof of Theorem \ref{mainThm}. For clarity, we split the proof into two steps.\\

{\bf \raggedleft Step 1.} Let $\tau_0$ be as in Proposition \ref{S3MainP2}, and let $\tilde{\tau} \in (0,1)$ be sufficiently small so that $\tilde{\tau} \leq \tau_0$ and for $\tau \in (0, \tilde{\tau}]$ we have that (\ref{S1TauSmall}) holds. This specifies $\tilde{\tau}$ in the statement of the theorem. In the remainder we fix $\tau \in (0, \tilde{\tau}]$ and a sequence $t_n > 0$, such that $t_n \uparrow \infty$ as $n \rightarrow \infty$.

For $y \in \mathbb{R}$ and $n \in \mathbb{N}$ we define 
$$f_n(y) = \frac{1}{( - \tau^{-t_n^{1/3} y}; \tau)_{\infty}}.$$
We further introduce the random variables
\begin{equation}\label{LS1}
X_n = t_n^{-1/3} \cdot \left(t_n/4 + (1/2)( \lfloor t_n^{2/3}\alpha \rfloor - 1 ) - N_{\lfloor  t_n^{2/3} \alpha \rfloor}(t_n/\gamma)  \right),
\end{equation}
where $N_x(t)$ is as in (\ref{DefN}). We claim that for each $\tilde{r} \in \mathbb{R}$ we have
\begin{equation}\label{LS2}
\lim_{n \rightarrow \infty} \Ehf \left[ f_n (X_n - \tilde{r}) \right] = \mathbb{P}\left(  \mathcal{A}_{2 \rightarrow 1} (2^{-1/3} \alpha) \leq 2^{4/3} \cdot \tilde{r} + {\bf 1} \{\alpha \leq 0\} \cdot 2^{-2/3} \alpha^2 \right),
\end{equation}
where $\mathcal{A}_{2 \rightarrow 1}$ is as in Definition \ref{Airy21}. We prove (\ref{LS2}) in Step 2. Here, we assume its validity and conclude the proof of the theorem.\\

From \cite[Lemma 5.1]{FerVet} we have that 
\begin{equation}\label{seqfeq}
h_q(y) := \frac{1}{(- \tau^{-qy}; \tau)_\infty} = \prod_{k = 1}^\infty \frac{1}{1 + \tau^{-qy + k}}
\end{equation}
is strictly decreasing for all $q > 0$. Moreover, for each $\delta > 0$ one has $h_q(y) \rightarrow {1}_{\{y < 0\}}$ uniformly on $\mathbb{R} \backslash [-\delta, \delta]$ as $q \rightarrow \infty$. The latter implies that $f_n$ satisfy the conditions of Lemma \ref{prob}. We further note that the right side of (\ref{LS2}) is continuous in $\tilde{r}$ from the second part of Lemma \ref{S1LWD}. In particular, we see that the conditions of Lemma \ref{prob} are all satisfied, and so 
\begin{equation}\label{LS3}
\begin{split}
&\lim_{n \rightarrow \infty} \Phf \left( t_n^{-1/3} \cdot \left(t_n/4 + (1/2)( \lfloor t_n^{2/3}\alpha \rfloor - 1 ) - N_{\lfloor  t_n^{2/3} \alpha \rfloor}(t_n/\gamma)  \right) \leq \tilde{r}  \right) \\
&= \mathbb{P}\left(  \mathcal{A}_{2 \rightarrow 1} (2^{-1/3} \alpha) \leq 2^{4/3} \cdot \tilde{r} + {\bf 1} \{\alpha \leq 0\} \cdot 2^{-2/3} \alpha^2 \right).
\end{split}
\end{equation}
Combining (\ref{DefH}) and (\ref{LS3}), we readily deduce (\ref{mainLE1}) with $y = 2\tilde{r} + (1/2){\bf 1} \{\alpha \leq 0\} \cdot \alpha^2 $.\\

{\bf \raggedleft Step 2.} In this step we prove (\ref{LS2}). From Theorem \ref{S1Thm1} (here we use that $\tau \in (0, \tilde{\tau}]$) and Definition \ref{DefScale} we know that 
\begin{equation}\label{LS4}
\Ehf \left[ f_n (X_n - \tilde{r}) \right] = 1 + \sum_{k = 1}^\infty I(k, t_n).
\end{equation}
We observe that each summand in (\ref{LS4}) converges by Proposition \ref{S3MainP1}, and also from Proposition \ref{S3MainP2} we may exchange the order of the series and the limit by the dominated convergence theorem with dominating series $A^k \cdot k^{-k/2}$. Consequently, we conclude from Propositions \ref{S3MainP1} and \ref{S3MainP2} that 
\begin{equation*}
\begin{split}
\lim_{n \rightarrow \infty} \Ehf \left[ f_n (X_n - \tilde{r}) \right] = 1 + \sum_{k = 1}^\infty & \frac{1}{k!(2\pi \i)^{2k}} \int_{C_{0,\pi/4}^k}\hspace{-3mm} d\vec{u}  \int_{C_{-1,3\pi/4}^k}  \hspace{-3mm} d\vec{v} \\
& \det \left[ \frac{2v_a e^{u_a^3/48 - u_a^2 \alpha/ 8 - u_a \tilde{r}} }{(u_a-v_a)(v_b^2 - u_a^2)e^{v_a^3/48 - v_a^2 \alpha / 8 - v_a \tilde{r}}}\right]_{a,b = 1}^k.
\end{split}
\end{equation*}
We may now apply the change of variables $w_a = - 2^{-4/3} v_a$, $z_a = -2^{-4/3}u_a$ for $a = 1,\dots, k$ to get
 \begin{equation}\label{LS5}
\begin{split}
\lim_{n \rightarrow \infty} \Ehf \left[ f_n (X_n - \tilde{r}) \right] = 1 + \sum_{k = 1}^\infty & \frac{1}{k!(2\pi \i)^{2k}} \int_{C_{0,3\pi/4}^k}\hspace{-3mm} d\vec{z}  \int_{C_{2^{-4/3},\pi/4}^k}  \hspace{-3mm} d\vec{w} \\
& \det \left[ \frac{2w_a e^{-z_a^3/3 - z_a^2 2^{-1/3} \alpha + z_a 2^{4/3} \tilde{r}} }{(w_a - z_a)( z_a^2- w_b^2)e^{-w_a^3/3 - w_a^2 2^{-1/3} \alpha +  w_a 2^{4/3} \tilde{r}}}\right]_{a,b = 1}^k.
\end{split}
\end{equation}
In the last integral we may deform the $C_{2^{-4/3},\pi/4}$ contours to $C_{1,\pi/4}$ without crossing any poles of the integrands, and thus without affecting the value of the integrals. The decay necessary to deform the contours near infinity comes from the cubic terms in the exponential functions. At this point, we see that the right side of (\ref{LS5}) agrees with the right side of the first line of (\ref{CrossFD2}) with $t_1 = 2^{-1/3} \alpha$ and $y_1 = 2^{4/3} \cdot \tilde{r} + {\bf 1} \{\alpha \leq 0\} \cdot 2^{-2/3} \alpha^2$. From (\ref{LS5}) and Lemma \ref{S1LWD} we conclude (\ref{LS2}).

%-------------------------------------------------------------------------------------------------------------------------------------------------------------------------------------------------
%
% Section 3.3
%
%-------------------------------------------------------------------------------------------------------------------------------------------------------------------------------------------------
\subsection{Change of variables}\label{Section3.3} In this section we rewrite the function $H_k(\zeta)$ from (\ref{S1KthSum}) in a way that is suitable for our asymptotic analysis in the next two sections, see Lemma \ref{S3LRewrite}. In order to state our new formula we require a bit of notation, and our exposition here follows \cite[Section 3]{ED2020}.

\begin{definition}\label{DefFunS}
Fix $\tau \in (0,1)$. Let $w,z, u \in \mathbb{C}$ be such that $zw\neq 0$, $|z| \neq \tau^n |w|$ for any $n \in \mathbb{Z}$ and $u \not \in [0, \infty)$. For such a set of parameters we define the function
\begin{equation}\label{SpiralDef}
S(w, z; u,\tau) =   \sum_{m \in \mathbb{Z}} \frac{\pi \cdot [ - u ]^{ [\log \tau]^{-1} [\log z - \log w -  2m \pi  \i]}}{\sin(-\pi [\log \tau ]^{-1} [\log z - \log w -  2m \pi  \i])},
\end{equation}
where everywhere we take the principal branch of the logarthm, i.e. if $v = re^{\i \theta}$ with $r > 0$ and $\theta \in (-\pi, \pi]$ we set $\log v = \log r + \i \theta$. Observe that 
$$\Re \left[ -\pi [\log \tau]^{-1} [\log z - \log w -  2m \pi  \i ]  \right] = -\pi \cdot \frac{\log|z| - \log |w|}{\log \tau } \not \in \pi \cdot \mathbb{Z},$$
which implies that each of the summands in (\ref{SpiralDef}) is well-defined and finite.

It follows from (\ref{S2BoundSine}) that the series in (\ref{SpiralDef}) is absolutely convergent as we explain here. Put $A = [\log z - \log w] [\log \tau]^{-1}$, $B = - 2\pi [\log \tau]^{-1}$ and $-u = R e^{\i \phi}$ with $\phi \in (-\pi, \pi)$. From our assumption that $|z| \neq \tau^n |w|$ for any $n \in \mathbb{Z}$ and (\ref{S2BoundSine}) we conclude that for any $m \in \mathbb{Z}$ 
$$\left| \frac{\pi \cdot [ - u ]^{[\log \tau ]^{-1}  [\log z - \log w -  2m \pi \i ]}}{\sin(-\pi [\log \tau ]^{-1} [\log z - \log w  -  2m \pi  \i ])} \right| \leq  \pi c' \cdot e^{ |\phi| | A| + |\log R| |A| } e^{(|\phi| - \pi) B |m| }.  $$
The above shows that the sum in (\ref{SpiralDef}) is absolutely convergent by comparison with the geometric series $e^{(|\phi| - \pi) B |m|}.$
\end{definition}
\begin{remark}
In equation (\ref{SpiralDef}) we chose the principal branch of the logarithm for expressing $\log w$ and $\log z$. However, we could have chosen different branches for $\log w$ and $\log z$, and note that then $[\log z - \log w] [\log \tau]^{-1}$ would shift by $2k \pi \i [\log \tau ]^{-1}$ for some $k \in \mathbb{Z}$. Since the sum in the definition of $S(w, z; u,\tau)$ is over $\mathbb{Z}$ we see that such a shift does not change the value of $S(w, z; u,\tau)$. So even though the logarithm is a multi-valued function for fixed $u,\tau$ the function $S(w, z; u,\tau) $ as a function of $z,w$ is single valued and well-defined as long as $|z| \neq \tau^{n} |w|$ for some $n \in \mathbb{Z}$.
\end{remark}

We next summarize the main properties we require for the function $S(w, z; u,t)$ from Definition \ref{DefFunS} in the following lemma. 
\begin{lemma}\label{S3Analyticity}\cite[Lemma 3.9]{ED2020} Fix $\tau \in (0,1)$ and $R, r \in (0, \infty)$ such that $R > r > \tau R$. Denote by $A(r,R) \subset \mathbb{C}$ the annulus of inner radius $r$ and outer radius $R$ that has been centered at the origin. Then, the function $S(w, z; u,\tau)$ from Definition \ref{DefFunS} is well-defined for $(w,z,u) \in Y = \{ (x_1, x_2, x_3) \in  A(r,R) \times A(r,R) \times (\mathbb{C} \setminus [0, \infty)) : |x_2|< |x_1| \}$ and is jointly continuous in those variables (for fixed $\tau$) over $Y$. If we fix $u \in \mathbb{C} \setminus [0, \infty)$ and $w \in A(r,R)$ then as a function of $z$, $S(w, z; u,\tau)$ is analytic on $\{ \zeta \in A(r,R)  : |\zeta| < |w|\} $; analogously, if we fix $u \in \mathbb{C} \setminus [0, \infty)$ and $z \in A(r,R)$ then as a function of $w$, $S(w, z; u,\tau)$ is analytic on $\{ \zeta \in A(r,R)  : |\zeta| > |z|\} $. Finally, if we fix $w,z \in A(r,R)$ with $|w| > |z|$ then $S(w, z; u,\tau)$ is analytic in $\mathbb{C} \setminus [0, \infty)$ as a function of $u$.
\end{lemma}

With the above notation in place we can state the main result of this section.
\begin{lemma}\label{S3LRewrite} Fix $k \in \mathbb{N}$, $\zeta \in \mathbb{C}\setminus [0, \infty)$, set $u = (1-\tau) \zeta$ and let $H_k(\zeta)$ be as in (\ref{S1KthSum}). Then
\begin{equation}\label{S3ChangeVar}
\begin{split}
H_k(\zeta) = \frac{1}{k! (2\pi \i)^{2k}} \oint_{C^k_w} d\vec{w} \oint_{C^k_z} d\vec{z}\hspace{0.5mm} T(\vec{w}, \vec{z}) D(\vec{w}, \vec{z}) B(\vec{w}, \vec{z}; u ) G(\vec{w},\vec{z}), \mbox{ where }
\end{split}
\end{equation}
\begin{equation}\label{S3NewFun}
\begin{split}
&T(\vec{w}, \vec{z}) = \prod_{a = 1}^k \exp \left( \frac{(q-p)t}{1 + w_a} - \frac{(q-p)t}{1 + z_a}  \right) \left( \frac{1 + z_a}{1 + w_a} \right)^{x-1} \frac{(-w_a;\tau)_{\infty} (z_a^2;\tau)_{\infty}}{(-z_a;\tau)_{\infty} (z_aw_a;\tau)_{\infty}},\\
&D(\vec{w}, \vec{z}) = \det \left[ \frac{1}{w_a - z_b}\right]_{a,b = 1}^k, \hspace{2mm} B(\vec{w}, \vec{z}; u ) = \prod_{a = 1}^k \frac{S(w_a, z_a; u, \tau)}{- [\log \tau] z_a}, \mbox{ and } \\
&G(\vec{w},\vec{z}) = \prod_{1 \leq a < b \leq k}  \frac{(w_aw_b;\tau)_{\infty} (z_az_b;\tau)_{\infty}}{(z_aw_b ;\tau)_{\infty} (w_a z_b;\tau)_{\infty}}.
\end{split}
\end{equation}
In equation (\ref{S3ChangeVar}) we have that $C_w, C_z$ are positively oriented circles, centered at the origin, of radii $R_w$, $R_z$, respectively, such that $R_w > 1 > R_w^{-1} > R_z > \tau R_w$. 
\end{lemma}
\begin{proof} We first note that by Cauchy's theorem we may deform $C_w$ to $\gamma_{-1,0}$ as in Proposition \ref{S1Prop1}, and $C_z$ to $\tau^{1/2} \cdot \gamma_{-1,0}$ without affecting the value of the integral. In the latter statement we used that $S(w, z; u, \tau)$ is analytic separately in $w$ and $z$, in view of Lemma \ref{S3Analyticity}, and that in the process of deformation we do not cross any poles of the integrand. We proceed to denote $\gamma_{-1,0}$ by $C_w$ and $\tau^{1/2} \cdot \gamma_{-1,0}$ by $C_z$ in the remainder of the proof. 

Expanding the determinant $D(\vec{w}, \vec{z}) $ on the right side of (\ref{S3ChangeVar}) and the Cauchy determinant $ \det \left[ \frac{-1}{w_a \tau^{s_a} - w_b}  \right]_{a,b = 1}^k$ in the definition of $H_k(\zeta)$, see (\ref{S1SpecFun}) and  (\ref{S1KthSum}), we see that to prove (\ref{S3ChangeVar}) it suffices to show that for each $\sigma \in S_k$ (the permutation group of $k$ elements) and $\vec{w} \in C_w^k$ we have
\begin{equation*}
\begin{split}
& \int_{(1/2 + \i \mathbb{R})^k} d\vec{s}   \prod_{a = 1}^k \frac{\pi (-\zeta)^{s_a}\mathfrak{f}(w_a;s_a) \mathfrak{g}(w_a; s_a) }{\sin(-\pi s_a)(w_a  - w_{\sigma(a)}\tau^{s_{\sigma(a)}} )} \cdot \prod_{1 \leq a < b \leq k} \mathfrak{h}(w_a, w_b; s_a, s_b)  \\
& = \oint_{C^k_z} d\vec{z} \hspace{0.5mm} T(\vec{w}, \vec{z}) B(\vec{w}, \vec{z}; u ) G(\vec{w},\vec{z}) \cdot \prod_{a = 1}^k \frac{1}{w_a - z_{\sigma(a)}}.
\end{split}
\end{equation*}
Writing 
$$\tilde{\mathfrak{f}}(w;n) = \exp \left( \frac{(q-p)t}{1 + w} - \frac{(q-p)t}{1 + \tau^n w}  \right) \cdot \left( \frac{1 + \tau^n w}{1 + w} \right)^{x-1},$$
and recalling the definition of $\mathfrak{f}(w;n)$ from (\ref{S1BasicFun}), and that $u = (1-\tau)\zeta$, we see that to conclude (\ref{S3ChangeVar}) it suffices to show that for each $\sigma \in S_k$ and $\vec{w} \in C_w^k$ we have
\begin{equation}\label{S3S3E1}
\begin{split}
& \int_{(1/2 + \i \mathbb{R})^k} d\vec{s}   \prod_{a = 1}^k \frac{\pi (-u)^{s_a}\tilde{\mathfrak{f}}(w_a;s_a) \mathfrak{g}(w_a; s_a) }{\sin(-\pi s_a)(w_a  - w_{\sigma(a)}\tau^{s_{\sigma(a)}} )} \cdot \prod_{1 \leq a < b \leq k} \mathfrak{h}(w_a, w_b; s_a, s_b)  \\
& = \oint_{C^k_z} d\vec{z} \hspace{0.5mm} T(\vec{w}, \vec{z}) B(\vec{w}, \vec{z}; u ) G(\vec{w},\vec{z}) \cdot \prod_{a = 1}^k \frac{1}{w_a - z_{\sigma(a)}}.
\end{split}
\end{equation}

We next note that 
\begin{equation}\label{S3S3E2}
\begin{split}
& \int_{(1/2 + \i \mathbb{R})^k} d\vec{s}   \prod_{a = 1}^k \frac{\pi (-u)^{s_a}\tilde{\mathfrak{f}}(w_a;s_a) \mathfrak{g}(w_a; s_a) }{\sin(-\pi s_a)(w_a  - w_{\sigma(a)}\tau^{s_{\sigma(a)}} )} \cdot \prod_{1 \leq a < b \leq k} \mathfrak{h}(w_a, w_b; s_a, s_b)   \\
& = \sum_{n_1 \in \mathbb{Z}} \cdots \sum_{n_k \in \mathbb{Z}}\int_{1/2 + \i \pi [\log \tau ]^{-1}}^{1/2 - \i \pi [\log \tau]^{-1}}  ds_1  \cdots \int_{1/2 + \i \pi [\log \tau ]^{-1}}^{1/2 - \i \pi [\log \tau]^{-1}}  ds_k   \\
&  \prod_{a = 1}^k \frac{ \pi (-u)^{s_a - 2\pi \i n_a [\log \tau]^{-1}} \tilde{\mathfrak{f}}(w_a;s_a- 2\pi \i n_a [\log \tau]^{-1}) \mathfrak{g}(w_a; s_a- 2\pi \i n_a [\log \tau]^{-1})}{\sin(-\pi[ s_a- 2\pi \i n_a [\log \tau]^{-1}]) (w_a  - w_{\sigma(a)}\tau^{s_{\sigma(a)} - 2\pi \i n_{\sigma(a)}[\log \tau]^{-1} }   )}  \\
&\times \prod_{1 \leq a < b \leq k} \mathfrak{h}(w_a, w_b; s_a- 2\pi \i n_a [\log \tau]^{-1}, s_b- 2\pi \i n_b [\log \tau]^{-1})  \\
& = \sum_{n_1 \in \mathbb{Z}} \cdots \sum_{n_k \in \mathbb{Z}}\int_{1/2 + \i \pi [\log \tau ]^{-1}}^{1/2 - \i \pi [\log \tau]^{-1}}  ds_1  \cdots \int_{1/2 + \i \pi [\log \tau ]^{-1}}^{1/2 - \i \pi [\log \tau]^{-1}}  ds_k \\
&  \prod_{a = 1}^k \frac{ \pi (-u)^{s_a - 2\pi \i n_a [\log \tau]^{-1}} \tilde{\mathfrak{f}}(w_a;s_a) \mathfrak{g}(w_a; s_a)}{\sin(-\pi[ s_a- 2\pi \i n_a [\log \tau]^{-1}]) (w_a  - w_{\sigma(a)}\tau^{s_{\sigma(a)}  }   )} \prod_{1 \leq a < b \leq k} \mathfrak{h}(w_a, w_b; s_a, s_b).
\end{split}
\end{equation}
We mention that in deriving the last equality we used that $\tau^{2\pi \i n [\log \tau]^{-1}} = 1$ for all $n \in \mathbb{Z}$ and the fact that $\tilde{\mathfrak{f}}(w;s)$, $\mathfrak{g}(w;s)$ depend on $s$ only through $\tau^s$, while $\mathfrak{h}(w_1,w_2; s_1, s_2)$ depends on $s_1, s_2$ only through $\tau^{s_1}$ and $\tau^{s_2}$, see (\ref{S1BasicFun}).

Applying the change of variables $z_a = w_a \tau^{s_a}$ for $a =1, \dots, k$ in (\ref{S3S3E2}) we conclude that 
\begin{equation}\label{S3S3E3}
\begin{split}
& \int_{(1/2 + \i \mathbb{R})^k} d\vec{s}   \prod_{a = 1}^k \frac{\pi (-u)^{s_a}\tilde{\mathfrak{f}}(w_a;s_a) \mathfrak{g}(w_a; s_a) }{\sin(-\pi s_a)(w_a  - w_{\sigma(a)}\tau^{s_{\sigma(a)}} )}  \prod_{1 \leq a < b \leq k} \mathfrak{h}(w_a, w_b; s_a, s_b)   \\
&= \sum_{n_1 \in \mathbb{Z}} \cdots \sum_{n_k \in \mathbb{Z}}\int_{C_z^k} d\vec{z}\hspace{0.5mm} T(\vec{w}, \vec{z}) G(\vec{w}, \vec{z})  \prod_{a = 1}^k \frac{1}{(w_a  - z_{\sigma(a)}) z_a [-\log \tau]} \times \\
& \prod_{a = 1}^k \frac{ \pi (-u)^{[\log z_a - \log w_a ][\log \tau]^{-1} - 2\pi \i n_a [\log \tau]^{-1}}}{\sin(-\pi[ [\log z_a - \log w_a ][\log \tau]^{-1}- 2\pi \i n_a [\log \tau]^{-1}])},
\end{split}
\end{equation}
where we mention that the extra $-1$ sign in (\ref{S3S3E3}) comes from the fact that $C_z$ is positively oriented while $w_a \tau^{s_a}$ covers $C_z$ with negative orientation as $s_a$ varies from $1/2 + \i \pi [\log \tau ]^{-1}$ upward to $1/2 - \i \pi [\log \tau ]^{-1}$. From Fubini's theorem and (\ref{S2BoundSine}) we can put the sums in (\ref{S3S3E3}) inside the integral at which point we obtain (\ref{S3S3E1}), as desired.
\end{proof}

%-------------------------------------------------------------------------------------------------------------------------------------------------------------------------------------------------
% Section 4
%
%-------------------------------------------------------------------------------------------------------------------------------------------------------------------------------------------------
\section{Asymptotic analysis: Part I}\label{Section4} The goal of this section is to prove Proposition \ref{S3MainP1}. In Section \ref{Section4.1} we use Lemma \ref{S3LRewrite} to express $I(k,t)$ from Proposition \ref{S3MainP1} as a $2k$-fold contour integral, see (\ref{S4Ikt}), which is suitable for asymptotic analysis. The formula for $I(k,t)$ in (\ref{S4Ikt}) involves several functions and in Section \ref{Section4.2} we establish various estimates for these functions. The proof of Proposition \ref{S3MainP1} is presented in Section \ref{Section4.3}. It is based on the method of steepest descent and relies on the formula for $I(k,t)$ from (\ref{S4Ikt}) and the results from Section \ref{Section4.2}. The lemmas in Section \ref{Section4.1} and \ref{Section4.2} are proved in Section \ref{Section4.4}.

%-------------------------------------------------------------------------------------------------------------------------------------------------------------------------------------------------
% Section 4.1
%
%-------------------------------------------------------------------------------------------------------------------------------------------------------------------------------------------------
\subsection{Formula for $I(k,t)$}\label{Section4.1} The goal of this section is to get a formula for $I(k,t)$ from Proposition \ref{S3MainP1} that is suitable for asymptotic analysis. We begin with a lemma, which is proved in Section \ref{Section4.4}.
\begin{lemma}\label{DetFor1} Let $k \in \mathbb{N}$, $z_1, \dots, z_k, w_1, \dots, w_k \in \mathbb{C}$ be such that $z_i \neq w_j$ and $z_i \neq w_j^{-1}$ for $1 \leq i,j \leq k$. Then we have 
\begin{equation}\label{S5CD1}
 \frac{\prod_{1 \leq i < j \leq k} (1 - w_i w_j) (1 - z_i z_j) (w_i - w_j)(z_j - z_i)}{\prod_{i,j = 1}^k(1 -w_iz_j)(w_i - z_j) } = \det \left[ \frac{1}{(w_i -z_j)(1 - w_iz_j)} \right]_{i,j = 1}^k. 
\end{equation} 
\end{lemma}

The next definition introduces various contours, which will be required in our arguments.
\begin{definition}\label{S4Contours} For each $\epsilon \in [0,1]$ and $\rho \geq 0$, we let $\gamma^{+}_{\rho, \epsilon} = \gamma^{+,0}_{\rho, \epsilon} \cup  \gamma^{+,1}_{\rho, \epsilon}$  and $\gamma^{-}_{\rho, \epsilon} = \gamma^{-,0}_{\rho, \epsilon} \cup  \gamma^{-,1}_{\rho, \epsilon}$ be two contours, where
\begin{equation*}
\begin{split}
&\gamma^{\pm,0}_{\rho, \epsilon} := \{ z \in \mathbb{C}:   z = \rho \pm (a + \i a) \mbox{ or } z = \rho \pm (a - \i a) \mbox{ for $a \in [0, \epsilon]$} \}, \\
& \gamma^{\pm,1}_{\rho, \epsilon} :=  \{z \in \mathbb{C}: \epsilon \leq |\Im(z)| \leq \pi, \Re(z) = \rho \pm \epsilon \}.
\end{split}
\end{equation*}
The contours $\gamma^{+}_{\rho, \epsilon}, \gamma^{-}_{\rho, \epsilon}$ are oriented to have increasing imaginary part.

For $t > 0$ we define the contours
$$\Gamma_{u, \epsilon, t} = t^{1/3} \cdot \gamma^{+}_{0, \epsilon}, \hspace{2mm} \Gamma_{v, \epsilon, t} = t^{1/3} \cdot \gamma^{-}_{-t^{-1/3}, \epsilon}, \hspace{2mm} \Gamma^i_{u, \epsilon, t} = t^{1/3} \cdot \gamma^{+,i}_{0, \epsilon}, \mbox{ and }\Gamma^{i}_{v, \epsilon, t} = t^{1/3} \cdot \gamma^{-,i}_{-t^{-1/3}, \epsilon} \mbox{ for } i = 0,1.$$
All of the above contours are oriented in the direction of increasing imaginary part. 

We also define $C_{w,\epsilon,t}$ and $C_{z,\epsilon,t}$ to be the positively oriented contours obtained from $\gamma^{+}_{0, \epsilon}$ and $\gamma^{-}_{-t^{-1/3}, \epsilon}$, respectively, under the map $x \rightarrow e^x$. Observe that $C_{w,\epsilon,t}$ and $C_{z,\epsilon,t}$ are piecewise smooth positively oriented contours, $C_{z,\epsilon,t}$ is contained in the closed unit disc in $\mathbb{C}$, which in turn is contained in the region enclosed by $C_{w,\epsilon,t}$. Some of the contours in the definition are depicted in Figure \ref{S4_1}.
\end{definition}
\begin{figure}[h]
\scalebox{0.6}{\includegraphics{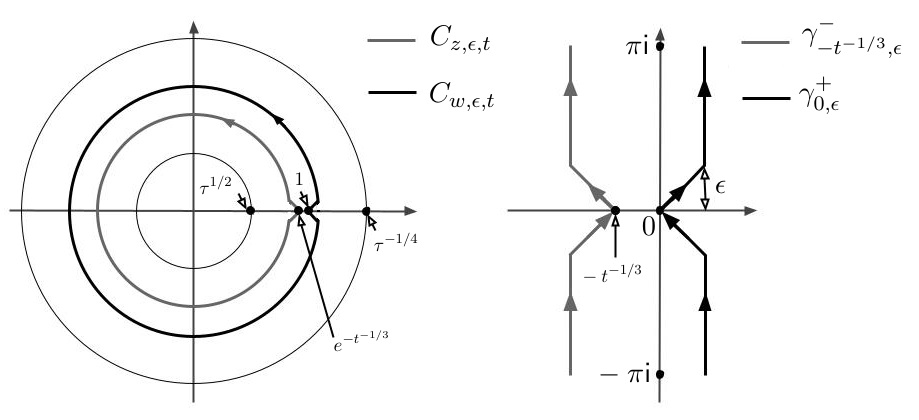}}
\captionsetup{width=\linewidth}
 \caption{ The left figure depicts the contours $C_{w,\epsilon,t}$, $C_{z,\epsilon,t}$ as well as the circles of radii $\tau^{-1/4}$ and $\tau^{1/2}$. The right figure depicts the contours $\gamma^{+}_{0, \epsilon}$ and $\gamma^{-}_{-t^{-1/3}, \epsilon}$.}
\label{S4_1}
\end{figure}

The following lemma introduces a certain analytic function $F(z)$, which will appear in our analysis, and establishes a few of its properties. Its proof is given in Section \ref{Section4.4}.
\begin{lemma}\label{S4LemmaTaylor} For $z \in \mathbb{C}$ such that $e^z \neq -1$ define the function
\begin{equation}\label{S4DefF}
F(z) = \frac{1}{1 + e^z} + \frac{z}{4} - \frac{1}{2}.
\end{equation}
There exist universal constants $\rr \in (0,1/2)$, $C_1, C_2 > 0$, such that
\begin{equation}\label{S4TaylorF}
\begin{split}
&|F(z) - z^3/48| \leq C_1 |z|^4 \mbox{ for $|z| \leq 3\rr$};\\
& \Re[ F(z)] \leq - \frac{|z|^3}{200} + C_2 |\rho|^3 \mbox{ when } z \in \gamma_{\rho, \epsilon}^+, |\Im(z)| \leq \epsilon \mbox{ with $\rho \in \mathbb{R}$, } 0 \leq |\rho| \leq \epsilon \leq \rr; \\
&  \Re[ F(z)] \geq  \frac{|z|^3}{200} - C_2 |\rho|^3 \mbox{ when } z \in \gamma_{\rho, \epsilon}^-, |\Im(z)| \leq \epsilon \mbox{ with $\rho \in \mathbb{R}$, } 0 \leq |\rho| \leq \epsilon \leq \rr,
\end{split}
\end{equation}
where $\gamma^{\pm}_{\rho, \epsilon}$ are as in Definition \ref{S4Contours}. Furthermore, for any $x > 0$ we have 
\begin{equation}\label{S4Decay}
\begin{split}
\frac{d}{dy} \Re [F(x \pm \i y)] \leq 0 \mbox{ and } \frac{d}{dy} \Re [F(- x \pm \i y)] \geq 0 \mbox{ for } y \in [0, \pi].
\end{split}
\end{equation}
\end{lemma}

\begin{definition}\label{S4DefEpsilon}
Assume the same notation as in Definition \ref{DefScale}. We let $\epsilon_0 > 0$ be sufficiently small so that $\tau^{-1} \geq e^{20\epsilon_0}$ and $\epsilon_0 \leq \rr$, where $\rr$ is as in Lemma \ref{S4LemmaTaylor}. We also let $T_0$ be sufficiently large so that for $t \geq T_0$ we have $t^{-1/3} \leq \epsilon_0$. 
\end{definition}

With the above notation, we are ready to state our formula for $I(k,t)$. Let $\epsilon_0, T_0$ be as in Definition \ref{S4DefEpsilon}, and for $t \geq T_0$ let $\Gamma_{u,\epsilon_0, t}$, $\Gamma_{v,\epsilon_0, t} $ be as in Definition \ref{S4Contours}. We then have the following formula for all $t \geq T_0$ and $k \in \mathbb{N}$
\begin{equation}\label{S4Ikt}
\begin{split}
I(k,t)= \frac{1}{k! (2\pi \i)^{2k}} \int_{\Gamma^k_{u,\epsilon_0, t} } d\vec{u}  \int_{\Gamma^k_{v,\epsilon_0, t} } d \vec{v} & \prod_{i = 1}^5 A_i ( \vec{u}, \vec{v}; t) \prod_{i = 1}^2 B_i( \vec{u}, \vec{v}; t),
\end{split}
\end{equation}
where 
\begin{equation}\label{S41As}
\begin{split}
&A_1 (\vec{u}, \vec{v}; t) = \prod_{a = 1}^k e^{t [F(t^{-1/3}u_a) - F(t^{-1/3}v_a)]}, \\
&A_2(\vec{u}, \vec{v}; t) = \prod_{a = 1}^k \left( \sum_{m \in \mathbb{Z}} \frac{t^{-1/3} [-\log \tau]^{-1} \pi \cdot e^{ 2m \pi  \i  [(1/4)t +(1/2)( \lfloor t^{2/3}\alpha \rfloor - 1 ) -  t^{1/3} \tilde{r} ]}}{\sin(-\pi  [\log \tau ]^{-1}  [t^{-1/3} (v_a - u_a) -  2m \pi  \i ])} \right), \\
& A_3 (\vec{u}, \vec{v}; t)  = \prod_{a = 1}^k \left( \frac{(1 + e^{t^{-1/3} v_a })e^{-t^{-1/3} v_a/2} }{(1 + e^{t^{-1/3} u_a}) e^{-t^{-1/3} u_a/2} } \right)^{\lfloor t^{2/3} \alpha \rfloor -1}, \\
& A_4(\vec{u}, \vec{v}; t)   = \prod_{a = 1}^k \exp \left( \tilde{r} [ v_a -  u_a] \right) \\
& A_5 (\vec{u}, \vec{v}; t) = t^{k/3}  \cdot  \prod_{a = 1}^k \frac{(-e^{t^{-1/3} u_a} ;\tau)_{\infty} (e^{2 t^{-1/3} v_a}  ;\tau)_{\infty} e^{t^{-1/3} u_a} }{(-e^{ t^{-1/3} v_a} ;\tau)_{\infty} (\tau e^{t^{-1/3} (u_a + v_a)} ;\tau)_{\infty} },
\end{split}
\end{equation}

\begin{equation}\label{S41Bs}
\begin{split}
&B_1(\vec{u}, \vec{v}; t) = \det \left[ \frac{t^{-2/3}}{(e^{t^{-1/3} u_a} -e^{t^{-1/3} v_b})(1 - e^{t^{-1/3} (u_a + v_b)})} \right]_{a,b = 1}^k, \\
&B_2(\vec{u}, \vec{v}; t)  = \prod_{1 \leq a < b \leq k } \frac{(\tau e^{t^{-1/3}(u_a + u_b)}; \tau)_{\infty}(\tau e^{t^{-1/3}(v_a + v_b)}; \tau)_{\infty}}{(\tau e^{t^{-1/3}(u_a + v_b)}; \tau)_{\infty}(\tau e^{t^{-1/3}(v_a + u_b)}; \tau)_{\infty}}.
\end{split}
\end{equation}

Let us briefly explain the origin of (\ref{S4Ikt}). Our starting point is the formula for $H_k(\zeta)$ from (\ref{S3ChangeVar}) in Lemma \ref{S3LRewrite}, where we take $R_w = \tau^{-1/4}$ and $R_z = \tau^{1/2}$. At this point we can deform the contours $C_w$ in (\ref{S3ChangeVar}) to $C_{w,\epsilon,t}$ as in Definition \ref{S4Contours} without crossing any poles and thus without affecting the value of the integral by Cauchy's theorem. Here we used that $S(w, z; u,\tau) $ is analytic in $w$ for fixed $z$, as follows from Lemma \ref{S3Analyticity}. We mention that the integrand in (\ref{S3ChangeVar}) has singularities at $w_a = \tau^m z_a$ for all $m \in \mathbb{Z}$, coming from $B(\vec{w}, \vec{z}; u)$,  at $w_a = z_b$, coming from $D(\vec{w}, \vec{z})$, at $w_a = z_b^{-1} \tau^{-m}$ for all $m \in \mathbb{Z}_{\geq 0}$, coming from $T(\vec{w}, \vec{z})$ and $G(\vec{w}, \vec{z})$, and all of these singularities lie outside of the annulus of inner radius $1$ and outer radius $\tau^{-1/4}$, where the deformation takes place. There is also a singularity at $w_a = -1$, coming from $T(\vec{w}, \vec{z})$, which is also not crossed in the deformation, see the left part of Figure \ref{S4_1}. After we deform all the $C_w$ contours to $C_{w,\epsilon,t}$ we proceed to deform the $C_z$ contours to $C_{z,\epsilon,t}$ as in Definition \ref{S4Contours}. Again, in the process of deformation we do not cross any poles, and by Cauchy's theorem we do not change the value of the integral. We mention that the only poles in the annulus of inner radius $\tau^{1/2}$ and outer radius $1$, where the deformation takes place, come from $z_a = -1$ (from $T(\vec{w}, \vec{z})$), $z_a = w_b$ (from $D(\vec{w}, \vec{z})$), and at $z_a = w_b^{-1}$ (from $T(\vec{w}, \vec{z})$ and $G(\vec{w}, \vec{z})$). As $C_{z,\epsilon,t}$ is strictly contained in the unit circle, and $C_{w,\epsilon,t}$ is outside of it we do not cross the $z_a = -1$ and the $z_a = w_b$ poles. On the other hand, the fact that $\gamma^{+}_{0, \epsilon}$ is to the left of $- \gamma^{-}_{-t^{-1/3}, \epsilon}$, as in Definition \ref{S4Contours}, ensures that we also do not cross $z_a = w_b^{-1}$.

Once we have deformed the contours in (\ref{S3ChangeVar}) to $C_{w,\epsilon,t}$ and $C_{z,\epsilon,t}$, we may apply the Cauchy determinant formula (\ref{S2CauchyDet}) and Lemma \ref{DetFor1} to get
\begin{equation*}
\begin{split}
&I(k,t) = \frac{1}{k! (2\pi \i)^{2k}} \oint_{C^k_{w,\epsilon,t}}\hspace{-2mm}  d\vec{w} \oint_{C^k_{z,\epsilon,t}} \hspace{-2mm} d\vec{z}  \det \left[ \frac{1}{(w_a - z_b)(1-z_aw_b)}\right]_{a,b = 1}^k \cdot  \hspace{-3mm} \prod_{1 \leq a < b \leq k}  \frac{(\tau w_aw_b;\tau)_{\infty} (\tau z_az_b;\tau)_{\infty}}{(\tau z_aw_b ;\tau)_{\infty} (\tau w_a z_b;\tau)_{\infty}}  \\
& \times \prod_{a = 1}^k \frac{S(w_a, z_a; u, \tau)}{- [\log \tau] z_a} \cdot \exp \left( \frac{t}{1 + w_a} - \frac{t}{1 + z_a}  \right) \cdot \left( \frac{1 + z_a}{1 + w_a} \right)^{\lfloor \alpha t^{2/3} \rfloor-1}  \cdot \frac{(-w_a;\tau)_{\infty} (z_a^2;\tau)_{\infty}}{(-z_a;\tau)_{\infty} (\tau z_a w_a;\tau)_{\infty} },
\end{split}
\end{equation*}
where $ u =  - \tau^{- (1/4)t - (1/2)( \lfloor t^{2/3}\alpha \rfloor - 1 )  + t^{1/3} \tilde{r} }$, and $\alpha, \tilde{r}$ are as in Definition \ref{S3ZetaScale}. Applying the change of variables $w_a = e^{t^{-1/3} u_a}$, $z_a = e^{t^{-1/3} v_a}$ for $a = 1, \dots, k$ and using that 
\begin{equation*}
\begin{split}
 &\prod_{a = 1}^k \frac{S(w_a, z_a; u, \tau)}{- [\log \tau] } = \prod_{a = 1}^k \left( \sum_{m \in \mathbb{Z}} \frac{ [-\log \tau]^{-1} \pi \cdot e^{ 2m \pi  \i  [(1/4)t +(1/2)( \lfloor t^{2/3}\alpha \rfloor - 1 ) -  t^{1/3} \tilde{r} ]}}{\sin(-\pi  [\log \tau ]^{-1}  [t^{-1/3} (v_a - u_a) -  2m \pi  \i ])} \right)  \\
&\times \prod_{a = 1}^k \exp\left(\frac{t}{4} (- t^{-1/3} v_a + t^{-1/3} u_a) + (\lfloor t^{2/3} \alpha \rfloor - 1) \left(-t^{-1/3} v_a/2 +  t^{-1/3} u_a/2 \right)  \tilde{r} [v_a - u_a] \right) ,
\end{split}
\end{equation*} 
which holds from (\ref{SpiralDef}), we arrive at (\ref{S4Ikt}).

%-------------------------------------------------------------------------------------------------------------------------------------------------------------------------------------------------
% Section 4.2
%
%-------------------------------------------------------------------------------------------------------------------------------------------------------------------------------------------------
\subsection{Preliminary estimates}\label{Section4.2} In this section we state various estimates for the functions $A_i(\vec{u}, \vec{v}; t) $ for $i = 1, \dots, 5$ and $B_j(\vec{u}, \vec{v}; t)$ for $j = 1,2$ from (\ref{S41As}) and (\ref{S41Bs}) in a sequence of lemmas, whose proofs are given in Section \ref{Section4.4}. These estimates will be required in the proof of Proposition \ref{S3MainP1} in Section \ref{Section4.3}. In the lemmas below $\Gamma_{u,\epsilon, t}, \Gamma_{v,\epsilon, t}$ are as in Definition \ref{S4Contours}.

\begin{lemma}\label{S42L1} Let $F(z)$ and $\rr > 0$ be as in Lemma \ref{S4LemmaTaylor}, and  $\epsilon \in (0, \rr]$. We can find $T_{\epsilon,1}, A_{\epsilon,1}, a_{\epsilon,1} > 0$, depending on $\epsilon$ alone, such that the following holds. If $t \geq T_{\epsilon,1}$, $u \in \Gamma_{u,\epsilon, t}, v \in \Gamma_{v,\epsilon, t}$
\begin{equation}\label{S42L1E}
\left| e^{t [F(t^{-1/3}u) - F(t^{-1/3}v)]}  \right| \leq A_{\epsilon,1} \cdot \exp \left( - a_{\epsilon,1} |u|^3 -  a_{\epsilon,1} |v|^3\right).
\end{equation}
\end{lemma}

\begin{lemma}\label{S42L2} Let $\rr > 0$ be as in Lemma \ref{S4LemmaTaylor}, $\tt \in (0,1)$ and $\epsilon \in (0,\rr]$ be sufficiently small so that $e^{20\epsilon} \leq \tt^{-1}$. We can find $T_{\epsilon, 2}, A_{\epsilon, 2} > 0$, depending on $\epsilon$ alone, such that the following holds. If $t \geq T_{\epsilon, 2}$, $\tau \in (0, \tt]$, $\alpha, \tilde{r} \in \mathbb{R}$, $u \in \Gamma_{u,\epsilon, t}, v \in \Gamma_{v,\epsilon, t}$
\begin{equation}\label{S42L2E}
\left|  \sum_{m \in \mathbb{Z}} \frac{t^{-1/3} [-\log \tau]^{-1} \pi \cdot e^{ 2m \pi  \i  [(1/4)t +(1/2)( \lfloor t^{2/3}\alpha \rfloor - 1 ) - t^{1/3} \tilde{r} ]}}{\sin(-\pi [\log \tau ]^{-1} [t^{-1/3} (v - u)   -  2m \pi  \i ])}  \right| \leq A_{\epsilon, 2}.
\end{equation}
\end{lemma}

\begin{lemma} \label{S42L3} Let $\rr > 0$ be as in Lemma \ref{S4LemmaTaylor}, $\alpha \in \mathbb{R}$ and  $\epsilon \in (0, \rr]$. We can find $T_{\epsilon,3}, a_{\epsilon,3} > 0$, depending on $\epsilon$ alone, such that the following holds. If $t \geq T_{\epsilon,3}$, $u \in \Gamma_{u,\epsilon, t}, v \in \Gamma_{v,\epsilon, t}$
\begin{equation}\label{S42L3E}
\left| \left( \frac{(1 + e^{t^{-1/3} v })e^{-t^{-1/3} v/2} }{(1 + e^{t^{-1/3} u}) e^{-t^{-1/3} u/2} } \right)^{\lfloor t^{2/3} \alpha \rfloor -1}  \right| \leq  \exp \left(  a_{\epsilon,3} \cdot (1 + |\alpha|) \cdot  (|u|^2 + |v|^2)\right).
\end{equation}
\end{lemma}

\begin{lemma} \label{S42L4} Let $\rr > 0$ be as in Lemma \ref{S4LemmaTaylor}, $\tt \in (0,1)$ and  $\epsilon \in (0, \rr]$ be sufficiently small so that $e^{20\epsilon} \leq \tt^{-1}$. We can find $T_{\epsilon,4}, A_{\epsilon,4} > 0$, depending on $\epsilon$ alone, such that the following holds. If $t \geq T_{\epsilon,4}$, $\tau \in (0, \tt]$, $u \in \Gamma_{u,\epsilon, t}, v \in \Gamma_{v,\epsilon, t}$
\begin{equation}\label{S42L4E}
\left|  \frac{t^{1/3} (-e^{t^{-1/3} u} ;\tau)_{\infty} (e^{2 t^{-1/3} v}  ;\tau)_{\infty} e^{t^{-1/3} u} }{(-e^{ t^{-1/3} v} ;\tau)_{\infty} (\tau e^{t^{-1/3} (u + v)} ;\tau)_{\infty} } \right| \leq A_{\epsilon,4} \cdot |v|.
\end{equation}
\end{lemma}

\begin{lemma}\label{S42L5} Let $\rr > 0$ be as in Lemma \ref{S4LemmaTaylor}, $ \epsilon \in (0, \rr]$ and $k \in \mathbb{N}$. Suppose that $u_a \in \Gamma_{u,\epsilon, t}, v_a \in \Gamma_{v,\epsilon, t}$ for $a \in \{1, \dots, k\}$. We can find $T_{\epsilon,5}, A_{\epsilon,5} > 0$, depending on $\epsilon$ alone, such that if $B_1(\vec{u}, \vec{v}; t)$ is as in (\ref{S41Bs}) and $t \geq T_{\epsilon,5}$, then
\begin{equation}\label{S42L5E}
\left| B_1(\vec{u}, \vec{v}; t)  \right| \leq A_{\epsilon,5}^k \cdot k^{k/2}.
\end{equation}
\end{lemma}

\begin{lemma}\label{S42L6} Let $\rr > 0$ be as in Lemma \ref{S4LemmaTaylor}, $\tt \in (0,1)$ and  $\epsilon \in (0, \rr]$ be sufficiently small so that $e^{20\epsilon} \leq \tt^{-1}$. We can find $T_{\epsilon,6}, A_{\epsilon,6} > 0$, depending on $\epsilon$ alone, such that the following holds. If $t \geq T_{\epsilon,6}$, $\tau \in (0, \tt]$, $u_a \in \Gamma_{u,\epsilon, t}, v_a \in \Gamma_{v,\epsilon, t}$ for $a \in \{1, \dots, k\}$ and $B_2(\vec{u}, \vec{v}; t)$ is as in (\ref{S41Bs}), then
\begin{equation}\label{S42L6E}
\left| B_2(\vec{u}, \vec{v}; t)  \right| \leq A_{\epsilon,6}^{k^2}.
\end{equation}
\end{lemma}

%-------------------------------------------------------------------------------------------------------------------------------------------------------------------------------------------------
% Section 4.3
%
%-------------------------------------------------------------------------------------------------------------------------------------------------------------------------------------------------
\subsection{Proof of Proposition \ref{S3MainP1}}\label{Section4.3} In this section we present the proof of Proposition \ref{S3MainP1}, and for clarity we split the proof into three steps.\\

{\bf \raggedleft Step 1.} Let $\epsilon_0$ be as in Definition \ref{S4DefEpsilon}, and let $T_1 > 0$ be sufficiently large, so that $T_1 \geq T_0$ and $T_1 \geq \max_{i = 1,\dots, 6} T_{\epsilon_0,i}$, where $T_0$ is as in Definition \ref{S4DefEpsilon}, and $T_{\epsilon_0,i}$ are as in Lemmas \ref{S42L1}-\ref{S42L6}. For $k \in \mathbb{N}$ and $t \geq T_1$, we define
\begin{equation}\label{S43E1}
 J(k,t) :=  \int_{(\Gamma^0_{u,\epsilon_0, t})^k } d\vec{u}  \int_{(\Gamma^0_{v,\epsilon_0, t})^k } d \vec{v} \prod_{i = 1}^5 A_i ( \vec{u}, \vec{v}; t) \prod_{i = 1}^2 B_i( \vec{u}, \vec{v}; t),
\end{equation}
where $\Gamma^0_{u,\epsilon_0, t}, \Gamma^0_{v,\epsilon_0, t}$ are as in Definition \ref{S4Contours}, $A_i ( \vec{u}, \vec{v}; t)$ are as in (\ref{S41As}), and $B_i( \vec{u}, \vec{v}; t)$ are as in (\ref{S41Bs}).
We claim that 
\begin{equation}\label{S43E2}
 \lim_{t \rightarrow \infty} J(k,t) =  \int_{C_{0,\pi/4}^k}\hspace{-3mm} d\vec{u}  \int_{C_{-1,3\pi/4}^k}  \hspace{-3mm} d\vec{v} \det \left[ \frac{2v_a e^{u_a^3/48 - u_a^2 \alpha/ 8 - u_a \tilde{r}} }{(u_a-v_a)(v_b^2 - u_a^2)e^{v_a^3/48 - v_a^2 \alpha / 8 - v_a \tilde{r}}}\right]_{a,b = 1}^k.
\end{equation}
We prove (\ref{S43E2}) in the steps below. Here, we assume its validity and conclude the proof of (\ref{S3KE1}).\\

It follows from (\ref{S4Ikt}) that 
\begin{equation}\label{QP1}
\begin{split}
&  k!(2\pi \i)^{2k} I(k,t) - J(k,t)  = \sum_{\sigma_1, \dots, \sigma_k, \tau_1, \dots, \tau_k \in \{0,1\}: \sum_i \sigma_i + \tau_i > 0}\int_{\Gamma^{\sigma_1}_{u,\epsilon_0, t} } \cdots \int_{\Gamma^{\sigma_k}_{u,\epsilon_0, t} }  d\vec{u}   \\
&\int_{\Gamma^{\tau_1}_{v,\epsilon_0, t} } \cdots \int_{\Gamma^{\tau_k}_{v,\epsilon_0, t} }  d \vec{v}  \prod_{i = 1}^5 A_i ( \vec{u}, \vec{v}; t) \prod_{i = 1}^2 B_i( \vec{u}, \vec{v}; t).
\end{split}
\end{equation}
It follows from Lemmas \ref{S42L1}, \ref{S42L2}, \ref{S42L3}, \ref{S42L4}, \ref{S42L5}, and \ref{S42L6} (all lemmas are applied for $\epsilon = \epsilon_0$ as in Definition \ref{S4DefEpsilon} and $\tau = \tau_1$) that for $\vec{u} \in \Gamma^k_{u,\epsilon_0, t}$, $\vec{v} \in \Gamma^k_{v,\epsilon_0, t}$, and $t \geq T_1$, we have
\begin{equation}\label{S43E3}
\begin{split}
&\left| \prod_{i = 1}^5 A_i ( \vec{u}, \vec{v}; t) \prod_{i = 1}^2 B_i( \vec{u}, \vec{v}; t)  \right| \leq A^k_{\epsilon_0,1} \cdot A^k_{\epsilon_0,2} \cdot A^k_{\epsilon_0,4}\cdot A^k_{\epsilon_0,5}\cdot A^{k^2}_{\epsilon_0,6} \cdot k^{k/2}  \\
&\times \prod_{i = 1}^k |v_i| \exp \left( - a_{\epsilon_0,1} ( |u_i|^3 + |v_i|^3) + a_{\epsilon_0,3} (1 + |\alpha|) (|u_i|^2 +|v_i|^2 ) + |\tilde{r}|(|u_i| + |v_i|) \right).
\end{split}
\end{equation}
Observe that if $t \geq T_1$, we have 
$$|z| \geq \epsilon_0 t^{1/3} \mbox{ for } z \in \Gamma^{1}_{u,\epsilon_0, t} \cup \Gamma^{1}_{v,\epsilon_0, t},\mbox{ and }|z| \leq 2\pi t^{1/3} \mbox{ for } z \in \Gamma_{u,\epsilon_0, t} \cup \Gamma_{v,\epsilon_0, t}.$$
Using the latter and (\ref{S43E3}), we conclude that there exist constants $A,a > 0$, such that for $t \geq T_1$
\begin{equation}\label{QP2}
\begin{split}
&\left| \prod_{i = 1}^5 A_i ( \vec{u}, \vec{v}; t) \prod_{i = 1}^2 B_i( \vec{u}, \vec{v}; t)  \right|  \leq A \exp \left( a t^{2/3} - a_{\epsilon_0,1} \epsilon_0^3 t \sum_{i = 1}^k (\sigma_i  + \tau_i) \right) \mbox{ if } \\
& (u_1, \dots, u_k) \in \Gamma^{\sigma_1}_{u,\epsilon_0, t} \times \cdots \times \Gamma^{\sigma_k}_{u,\epsilon_0, t} \mbox{ and } (v_1, \dots, v_k) \in \Gamma^{\tau_1}_{v,\epsilon_0, t} \times \cdots \times \Gamma^{\tau_k}_{v,\epsilon_0, t}.
\end{split}
\end{equation}
Combining (\ref{QP1}), (\ref{QP2}) with the fact that the lengths of $\Gamma^{i}_{u,\epsilon_0, t}, \Gamma^{i}_{v,\epsilon_0, t}$ for $i = 0,1$ are at most $4\pi t^{1/3}$, we conclude that for all large $t$ we have
$$\left| J(k,t)  - k!(2\pi \i)^{2k} I(k,t) \right| \leq \exp \left( - a_{\epsilon_0,1} \epsilon_0^3 t/2 \right).$$
The latter inequality and (\ref{S43E2}) together imply (\ref{S3KE1}).\\

{\bf \raggedleft Step 2.} From (\ref{S43E1}) we have that 
\begin{equation}\label{S43E4}
\begin{split}
&J(k,t) =  \int_{C^k_{0, \pi/4} }   \int_{C_{-1, 3\pi/4}^k } g_t(\vec{u}, \vec{v}) d \vec{v}  d\vec{u}, \mbox{ where } \\
&g_t(\vec{u}, \vec{v}) = \prod_{i =1}^k {\bf 1}\{ |\Im(u_i)| \leq \epsilon_0 t^{1/3}, |\Im(v_i)| \leq \epsilon_0 t^{1/3}  \}  \prod_{i = 1}^5 A_i ( \vec{u}, \vec{v}; t) \prod_{i = 1}^2 B_i( \vec{u}, \vec{v}; t).
\end{split}
\end{equation}
We claim that for all $\vec{u} \in C^k_{0, \pi/4}$ and $\vec{v} \in C_{-1, 3\pi/4}^k$ we have
\begin{equation}\label{S43E5}
\begin{split}
\lim_{t \rightarrow \infty} g_t(\vec{u}, \vec{v}) = \det \left[ \frac{2v_a e^{u_a^3/48 - u_a^2 \alpha/ 8 - u_a \tilde{r}} }{(u_a-v_a)(v_b^2 - u_a^2)e^{v_a^3/48 - v_a^2 \alpha / 8 - v_a \tilde{r}}}\right]_{a,b = 1}^k.
\end{split}
\end{equation}
We prove (\ref{S43E5}) in the next step. Here, we assume its validity and conclude the proof of (\ref{S43E2}).\\

In view of  (\ref{S43E5}) and the dominated convergence theorem, with dominating function given by the right side of (\ref{S43E3}), we may take the $t\rightarrow \infty$ limit in (\ref{S43E4}) to get (\ref{S43E2}).\\

{\bf \raggedleft Step 3.} In this step we prove (\ref{S43E5}), and in the sequel we assume that $\vec{u} \in C^k_{0, \pi/4}$, and $\vec{v} \in C_{-1, 3\pi/4}^k$ are fixed. We begin to investigate the limits of $A_i (\vec{u}, \vec{v}; t) $ and $B_i (\vec{u}, \vec{v}; t) $ in the definition of $g_t(\vec{u}, \vec{v}) $.

From the first line of (\ref{S4TaylorF}) we have for all large $t$ and fixed $z \in \mathbb{C}$
$$\left| t F(t^{-1/3}z) - z^3/48 \right| \leq C_1 t^{-1/3} |z|^4, \mbox{ and so }\lim_{t \rightarrow \infty} t F(t^{-1/3}z)  = z^3/48.$$
In particular, we conclude that 
\begin{equation}\label{QP3}
\lim_{t \rightarrow \infty} A_1 (\vec{u}, \vec{v}; t) = \prod_{a = 1}^k \exp \left( u_a^3/48 - v_a^3/48\right).
\end{equation}
We also observe that for each $z \in \mathbb{C}$
$$\lim_{ t\rightarrow \infty} \sum_{m \in \mathbb{Z}} \frac{t^{-1/3} [-\log \tau]^{-1} \pi \cdot e^{ 2m \pi  \i  [(1/4)t +(1/2)( \lfloor t^{2/3}\alpha \rfloor - 1 ) -  t^{1/3} \tilde{r} ]}}{\sin(-\pi [\log \tau ]^{-1}  [t^{-1/3} z -  2m \pi  \i ])} = \frac{1}{z} ,$$
where we can exchange the order of the sum and the limit by (\ref{S2BoundSine}), and only the $m = 0$ summand contributes to the limit. We conclude that 
\begin{equation}\label{QP4}
\lim_{t \rightarrow \infty} A_2 (\vec{u}, \vec{v}; t) = \prod_{a = 1}^k \frac{1}{v_a - u_a}.
\end{equation}
By a direct Taylor expansion we have for each $\alpha \in \mathbb{R}$ and $z \in \mathbb{C}$
$$\lim_{ t\rightarrow \infty} \left( \frac{(1 + e^{t^{-1/3} z })e^{-t^{-1/3} z/2} }{2} \right)^{ \lfloor t^{2/3} \alpha \rfloor -1} = e^{ z^2 \alpha/8},$$
from which we conclude that 
\begin{equation}\label{QP5}
\lim_{t \rightarrow \infty} A_3 (\vec{u}, \vec{v}; t) = \prod_{a = 1}^k \exp \left( v_a^2 \alpha/8  -  u_a^2 \alpha/8\right).
\end{equation}
For every $z,w \in \mathbb{C}$ we have
\begin{equation*}
\begin{split}
&\lim_{t \rightarrow \infty}  \frac{t^{1/3} (-e^{t^{-1/3} z} ;\tau)_{\infty} (e^{2 t^{-1/3} w}  ;\tau)_{\infty} e^{t^{-1/3} z} }{(-e^{ t^{-1/3} w} ;\tau)_{\infty} (\tau e^{t^{-1/3} (z + w)} ;\tau)_{\infty} } = \lim_{t \rightarrow \infty}  t^{1/3} (1 - e^{2 t^{-1/3} w}) \\
& \times \lim_{t \rightarrow \infty}  \frac{ (-e^{t^{-1/3} z} ;\tau)_{\infty} (\tau e^{2 t^{-1/3} w}  ;\tau)_{\infty} e^{t^{-1/3} z} }{(-e^{ t^{-1/3} w} ;\tau)_{\infty} (\tau e^{t^{-1/3} (z + w)} ;\tau)_{\infty} } =   - 2w,
\end{split}
\end{equation*}
from which we conclude that 
\begin{equation}\label{QP6}
\lim_{t \rightarrow \infty} A_5 (\vec{u}, \vec{v}; t) = \prod_{a = 1}^k ( -2 v_a).
\end{equation}
Finally, we have the straightforward limits
\begin{equation}\label{QP7}
\lim_{t \rightarrow \infty} B_1(\vec{u}, \vec{v};t) = \det \left[ \frac{1}{v_b^2 - u_a^2} \right]_{a,b = 1}^k \mbox{, and } \lim_{t \rightarrow \infty} B_2(\vec{u}, \vec{v};t) = 1.
\end{equation}

Combining (\ref{QP3}), (\ref{QP4}, (\ref{QP5}), (\ref{QP6}) and (\ref{QP7}) we conclude that 
\begin{equation*}
\lim_{t \rightarrow \infty}   \prod_{i = 1}^5 A_i ( \vec{u}, \vec{v}; t) \prod_{i = 1}^2 B_i( \vec{u}, \vec{v}; t) = \prod_{a = 1}^k \frac{e^{u_a^3/48 -  u_a^2 \alpha /8 -  u_a \tilde{r} }}{e^{v_a^3/48 -  v_a^2 \alpha/8 -  v_a \tilde{r} }} \cdot \frac{(-2v_a)}{v_a - u_a} \times \det \left[ \frac{1}{v_b^2 - u_a^2} \right]_{a,b = 1}^k.
\end{equation*}
The last equation implies (\ref{S43E5}) once we use the multilinearity of the determinant.

%-------------------------------------------------------------------------------------------------------------------------------------------------------------------------------------------------
% Section 4.4
%
%-------------------------------------------------------------------------------------------------------------------------------------------------------------------------------------------------
\subsection{Proof of the lemmas from Sections \ref{Section4.1} and \ref{Section4.2}}\label{Section4.4} In this section we present the proofs of the eight lemmas from Sections \ref{Section4.1} and \ref{Section4.2}.\\

\begin{proof}[Proof of Lemma \ref{DetFor1}]
The proof we present here is an adaptation of the one in \cite[Section A]{Garcia20}. We proceed to prove (\ref{S5CD1}) by induction on $k$ with base case $k =1$ being obvious. Assuming the result for $k$ we proceed to prove it for $k+1$. Let us denote for convenience 
$$m(w,z) = \frac{1}{(w-z) (1 - wz)} \mbox{ and } g(w,z) = (w - z) (1 - wz).$$
By dividing each row of $\left( m(w_i, z_j) \right)_{1 \leq i, j \leq k+1}$ by the entry in the first column, and subtracting the first row from all other rows, we obtain 
$$\det \left[ m(w_i, z_j) \right]_{1 \leq i, j \leq k+1} = \prod_{i = 1}^{k+1} m(w_i, z_1) \cdot \det \left[ \frac{m(w_i, z_j)}{m(w_i, z_1)} - \frac{m(w_1, z_j)}{m(w_1,z_1)}\right]_{2 \leq i,j \leq k+1}.$$
We next note by a direct computation that  
$$ \frac{m(w_i, z_j)}{m(w_i, z_1)} - \frac{m(w_1, z_j)}{m(w_1,z_1)} = - g(z_1,z_j)g(w_1,w_i) m(w_1,z_j) m(w_i,z_j).$$
Combining the last two equalities and the multi-linearity of the determinant we get
\begin{equation*}
\begin{split}
\det \left[ m(w_i, z_j) \right]_{1 \leq i, j \leq k+1}= \hspace{2mm} &(-1)^k \prod_{i = 1}^{k+1} m(w_i, z_1)  \prod_{j = 2}^{k+1}  g(z_1,z_j)m(w_1,z_j)  \\
&\times \prod_{i = 2}^{k+1} g(w_1,w_i)  \cdot \det \left[ m(w_i,z_j) \right]_{2 \leq i,j \leq k+1} . 
\end{split}
\end{equation*}
Applying the induction hypothesis to the last determinant we arrive at (\ref{S5CD1}) for $k+1$, which completes the induction step.
\end{proof}

\begin{proof}[Proof of Lemma \ref{S4LemmaTaylor}]
We note that for any $a \neq 0$
$$\Re [ F(a \pm \i y) ] = \Re \frac{1}{1+e^a[ \cos(y) \pm \i \sin (y)] } + \frac{a}{4}- \frac{1}{2} = \frac{e^a \cos (y) + 1}{e^{2a} + 1 + 2 e^a \cos(y)} + \frac{a}{4}- \frac{1}{2}.$$
By direct computation, we conclude that 
$$\frac{d}{dy}\Re [F(a + \i y)]  = - \frac{e^a \sin(y) (e^{2a} - 1)}{(e^{2a} + 1 + 2e^a \cos(y))^2}.$$
The last equality for $a = x$ and $a = -x$ implies (\ref{S4Decay}). In the remainder we prove (\ref{S4TaylorF}).\\

We note that $F(z)$ is a meromorphic function on $\mathbb{C}$ with simple poles at $\pi \i + 2\pi \i \cdot \mathbb{Z}$, and so it is analytic in the zero-centered disc of radius $\pi$. If $F(z) = \sum_{n = 0}^\infty a_n z^n$ is the Taylor expansion of $F(z)$ at the origin, we directly compute $a_0 = a_1 = a_2 = 0$ and $a_3 = 1/48$. The latter suggests that $z^{-4} \cdot (F(z) - z^3/48)$ is an analytic function in the zero-centered disc of radius $\pi$, so that there is a constant $C_1 > 0$, such that for all $|z| \leq 3/2$ we have
$$\left| \frac{F(z) - z^3/48}{z^4} \right| \leq C_1.$$
This specifies our choice of $C_1$ and proves the first inequality in (\ref{S4TaylorF}) for all $\rr \leq 1/2$. 

Let $\rr \in (0,1/2)$ be sufficiently small so that $3 C_1 \cdot \rr \leq 101^{-3}$. We also let $C_2 = 10^6$. This specifies our choice of $\rr$ and $C_2$ and we proceed to prove the second and third line in (\ref{S4TaylorF}). As the proofs of the second and third inequality in (\ref{S4TaylorF}) are quite similar, we only establish the second. 

We fix $\rho \in \mathbb{R}, \epsilon \geq 0$, such that $0 \leq |\rho| \leq \epsilon \leq \rr$. We also fix $z\in \gamma_{\rho, \epsilon}^+$, such that $|\Im(z)| \leq \epsilon$, and note that $z = \rho + r e^{\i \phi}$, where $\phi \in \{-\pi/4, \pi/4\}$, while $0 \leq r \leq \sqrt{2} \epsilon$. In particular, we see that $|z| \leq |\rho| + r  \leq 3 \rr$, so that from the first inequality in (\ref{S4TaylorF}) we have
\begin{equation}\label{GR1}
\Re[ F(z)] \leq  \Re[z^3/48] + C_1 |z|^4 \leq \Re[z^3/48] + C_1 \cdot 3 \rr |z|^3 \leq \Re[z^3/48] + 101^{-3} |z|^3.
\end{equation}  
We also have from $\phi \in \{-\pi/4, \pi/4\}$ that 
\begin{equation}\label{GR2}
 \Re[z^3] = \Re[(\rho + r e^{\i \phi})^3] = \rho^3 - 2^{-3/2} \cdot r^3  + 3\cdot 2^{-1/2} \cdot  \rho^2 r .
\end{equation}  

If $r \leq 100 |\rho|$, then (\ref{GR2}) implies that
\begin{equation}\label{GR3}
  \Re[z^3/48] \leq 5 \cdot |\rho|^3 , \mbox{ and } |z|^3 \leq (101)^3 |\rho|^3.
\end{equation}  
In particular, from (\ref{GR1}) and (\ref{GR3}) we get
$$\Re[ F(z)] \leq \Re[z^3/48] + 101^{-3} |z|^3 \leq 6 \cdot  |\rho|^3  \leq C_2 |\rho|^3 - 200^{-1} \cdot (101)^3 |\rho|^3 \leq - \frac{|z|^3}{200} + C_2 |\rho|^3,$$
which proves the second line in (\ref{S4TaylorF}) if $r \leq 100 |\rho|$.

Finally, we suppose that $r \geq 100 |\rho|$. In this case, we have 
$$(101/100) \cdot r \geq r + |\rho| \geq |z| \geq r - |\rho| \geq (99/100) \cdot r. $$
The latter and (\ref{GR2}) imply
$$ \Re[z^3/48]  \leq \frac{r^3 }{48} \cdot (- 2^{-3/2} + 10^{-6} + 3 \cdot 2^{-1/2} \cdot 10^{-4}) \leq - \frac{r^3}{144} \leq -|z|^3 \cdot \left(200^{-1} + 101^{-3}\right) .$$
The last inequality and (\ref{GR1}) prove (\ref{S4TaylorF}) if $r \geq 100 |\rho|$.
\end{proof}

\begin{proof}[Proof of Lemma \ref{S42L1}] Let $C_2$ be as in Lemma \ref{S4LemmaTaylor}. We will prove that if $t \geq \epsilon^{-3}$ 
\begin{equation}\label{AR1}
\begin{split}
&\left| e^{t F(t^{-1/3}u) }\right| \leq \exp \left(C_2 - \frac{|u|^3 \epsilon^3}{800 \pi^3} \right) \mbox{ for $u \in \Gamma_{u,\epsilon, t}$, and  } \\
&\left| e^{-t F(t^{-1/3}v) }\right| \leq \exp \left(C_2 - \frac{|v|^3 \epsilon^3}{800 \pi^3} \right) \mbox{ for $v \in \Gamma_{v,\epsilon, t}$}.
\end{split}
\end{equation}
Notice that (\ref{AR1}) implies the statement of the lemma with $T_{\epsilon,1} = \epsilon^{-3}$, $A_{\epsilon,1} = \exp(2 C_2)$, and $a_{\epsilon,1} = \frac{\epsilon^3}{800 \pi^3}$. As the proofs of the two lines of (\ref{AR1}) are quite similar, we only establish the second. \\

Notice that if $|\Im (v)| \leq \epsilon t^{1/3}$ and $v \in \Gamma_{v,\epsilon, t}$, we have from the third line of (\ref{S4TaylorF}) with $\rho = -t^{-1/3}$
\begin{equation}\label{AR2}
\Re\left[-t F(t^{-1/3}v) \right] \leq - \frac{|v|^3}{200} + C_2.
\end{equation}
Suppose that $v \in \Gamma_{v,\epsilon, t}$ is such that $|\Im (v)| \geq \epsilon t^{1/3}$, and let $v_0 = -1 - \epsilon t^{1/3} \pm \i \epsilon t^{1/3}$, where the sign is chosen to agree with the imaginary part of $v$. Then, we observe that 
\begin{equation}\label{AR3}
\Re\left[-t F(t^{-1/3}v) \right] \leq \Re\left[-t F(t^{-1/3}v_0) \right] \leq - \frac{|v_0|^3}{200} + C_2 \leq - \frac{|v|^3 \epsilon^3}{800 \pi^3} + C_2,
\end{equation}
where in the first inequality we used the second inequality in (\ref{S4Decay}) with $x = t^{-1/3} + \epsilon$, in the second inequality we used (\ref{AR2}) with $v = v_0$, and in the last inequality we used that $|v_0| \geq \epsilon$, $|v| \leq \sqrt{2}\pi$. Equations (\ref{AR2}) and (\ref{AR3}) imply the second line of (\ref{AR1}), as $|e^z| = \exp ( \Re(z) )$ for all $z \in \mathbb{C}$.  
\end{proof}

\begin{proof}[Proof of Lemma \ref{S42L2}]
We set $T_{\epsilon, 2} = \epsilon^{-3}$ and note that for $t \geq T_{\epsilon, 2}$, $u \in \Gamma_{u,\epsilon, t}$, and $v \in \Gamma_{v,\epsilon, t}$ we have
$$-2\epsilon \leq t^{-1/3} \Re(v) \leq -t^{-1/3} \mbox{ and } \epsilon \geq t^{-1/3}\Re(u) \geq 0.$$
The latter implies that the conditions of \cite[Lemma 4.5]{ED18} are satisfied with $t = \tau$, $u = t^{-1/3}$, $U = 2\epsilon$ (here we used that $\epsilon \leq \rr  < 1/2$ as in Lemma \ref{S4LemmaTaylor}, and $e^{20\epsilon} \leq \tt^{-1}$). From the proof of \cite[Lemma 4.5]{ED18}, see the displayed equation after \cite[(6.15)]{ED18}, we conclude that 
$$\sum_{m \in \mathbb{Z}} \left| \frac{1}{\sin(-\pi [\log \tau ]^{-1}  [t^{-1/3} (v - u)  -  2m \pi  \i])} \right| \leq 2 [-\log \tau] t^{1/3} \sum_{k \geq 0} \tau^{2k \pi^2}.$$
The latter, the triangle inequality, and the fact that $|e^{\i x} | = 1$ for $x \in \mathbb{R}$, imply that the left side of (\ref{S42L2E}) is bounded by
$$2\pi \cdot \sum_{k \geq 0} \tau^{2k \pi^2} \leq 2\pi \cdot \sum_{k \geq 0} e^{-40 k \epsilon \pi^2},$$
which proves (\ref{S42L2E}) with $A_{\epsilon, 2} =  2\pi \cdot \sum_{k \geq 0} e^{-40 k \epsilon \pi^2}$ and $T_{\epsilon, 2} = \epsilon^{-3}$.
\end{proof}

\begin{proof}[Proof of Lemma \ref{S42L3}]
Let $K_{\epsilon} = \{z = x + \i y \in \mathbb{C}: |x| \leq 2\epsilon, |y| \leq \pi \mbox{ and } |z \pm \i \pi| \geq \epsilon\}$. We claim that we can find $a_{\epsilon,3} > 0$, such that for $z \in K_{\epsilon}$ we have
\begin{equation}\label{CR1}
\begin{split}
e^{-(a_{\epsilon,3}/2)|z|^2} \leq \left| \frac{e^{z/2} + e^{-z/2}}{2} \right| \leq e^{(a_{\epsilon,3}/2)|z|^2}.
\end{split}
\end{equation}
If (\ref{CR1}) holds, then (\ref{S42L3E}) would follow with the same choice of $a_{\epsilon,3}$ and $T_{\epsilon,3} = \epsilon^{-3}$, since $|\lfloor t^{2/3} \alpha \rfloor -1| \leq t^{2/3}|\alpha| + 2$ and for $t \geq T_{\epsilon,3}$ we have $t^{-1/3} \Gamma_{u,\epsilon, t} \subseteq K_{\epsilon}$, $t^{-1/3} \Gamma_{v,\epsilon, t} \subseteq K_{\epsilon}$. 

In the remainder we prove (\ref{CR1}). By Taylor expansion, we can find $\delta, A > 0$, such that for $|z| \leq \delta$ the function $\log \left( \frac{e^z + e^{-z}}{2} \right)$ is well-defined, and 
$$ \left| \log \left( \frac{e^z + e^{-z}}{2} \right) \right| \leq A|z|^2.$$
As always, the logarithm is with respect to the principal branch. This shows that for $|z| \leq \delta$
\begin{equation}\label{CR2}
\begin{split}
e^{-A|z|^2} \leq \left| \frac{e^{z/2} + e^{-z/2}}{2} \right| \leq e^{A|z|^2}.
\end{split}
\end{equation}
Since $\frac{e^{z/2} + e^{-z/2}}{2}$ is continuous and does not vanish on $K_{\epsilon}$ (note that the zeros are at $\pi \i  + 2\pi \i \mathbb{Z}$), we conclude that there is a constant $M > 0$ such that for $z \in K_{\epsilon}$ 
$$ e^{-M} \leq \left| \frac{e^{z/2} + e^{-z/2}}{2} \right| \leq e^M.$$
In particular, we see that if $z \in K_{\epsilon}$ and $|z| \geq \delta$ we have
\begin{equation}\label{CR3}
\begin{split}
e^{-M \delta^{-2}|z|^2} \leq e^{-M} \leq \left| \frac{e^{z/2} + e^{-z/2}}{2} \right| \leq e^M \leq e^{ M \delta^{-2} |z|^2}.
\end{split}
\end{equation}
Combining, (\ref{CR2}) and (\ref{CR3}) we conclude (\ref{CR1}) with $a_{\epsilon,3} = 2 \cdot \max( A, M\delta^{-2})$. 
\end{proof}

\begin{proof}[Proof of Lemma \ref{S42L4}]
We set $T_{\epsilon,4} = \epsilon^{-3}$ and proceed to estimate the various terms that appear in (\ref{S42L4E}) for $t \geq T_{\epsilon,4}$. Throughout we will use frequently that $\tau \leq \tau_1 \leq e^{-20\epsilon}$, which follows from our assumptions. 

We note that if $|z| = r \in [0, 1)$ and $\tau \in (0, \tau_1]$, then
\begin{equation}\label{DR1}
\prod_{m = 0}^\infty (1 - r \cdot \tau_1^m) \leq \prod_{m = 0}^\infty (1 - r \cdot \tau^m) \leq |(z;\tau)_{\infty}| \leq \prod_{m = 0}^\infty (1 + r \cdot \tau^m) \leq \prod_{m = 0}^\infty (1 + r \cdot \tau_1^m).
\end{equation}
Since $\frac{1 - e^{2z}}{z}$ is entire, we can find a constant $M_{1}$, such that if $|z| \leq 2\pi$ we have
$$\left| 1 - e^{2z} \right| \leq M_1 |z|.$$
Since for $t \geq T_{\epsilon,4}$ we have $|v t^{-1/3}| \leq 2\pi$ for all $v \in \Gamma_{v,\epsilon, t}$, we conclude that
\begin{equation}\label{DR2}
\begin{split}
&\left| t^{1/3}  (e^{2 t^{-1/3} v}  ;\tau)_{\infty}\right| = \left| t^{1/3}  (1- e^{2t^{-1/3}v}) \right| \cdot  \left| (\tau e^{2 t^{-1/3} v}  ;\tau)_{\infty} \right| \\
& \leq M_1 |v| \cdot \prod_{m = 1}^\infty (1 + \tau_1^m | e^{2m t^{-1/3} v}| ) \leq  M_1 |v| \cdot \prod_{m = 1}^\infty (1 + \tau_1^m) \leq M_1 |v| \cdot \prod_{m = 0}^\infty (1 + e^{-20 (m+1) \epsilon}),
\end{split}
\end{equation}
where in going to the second line we also used (\ref{DR1}), in the second inequality we used that $\Re(v) \leq 0$, and in the last one we used that $\tau_1 \leq e^{-20\epsilon}$.

Since $\Re(ut^{-1/3}) \leq \epsilon$ for $u \in \Gamma_{u,\epsilon, t}$, we see that  
\begin{equation}\label{DR3}
\begin{split}
&\left| (-e^{t^{-1/3} u} ;\tau)_{\infty} \right| \leq \prod_{m = 0}^\infty (1 + e^{\epsilon} \tau^m) \leq  \prod_{m = 0}^\infty (1 + e^{ \epsilon - 20 m \epsilon}),
\end{split}
\end{equation}
where again we used $\tau \leq \tau_1 \leq e^{-20\epsilon}$.

In addition, for $t \geq T_{\epsilon,4}$, $u \in \Gamma_{u,\epsilon, t}, v \in \Gamma_{v,\epsilon, t}$ we note that $t^{-1/3} \Re(u+v) \leq \epsilon$ and hence from (\ref{DR1}) we get
\begin{equation}\label{DR4}
\begin{split}
\left| (\tau e^{t^{-1/3} (u + v)} ;\tau)_{\infty} \right| \geq \prod_{ m = 0}^{\infty}(1 - \tau_1^{m} \tau |e^{t^{-1/3} (u + v)}| ) \geq  \prod_{ m = 0}^{\infty}(1 -   e^{-19\epsilon - 20m \epsilon } ),
\end{split}
\end{equation}
where again we used $\tau \leq \tau_1 \leq e^{-20\epsilon}$.

Finally, let $K_{\epsilon}$ be the compact set $K_{\epsilon} = \{z = x + \i y \in \mathbb{C}: |x| \leq 2\epsilon, |y| \leq \pi \mbox{ and } |z \pm \i \pi| \geq \epsilon\}$. Notice that $e^z + 1$ does not vanish on $K_{\epsilon}$ and so we can find $L_{\epsilon} > 0$ such that for $z \in K_{\epsilon}$ we have
$$|e^z + 1| \geq L_{\epsilon}$$
For $t \geq T_{\epsilon,4}$, and $ v \in \Gamma_{v,\epsilon, t}$ we have that $t^{-1/3}v \in K_{\epsilon}$ and so we get 
\begin{equation}\label{DR5}
\begin{split}
\left| (-e^{ t^{-1/3} v} ;\tau)_{\infty} \right| \geq \left|1 + e^{t^{-1/3}v} \right| \cdot \prod_{m = 0}^\infty \left(1 - \tau_1^m \tau \left| e^{ t^{-1/3} v} \right| \right) \geq L_{\epsilon} \cdot \prod_{m = 0}^\infty \left(1 - e^{ - 20(m+1) \epsilon }\right),
\end{split}
\end{equation}
where we used (\ref{DR1}), $\tau \leq \tau_1 \leq e^{-20\epsilon}$ and that $\Re(v) \leq 0$. Combining (\ref{DR2}), (\ref{DR3}), (\ref{DR4}) and (\ref{DR5}), we obtain (\ref{S42L4E}) with 
$$ A_{\epsilon,4} = \frac{1}{L_{\epsilon}} \prod_{m = 0}^{\infty} \frac{(1 + e^{-20 (m+1) \epsilon}) (1 + e^{ \epsilon - 20 m \epsilon}) }{(1 -   e^{-19\epsilon - 20m \epsilon } )(1 - e^{ - 20(m+1) \epsilon })}.$$
\end{proof}

\begin{proof}[Proof of Lemma \ref{S42L5}]
Let $T_{\epsilon,5} = \epsilon^{-3}$ and $\delta > 0$ be small enough so that for $x \in [0,1]$ we have
$$1 - e^{-x} \geq \delta x.$$
We note that if $t \geq T_{\epsilon,5}$ we have $t^{-1/3} \Re(u) \in [0, \epsilon]$ and $t^{-1/3} \Re(v) \in [-2\epsilon, - t^{-1/3}]$ for $u \in  \Gamma_{u,\epsilon, t}$ and $v \in \Gamma_{v,\epsilon, t}$. The latter implies that for $t \geq T_{\epsilon,5}$, $u \in  \Gamma_{u,\epsilon, t}$ and $v \in \Gamma_{v,\epsilon, t}$ we have
\begin{equation}\label{ER1}
\left|\frac{t^{-1/3}}{e^{t^{-1/3} u} - e^{t^{-1/3} v}}  \right| \leq \frac{t^{-1/3}}{|e^{t^{-1/3} u}| - |e^{t^{-1/3} v}|} \leq \frac{t^{-1/3}}{1 - e^{-t^{-1/3}}} \leq \delta^{-1}.
\end{equation}

Next, we let $K_{\epsilon}$ be the compact set $K_{\epsilon} = \{z = x + \i y \in \mathbb{C}: |x| \leq 2\epsilon , |y| \leq \pi\}$. We observe that $\frac{e^z - 1}{z}$ does not vanish on $K_{\epsilon}$, and so we can find $L_{\epsilon} > 0$, such that for all $z \in K_{\epsilon}$ 
$$\left| \frac{e^z - 1}{z} \right| \geq L_{\epsilon}.$$
Notice that for $t \geq T_{\epsilon,5}$, $u \in  \Gamma_{u,\epsilon, t}$ and $v \in\Gamma_{v,\epsilon, t}$ we have $t^{-1/3}(u+v)$, upto a shift by $2m \pi \i$ for some $m \in \{-1,0,1\}$, belongs to $K_{\epsilon} \cap \{z \in \mathbb{C}: |z| \geq t^{-1/3}\}$. The latter implies that for $t \geq T_{\epsilon,5}$, $u \in  \Gamma_{u,\epsilon, t}$ and $v \in \Gamma_{v,\epsilon, t}$ we have
\begin{equation}\label{ER2}
\left|\frac{t^{-1/3}}{1 - e^{t^{-1/3} (u +v)}}  \right| \leq \frac{t^{-1/3}}{L_{\epsilon} t^{-1/3}} = L^{-1}_{\epsilon}.
\end{equation}
Combining (\ref{ER1}), (\ref{ER2}) and Hadamard's inequality from Lemma \ref{DetBounds}, we get for $t \geq T_{\epsilon,5}$
$$ \left |\det \left[ \frac{t^{-2/3}}{(e^{t^{-1/3} u_a} -e^{t^{-1/3} v_b})(1 - e^{t^{-1/3} (u_a + v_b)})} \right]_{a,b = 1}^k  \right| \leq k^{k/2} (\delta L_{\epsilon})^{-k}, $$
which proves (\ref{S42L5E}) with $A_{\epsilon,5} = (\delta L_{\epsilon})^{-1}$. 
\end{proof}

\begin{proof}[Proof of Lemma \ref{S42L6}]
Let $T_{\epsilon,6} = \epsilon^{-3}$. Note that for $t \geq T_{\epsilon,6}$, $u \in \Gamma_{u,\epsilon, t}$, and $v \in \Gamma_{v,\epsilon, t}$ we have
$$-2\epsilon \leq t^{-1/3} \Re(v) \leq 0 \mbox{ and } \epsilon \geq t^{-1/3}\Re(u) \geq 0.$$
The latter inequalities with (\ref{DR1}) and $\tau \leq \tau_1 \leq e^{-20\epsilon}$ (which follows from the assumptions in the lemma) together imply for $t \geq T_{\epsilon,6}$, $u \in \Gamma_{u,\epsilon, t}$, and $v \in \Gamma_{v,\epsilon, t}$ that 
\begin{equation*}
\begin{split}
\left|\frac{(\tau e^{t^{-1/3}(u_a + u_b)}; \tau)_{\infty}(\tau e^{t^{-1/3}(v_a + v_b)}; \tau)_{\infty}}{(\tau e^{t^{-1/3}(u_a + v_b)}; \tau)_{\infty}(\tau e^{t^{-1/3}(v_a + u_b)}; \tau)_{\infty}} \right| \leq \hspace{-1mm} \prod_{m = 0}^{\infty} \hspace{-1mm} \frac{(1 + \tau^{m+1}e^{2\epsilon})(1 + \tau^{m+1}) }{(1 - \tau^{m+1} e^{\epsilon})^2} \leq \hspace{-1mm} \prod_{m = 0}^{\infty}\hspace{-1mm} \frac{(1 + e^{2\epsilon - 20(m+1) \epsilon})^2}{(1 - e^{ \epsilon - 20(m+1) \epsilon})^2}.
\end{split}
\end{equation*}
The latter proves (\ref{S42L6E}) with $A_{\epsilon,6} = \prod_{m = 0}^{\infty} \frac{1 + e^{2\epsilon - 20(m+1) \epsilon}}{1 - e^{\epsilon - 20(m+1) \epsilon}}$.
\end{proof}

%-------------------------------------------------------------------------------------------------------------------------------------------------------------------------------------------------
% Section 5
%
%-------------------------------------------------------------------------------------------------------------------------------------------------------------------------------------------------

\section{Asymptotic analysis: Part II}\label{Section5} In this section we present the proof of Proposition \ref{S3MainP2}. In Section \ref{Section5.1} we present several lemmas, which will be required for our arguments, and whose proofs are given in Section \ref{Section5.4}. The proof of Proposition \ref{S3MainP2} is established by considering the cases $1 \leq k \leq t$ and $t \leq k$ separately in Sections \ref{Section5.2} and \ref{Section5.3}, respectively.

%-------------------------------------------------------------------------------------------------------------------------------------------------------------------------------------------------
% Section 5.1
%
%-------------------------------------------------------------------------------------------------------------------------------------------------------------------------------------------------
\subsection{Preliminary results}\label{Section5.1} In this section we summarize various estimates that will be required in the proof of Proposition \ref{S3MainP2}. As we explained earlier, the estimate of $I(k,t)$ in Proposition \ref{S3MainP2} is established by considering the cases $1 \leq k \leq t$ and $t \geq k$ separately. In the case $1 \leq k \leq t$, we will use the formula for $I(k,t)$ from (\ref{S4Ikt}) and to upper bound it we will use the estimates in Lemmas \ref{S42L1}-\ref{S42L5}. In addition, we will use Lemma \ref{S5CTBound} below in place of Lemma \ref{S42L6} to upper bound the term $B_2(\vec{u}, \vec{v}; t)$ from (\ref{S41Bs}). For the case $k \geq t$, we use a different formula for $I(k,t)$, established in (\ref{S6Ikt}). This formula involves a different set of contours and functions $\tilde{A}_i(\vec{w}, \vec{z};t)$ and $\tilde{B}_i(\vec{w}, \vec{z};t)$ and we estimate these new functions in Lemmas \ref{S61L1}-\ref{S61L5}. The proofs of all lemmas from this section can be found in Section \ref{Section5.4}.

\begin{lemma}\label{S5CTBound} There is a universal constant $\cc > 0$ such that the following holds for  $\tau \in (0, e^{-8\pi}]$. If $k \in \mathbb{N}$ and $W_i, Z_i \in \mathbb{C}$ with $|W_i| \leq 2\pi, |Z_i| \leq 2\pi$ for $i = 1, \dots,k$, then
\begin{equation}\label{CTUB1}
\begin{split}
&\left| \prod_{1 \leq a < b \leq k } \frac{(\tau e^{W_a + W_b}; \tau)_{\infty}(\tau e^{Z_a + Z_b}; \tau)_{\infty}}{(\tau e^{W_a+ Z_b}; \tau)_{\infty}(\tau e^{Z_a + W_b}; \tau)_{\infty}} \right| \leq \exp \left( - A(\tau)\cdot \Re \left[\left( \sum_{a = 1}^k (W_a  - Z_a)  \right)^2 \right] \right)  \cdot \\
& \exp \left( A(\tau) \cdot \Re \left[ \sum_{a = 1}^k (W_a - Z_a)^2 \right] \right) \cdot \prod_{a = 1}^k \exp \left(\cc \tau k  \cdot \sum_{a = 1}^k (|Z_a|^3 + |W_a|^3) \right),
\end{split}
\end{equation}
where
\begin{equation}\label{AtauDef}
A(\tau) := \frac{1}{2} \sum_{n = 1}^{\infty} \sum_{k = 1}^{\infty} k \tau^{nk} = \frac{1}{2}\sum_{k = 1}^{\infty} \frac{k \tau^{k}}{1 - \tau^{k}} \in (0, \infty). 
\end{equation}
\end{lemma}

\begin{definition}\label{S6Cont} Let $\tau \in (0, e^{-8\pi}]$, and let $C_w, C_z$ be the positively oriented circles, centered at the origin, of radii $R_w = e$ and $R_z = \tau^{3/4}$. We observe that the choice of $R_w, R_z$ satisfies $R_w > 1 > R_w^{-1} > R_z > \tau R_w$, since $\tau \in (0, e^{-8\pi}]$.
\end{definition}

Using the definition of $I(k,t)$ from Definition \ref{DefScale} and Lemma \ref{S3LRewrite} we conclude that
\begin{equation}\label{S6Ikt}
\begin{split}
I(k,t)= \frac{1}{k! (2\pi \i)^{2k}} \oint_{C_w^k } d\vec{w}  \oint_{C_z^k } d \vec{z} & \prod_{i = 1}^4 \tilde{A}_i ( \vec{w}, \vec{z}; t) \prod_{i = 1}^2 \tilde{B}_i( \vec{w}, \vec{z}; t),
\end{split}
\end{equation}
where $C_w, C_z$ are as in Definition \ref{S6Cont}, and 
\begin{equation}\label{S61Bs}
\begin{split}
&\tilde{B}_1(\vec{w}, \vec{z}; t) = \det \left[ \frac{1}{w_a - z_b} \right]_{a,b = 1}^k, \hspace{2mm} \tilde{B}_2(\vec{w}, \vec{z}; t)  = \prod_{1 \leq a < b \leq k } \frac{(w_aw_b;\tau)_\infty (z_az_b;\tau)_\infty }{(z_aw_b;\tau)_\infty (w_az_b;\tau)_\infty}.
\end{split}
\end{equation}
\begin{equation}\label{S61As}
\begin{split}
&\tilde{A}_1 (\vec{w}, \vec{z}; t) = \prod_{a = 1}^k \exp \left( t \left[\frac{1}{1 + w_a} + \frac{\log w_a}{4} - \frac{1}{1+z_a} - \frac{\log z_a}{4} \right] \right), \\
&\tilde{A}_2(\vec{w}, \vec{z}; t) = \prod_{a = 1}^k \left( \sum_{m \in \mathbb{Z}} \frac{ [-\log \tau]^{-1} \pi \cdot e^{ 2m \pi  \i  [(1/4)t +(1/2)( \lfloor t^{2/3}\alpha \rfloor - 1 ) -  t^{1/3} \tilde{r} ]}}{\sin(-\pi  [\log \tau ]^{-1} [\log z_a - \log w_a   -  2m \pi  \i  ])} \right), \\
& \tilde{A}_3 (\vec{w}, \vec{z}; t)  = \prod_{a = 1}^k \left( \frac{(1 + z_a )e^{-\log z_a /2} }{(1 + w_a) e^{-\log w_a/2} } \right)^{\lfloor t^{2/3} \alpha \rfloor -1},  \\
& \tilde{A}_4 (\vec{w}, \vec{z}; t) =     \prod_{a = 1}^k \frac{(-w_a ;\tau)_{\infty} (z_a^2  ;\tau)_{\infty} }{(-z_a ;\tau)_{\infty} (z_a w_a  ;\tau)_{\infty} \cdot z_a } \cdot \exp \left( t^{1/3} \tilde{r} [ \log z_a  -  \log w_a] \right),
\end{split}
\end{equation}
In equations (\ref{S61As}), and (\ref{S61Bs}) we take the principal branch of the logarithm everywhere.

In the remainder of this section, we estimate the functions that appear in $\tilde{A}_i (\vec{w}, \vec{z}; t)$, $\tilde{B}_i(\vec{w}, \vec{z}; t)$ over the contours $C_w, C_z$ in a sequence of lemmas.
\begin{lemma}\label{S61L1} Let $\tau, C_w, C_z$ be as in Definition \ref{S6Cont}. For $w \in C_w$, and $z \in C_z$ we have
\begin{equation}\label{S61E1}
\left| \exp \left(\frac{1}{1 + w} + \frac{\log w}{4} - \frac{1}{1+z} - \frac{\log z}{4} \right) \right| \leq \tau^{-1/4}.
\end{equation}
\end{lemma}

\begin{lemma}\label{S61L2} Let $\tau, C_w, C_z$ be as in Definition \ref{S6Cont}. There exists a constant $B_{\tau,2} > 0$, depending on $\tau$ alone, such that for all $t > 0$, $\alpha, \tilde{r} \in \mathbb{R}$, $w \in C_w$, and $z \in C_z$ we have
\begin{equation}\label{S61E2}
\left| \sum_{m \in \mathbb{Z}} \frac{ [-\log \tau]^{-1} \pi \cdot e^{ 2m \pi  \i  [(1/4)t +(1/2)( \lfloor t^{2/3}\alpha \rfloor - 1 ) -  t^{1/3} \tilde{r} ]}}{\sin(-\pi  [\log \tau ]^{-1} [ \log z - \log w  -  2m \pi  \i  ])} \right| \leq B_{\tau,2}.
\end{equation}
\end{lemma}

\begin{lemma}\label{S61L3} Let $\tau, C_w, C_z$ be as in Definition \ref{S6Cont}. There exists a constant $b_{\tau,3} > 0$, depending on $\tau$ alone, such that for all $t \geq 1$, $\alpha \in \mathbb{R}$, $w \in C_w$, and $z \in C_z$ we have
\begin{equation}\label{S61E3}
\left| \left( \frac{(1 + z )e^{-\log z /2} }{(1 + w) e^{-\log w/2} } \right)^{\lfloor t^{2/3} \alpha \rfloor -1} \right| \leq \exp \left( b_{\tau,3}  \cdot  t^{2/3} \cdot (1 + |\alpha|) \right).
\end{equation}
\end{lemma}

\begin{lemma}\label{S61L4} Let $\tau, C_w, C_z$ be as in Definition \ref{S6Cont}. There exists a constant $b_{\tau,4} > 0$, depending on $\tau$ alone, such that for all $t \geq 1$, $ \tilde{r} \in \mathbb{R}$, $w \in C_w$, and $z \in C_z$ we have
\begin{equation}\label{S61E4}
\left|\frac{(-w ;\tau)_{\infty} (z^2  ;\tau)_{\infty} }{(-z ;\tau)_{\infty} (z w  ;\tau)_{\infty} \cdot z } \cdot \exp \left( t^{1/3} \tilde{r} [ \log z  -  \log w] \right)  \right| \leq \exp \left( b_{\tau,4}  \cdot  t^{1/3} \cdot (1 + |\tilde{r}|) \right).
\end{equation}
\end{lemma}

\begin{lemma}\label{S61L5} Let $\tau, C_w, C_z$ be as in Definition \ref{S6Cont}, and fix $k \in \mathbb{N}$. Suppose that $t > 0$, $w_a \in C_w$, $z_a \in C_z$ for $a =1 , \dots, k$. Then, we have
\begin{equation}\label{S61E5}
\left| \tilde{B}_1(\vec{w}, \vec{z}; t)  \right| \leq k^{k/2} \cdot \tau^{-3k/8} \cdot [ 2 \tau^{3/8}]^{k^2}.
\end{equation}
\end{lemma}

\begin{lemma}\label{S61L6} Let $\tau, C_w, C_z$ be as in Definition \ref{S6Cont}. We can find a universal constant $B_6 > 0$, such that for $t > 0$, $k \in \mathbb{N}$, $w_a \in C_w$, $z_a \in C_z$ for $a =1 , \dots, k$
\begin{equation}\label{S61E6}
\left| \tilde{B}_2(\vec{w}, \vec{z}; t)  \right| \leq B_6^{k^2}.
\end{equation}
\end{lemma}

%-------------------------------------------------------------------------------------------------------------------------------------------------------------------------------------------------
% Section 5.2
%
%-------------------------------------------------------------------------------------------------------------------------------------------------------------------------------------------------
\subsection{Proof of Proposition \ref{S3MainP2}: Part I}\label{Section5.2} In this section we specify the choice of $\tau_0$, $A$ and $T$ as in the statement of Proposition \ref{S3MainP2} and prove (\ref{S3KE2}) when $1 \leq k \leq t$. 

Let $\rr$ be as in Lemma \ref{S4LemmaTaylor} and $\cc$ as in Lemma \ref{S5CTBound}.  We assume that $\tau_0 \in (0,1)$ is sufficiently small so that the following all hold for $\tau \in (0, \tau_0]$
\begin{equation}\label{S5IneqTau}
\begin{split}
\tau \leq e^{-8 \pi}, \hspace{2mm} \cc \tau \leq a_{\rr,1}/2, \hspace{2mm} 64 \pi A(\tau) \leq \rr^2 a_{\rr,1}, \hspace{2mm} 4 \tau^{1/8} B_6 \leq 1,
\end{split}
\end{equation}
where $a_{\rr,1}$ as in Lemma \ref{S42L1}, $A(\cdot)$ is as in (\ref{AtauDef}), and $B_6$ is as in Lemma \ref{S61L6}. This specifies $\tau_0$ in the statement of the proposition. Below we fix $\tau \in (0, \tau_0]$.

We let $A \geq 1$ be sufficiently large so that for $t \geq 1$ we have
\begin{equation}\label{S5IneqA}
\begin{split}
&A_{\rr,1} A_{\rr,2} A_{\rr,4} A_{\rr,5} \int_{\Gamma_{u,\rr, t}}  \exp \left(  (a_{\rr,3}+1) |u|^2 + |\tilde{r}| |u| - (a_{\rr,1}/4) |u|^3  \right) |du|  \\
& \times \int_{\Gamma_{v,\rr, t}} |v| \exp \left(  (a_{\rr,3}+1) |v|^2 + |\tilde{r}| |v| - (a_{\rr,1}/4) |v|^3  \right) |dv|  \leq A,
\end{split}
\end{equation}
where $\rr$ is as in Lemma \ref{S4LemmaTaylor}, $\Gamma_{u,\rr, t}$ and $\Gamma_{v,\rr, t}$ are as in Definition \ref{S4Contours}, $A_{\rr,1}, a_{\rr,1}$ is as in Lemma \ref{S42L1}, $A_{\rr,2}$ is as in Lemma \ref{S42L2}, $a_{\rr,3}$ is as in Lemma \ref{S42L3}, $A_{\rr,4}$ is as in Lemma \ref{S42L4}, $A_{\rr,5}$ is as in Lemma \ref{S42L5}, and $\tilde{r}$ is as in the statement of the proposition.  In (\ref{S5IneqA}) the integrals are with respect to arc-length. This specifies $A$ in the statement of the proposition.

We now fix $T \geq 1$ sufficiently large so that for $t \geq T$ we have
\begin{equation}\label{S5IneqT}
\begin{split}
& t \geq \max_{1 \leq i \leq 5} T_{\rr, i}, \hspace{2mm}  t^{-1/3} \leq \rr, \hspace{2mm} t^{-2/3} A(e^{-8\pi}) \leq 1/2, \mbox{ and for $k \geq T$ we have } \\
&  \frac{k^{k/2}}{k!} \cdot e^{k} \cdot 2^{-k^2} \cdot B_{\tau,2}^k \cdot \exp \left( k^{5/3} b_{\tau,3} (1 + |\alpha|) +  k^{4/3} b_{\tau,4} (1 + |\tilde{r}|)  \right) \leq k^{-k/2},
\end{split}
\end{equation}
where $T_{\rr,i}$ for $i = 1,\dots, 5$ are as in Lemmas \ref{S42L1} - \ref{S42L5} (here $\tau_1 = e^{-8\pi}$ and so $e^{20 \rr} \leq e^{10} \leq \tau_1^{-1}$). In addition, $A(\cdot)$ is as in (\ref{AtauDef}), $B_{\tau,2}$ is as in Lemma \ref{S61L2}, $b_{\tau,3}$ is as in Lemma \ref{S61L3}, $b_{\tau,4} $ is as in Lemma \ref{S61L4}, $\alpha, \tilde{r}$ are as in the statement of the proposition. This specifies $T$ in the statement of the proposition. \\

In the remainder of this section we assume $t \geq T$, $\tau \in (0, \tau_0]$, $1 \leq k \leq t$ and proceed to prove (\ref{S3KE2}). It follows from (\ref{S4Ikt}) that 
\begin{equation}\label{S5Ikt1}
\begin{split}
\left| I(k,t) \right| \leq  \frac{1}{k! (2\pi )^{2k}} \int_{\Gamma^k_{u,\rr, t} } |d\vec{u}|  \int_{\Gamma^k_{v,\rr, t} } |d \vec{v}| & \left|\prod_{i = 1}^5 A_i ( \vec{u}, \vec{v}; t) \prod_{i = 1}^2 B_i( \vec{u}, \vec{v}; t) \right|,
\end{split}
\end{equation}
where $A_i ( \vec{u}, \vec{v}; t) $ for $i = 1,\dots, 5$ are as in (\ref{S41As}), $B_j( \vec{u}, \vec{v}; t)$ for $j = 1,2$ are as in (\ref{S41Bs}), $|d\vec{u}| = |du_1| \cdots |du_k|$, and $|d\vec{v}| = |dv_1|\cdots |dv_k|$. We mention that in obtaining (\ref{S5Ikt1}) we implicitly used that $t^{-1/3} \leq \rr$, as follows from the second inequality in (\ref{S5IneqT}), and that $e^{20 \rr} \leq e^{10} \leq \tau^{-1}$ as follows from the first inequality in (\ref{S5IneqTau}).

It follows from Lemmas \ref{S42L1}-\ref{S42L5} that if $\vec{u} \in \Gamma^k_{u,\rr,t}$ and $\vec{v} \in \Gamma^k_{v,\rr,t}$ we have 
\begin{equation*}
\begin{split}
&\left| \prod_{i=1}^5 A_i(\vec{u}, \vec{v}; t) \cdot B_1(\vec{u}, \vec{v}; t) \right| \leq A_{\rr,1}^k A_{\rr,2}^k A_{\rr,4}^k A_{\rr,5}^k \cdot k^{k/2} \cdot \prod_{a = 1}^k |v_a|  \\
& \times \exp \left( \sum_{a = 1}^k a_{\rr,3} |u_a|^2 + |\tilde{r}| |u_a| - a_{\rr,1} |u_a|^3 + a_{\rr,3} |v_a|^2 + |\tilde{r}| |v_a| - a_{\rr,1} |v_a|^3   \right).
\end{split}
\end{equation*}
In addition, from Lemma \ref{S5CTBound} we have 
\begin{equation*}
\begin{split}
&\left| B_2(\vec{u}, \vec{v}; t) \right| \leq \exp \left( - A(\tau) t^{-2/3}\cdot \Re \left[\left( \sum_{a = 1}^k (u_a  - v_a)  \right)^2 \right] \right)  \cdot \\
& \exp \left( A(\tau) t^{-2/3} \cdot \Re \left[ \sum_{a = 1}^k (u_a - v_a)^2 \right] \right) \cdot \prod_{a = 1}^k \exp \left(\cc \tau k t^{-1}  \cdot \sum_{a = 1}^k (|u_a|^3 + |v_a|^3) \right).
\end{split}
\end{equation*}
We next observe that our assumptions give the inequalities 
$$A(\tau) t^{-2/3} \leq 1/2, \hspace{2mm}\cc \tau \leq a_{\rr,1}/2, \hspace{2mm} k \leq t, \hspace{2mm} |\Re[z_1-z_2]^2| \leq 2(|z_1|^2 + |z_2|^2),$$ 
where the first inequality used the first inequality in (\ref{S5IneqTau}) and the third inequality in (\ref{S5IneqT}). 

Combining the last three inequalities, we get
\begin{equation}\label{S5S3E1}
\begin{split}
& \left| \prod_{i=1}^5 A_i(\vec{u}, \vec{v}; t)  \prod_{i=1}^2 B_2(\vec{u}, \vec{v}; t) \right| \leq  A_{\rr,1}^k A_{\rr,2}^k A_{\rr,4}^k A_{\rr,5}^k \cdot k^{k/2} \cdot \prod_{a = 1}^k |v_a| \\
& \times \exp \left( \sum_{a = 1}^k (a_{\rr,3}+1) (|u_a|^2 + |v_a|^2) + |\tilde{r}| ( |u_a|  + |v_a|) - (a_{\rr,1}/2) (|u_a|^3 + |v_a|^3)   \right) \\
& \times \exp \left( - A(\tau) t^{-2/3}\cdot \Re \left[\left( \sum_{a = 1}^k (u_a  - v_a)  \right)^2 \right] \right).
\end{split}
\end{equation}

We next claim that 
\begin{equation}\label{S5S3E2}
\begin{split}
- A(\tau) t^{-2/3}\cdot \Re \left[\left( \sum_{a = 1}^k (u_a  - v_a)  \right)^2 \right]   \leq \frac{a_{\rr,1}}{4} \sum_{a = 1}^k (|u_a|^3 + |v_a|^3) .
\end{split}
\end{equation}
To see why (\ref{S5S3E2}) holds let us denote by $S$ the set indices $s$ in $\{1, \dots, k\}$ such that $ u_s \in \Gamma_{u, \rr, t}^0$ and $ v_s \in \Gamma_{v, \rr, t}^0$, where we recall that $\Gamma_{u, \rr, t}^0, \Gamma_{v, \rr, t}^0$ were defined in Definition \ref{S4Contours}. We also set $S^c = \{1, \dots, k \} \setminus S$. Using that $|t^{-1/3} u_a| \leq 2\pi$ and $|t^{-1/3} v_a| \leq 2\pi$, we see that 
\begin{equation*}
\begin{split}
- t^{-2/3}\cdot \Re \left[\left( \sum_{a = 1}^k (u_a  - v_a)  \right)^2 \right] \leq - t^{-2/3}\cdot \Re \left[\left( \sum_{a \in S} (u_a  - v_a)  \right)^2 \right] + 8 k \pi  t^{-1/3} \sum_{a \in S^c} \left( |u_a| + |v_a| \right).
\end{split}
\end{equation*}
We also have for $a \in S^c$ that $|u_a|  \geq \rr \cdot t^{1/3}$ or $|v_a|  \geq \rr \cdot t^{1/3}$ (or both), which implies that for $a \in S^c$
\begin{equation*}
\begin{split}
  8 A(\tau)  k \pi  t^{-1/3} (|u_a| + |v_a|) \leq 8 A(\tau) \pi  t^{2/3} (|u_a| + |v_a|) \leq \frac{a_{1,\rr}}{4} (|u_a|^3 + |v_a|^3). 
\end{split}
\end{equation*}
We remark that in the first inequality we used that $k \leq t$ and in the second one we used the third inequality in (\ref{S5IneqTau}) and the fact that $2(x^3 + y^3) \geq  (x+y)(x^2 + y^2)$ for $x,y \geq 0$.

Finally, we note that for $a \in S$ we have $u_a = x_a + \i \epsilon_a x_a$ and $v_a = - 1 - y_a + \i \delta_a y_a$ where $x_a, y_a \geq 0$ and $\epsilon_a, \delta_a \in \{-1, 1\}$. The latter implies that 
$$\Re \left[\left( \sum_{a \in S} (u_a  - v_a)  \right)^2 \right]  =\left( \sum_{a \in S} \left(x_a + y_a + 1 \right) \right)^2 - \left( \sum_{a \in S} \left(\epsilon_a x_a - \delta_a y_a \right)  \right)^2 \geq 0.$$
Combining the last three inequalities, we deduce (\ref{S5S3E2}).\\

We now combine (\ref{S5S3E1}) and (\ref{S5S3E2}) to conclude that
\begin{equation}\label{S5S3E3}
\begin{split}
& \left| \prod_{i=1}^5 A_i(\vec{u}, \vec{v}; t)  \prod_{i=1}^2 B_2(\vec{u}, \vec{v}; t) \right| \leq   A_{\rr,1}^k A_{\rr,2}^k A_{\rr,4}^k A_{\rr,5}^k \cdot k^{k/2} \cdot \prod_{a = 1}^k |v_a| \\
& \times \exp \left( \sum_{a = 1}^k (a_{\rr,3}+1) (|u_a|^2 + |v_a|^2) + |\tilde{r}| ( |u_a|  + |v_a|) - (a_{\rr,1}/4) (|u_a|^3 + |v_a|^3)   \right).
\end{split}
\end{equation}
From (\ref{S5IneqA}), (\ref{S5Ikt1}) and (\ref{S5S3E3}) we conclude that 
$$\left| I(k,t) \right| \leq \frac{k^{k/2}}{k! (2\pi)^{2k}} \cdot A^k \leq \frac{k^{k/2}}{k^{k} e^{-k} (2\pi)^{2k}} \cdot A^k \leq k^{-k/2} \cdot A^k,$$
where in the middle inequality we used that $k! \geq k^{k} e^{-k}$, see (\ref{S2Rob}). The last inequality proves (\ref{S3KE2}) when $1 \leq k \leq t$.

%-------------------------------------------------------------------------------------------------------------------------------------------------------------------------------------------------
% Section 5.3
%
%-------------------------------------------------------------------------------------------------------------------------------------------------------------------------------------------------
\subsection{Proof of Proposition \ref{S3MainP2}: Part II}\label{Section5.3} In this section we assume $ k \geq t \geq T$, $\tau \in (0, \tau_0]$, and proceed to prove (\ref{S3KE2}). It follows from (\ref{S6Ikt}) that 
\begin{equation}\label{S5Ikt2}
\begin{split}
\left| I(k,t) \right| \leq  \frac{1}{k! (2\pi )^{2k}} \int_{C_w^k } |d\vec{w}|  \int_{C_z^k } |d \vec{z}| & \left| \prod_{i = 1}^4 \tilde{A}_i ( \vec{w}, \vec{z}; t) \prod_{i = 1}^2 \tilde{B}_i( \vec{w}, \vec{z}; t) \right|,
\end{split}
\end{equation}
where $\tilde{A}_i ( \vec{w}, \vec{z}; t) $ for $i = 1,\dots, 4$ are as in (\ref{S61As}), $\tilde{B}_j( \vec{w}, \vec{z}; t)$ for $j = 1,2$ are as in (\ref{S61Bs}), $|d\vec{w}| = |dw_1| \cdots |dw_k|$, and $|d\vec{z}| = |dz_1|\cdots |dz_k|$ and $|dz_i|$, $|dw_i|$ denote integration with respect to arc-length. We mention that in obtaining (\ref{S5Ikt2}) we implicitly used that $\tau \leq e^{-8\pi}$ as follows from the first inequality in (\ref{S5IneqTau}).

From Lemmas \ref{S61L1}-\ref{S61L6}, and the fact that $k \geq t$, we have for $\vec{w} \in C_w^k$ and $\vec{z} \in C_z^k$ that 
\begin{equation*}
\begin{split}
 \left| \prod_{i = 1}^4 \tilde{A}_i ( \vec{w}, \vec{z}; t) \prod_{i = 1}^2 \tilde{B}_i( \vec{w}, \vec{z}; t) \right| \leq \hspace{2mm} & \tau^{-k^2/4} \cdot B_{\tau,2}^k \cdot \exp \left( k^{5/3} b_{\tau,3} (1 + |\alpha|) +  k^{4/3} b_{\tau,4} (1 + |\tilde{r}|)  \right)\\
& \times k^{k/2} \cdot \tau^{-3k/8} \cdot [ 2 \tau^{3/8}]^{k^2} \cdot B_6^{k^2}.
\end{split}
\end{equation*}
We mention that in applying Lemmas \ref{S61L1}-\ref{S61L6} we used that $0 < \tau \leq e^{-8\pi}$, and $t \geq 1$. Combining the latter with (\ref{S5Ikt2}), the fourth inequality in (\ref{S5IneqTau}), and the fact that the length of $C_w$ is $2\pi e$ and that of $C_z$ is $2\pi \tau^{3/8}$, we conclude that
\begin{equation}\label{NP1}
\begin{split}
& \left| I(k,t) \right| \leq  \frac{k^{k/2}}{k!} \cdot e^{k} \cdot 2^{-k^2} \cdot B_{\tau,2}^k \cdot \exp \left( k^{5/3} b_{\tau,3} (1 + |\alpha|) +  k^{4/3} b_{\tau,4} (1 + |\tilde{r}|)  \right).
\end{split}
\end{equation}
In view of (\ref{NP1}) and the second line in (\ref{S5IneqT}), we conclude that $\left| I(k,t) \right|  \leq k^{-k/2} \leq A^k k^{-k/2}$ as $A \geq 1$ by construction. This proves (\ref{S3KE2}) when $k \geq t$.

%-------------------------------------------------------------------------------------------------------------------------------------------------------------------------------------------------
% Section 5.4
%
%-------------------------------------------------------------------------------------------------------------------------------------------------------------------------------------------------
\subsection{Proof of the lemmas from Section \ref{Section5.1}}\label{Section5.4} In this section we present the proofs of the seven lemmas from Section \ref{Section5.1}.\\

\begin{proof}[Proof of Lemma \ref{S5CTBound}]
We first note that if $\lambda \in [0, e^{-8\pi}]$ and $|z| \leq 4\pi$ then 
$$\log \left(1 - \lambda e^{z} \right) = - \sum_{k = 1}^{\infty} \frac{\lambda^k e^{kz}}{k} = - \sum_{k = 1}^{\infty}  \frac{\lambda^k }{k} \sum_{m = 0}^{\infty} \frac{z^m k^m}{m!} = -\sum_{m = 0}^{\infty}\frac{z^m}{m!} \sum_{k = 1}^{\infty} k^{m-1}\lambda^k, $$
where in the last line the order of the sums can be exchanged by Fubini's theorem as
$$\sum_{k = 1}^{\infty}  \sum_{m = 0}^{\infty}  \frac{\lambda^k }{k} \cdot \frac{|z|^m k^m}{m!} \leq \sum_{k = 1}^{\infty}   \frac{\lambda^k e^{k |z|}}{k} < \infty.$$
In particular, if we set 
$$A_m(\lambda) =  \frac{1}{m!}\sum_{k = 1}^{\infty} k^{m-1}\lambda^{k-1} \mbox{ and } B_m = \frac{1}{m!}\sum_{k = 1}^{\infty} k^{m-1}e^{-8\pi (k-1) }$$
we see that for $\lambda \in [0, e^{-8\pi}]$ and $|z| \leq 4\pi$ we have
\begin{equation}\label{S5L2E1}
0 \leq  A_m(\lambda) \leq  B_m, \hspace{2mm} B := \sum_{m = 0}^{\infty} B_m \cdot (4\pi )^m < \infty \mbox{, and } \log \left(1 - \lambda e^{z} \right) =- \lambda \cdot \sum_{m = 0}^{\infty} A_{m}(\lambda) z^m,
\end{equation}
and the latter is an absolutely convergent in the closed disc of radius $4\pi$. 

We next observe that if $z \in \{ W_a + W_b, Z_a + Z_b, W_a + Z_b, Z_a + W_b\}$ we have $|z| \leq 4\pi$ and so from (\ref{S5L2E1}) we conclude for any $\lambda \in [0, e^{-8\pi}]$ that 
\begin{equation*}
\begin{split}
&\left| \frac{(1 - \lambda e^{W_a + W_b})(1 - \lambda e^{Z_a + Z_b})   }{(1 - \lambda e^{W_a+ Z_b})(1 - \lambda e^{Z_a + W_b})} \right| \\
& = \exp \Big ( - \lambda \sum_{m = 2}^{\infty} A_m(\lambda) \Re \left[ (W_a+ W_b)^m + (Z_a + Z_b)^m - (W_a + Z_b)^m - (Z_a+W_b)^m \right] \Big).
\end{split}
\end{equation*}
In particular, the last equation and (\ref{S5L2E1}) imply
\begin{equation}\label{S5L2E2}
\begin{split}
\left| \frac{(1 - \lambda e^{W_a + W_b})(1 - \lambda e^{Z_a + Z_b})   }{(1 - \lambda e^{W_a+ Z_b})(1 - \lambda e^{Z_a + W_b})} \right| \leq \hspace{2mm} &\exp \left( - 2\lambda A_2(\lambda) \cdot \Re [ (W_a - Z_a)(W_b - Z_b) ]  \right)  \\
&\times \exp \left(  \lambda B \cdot (|W_a|^3 + |W_b|^3 + |Z_a|^3 + |Z_b|^3) \right).
\end{split}
\end{equation}

Taking a product of (\ref{S5L2E2}) over $\lambda = \tau, \tau^2, \cdots$ as well as $1 \leq a < b \leq k$ we see that 
\begin{equation*}
\begin{split}
&\left| \prod_{1 \leq a < b \leq k } \frac{(\tau e^{W_a + W_b}; \tau)_{\infty}(\tau e^{Z_a + Z_b}; \tau)_{\infty}}{(\tau e^{W_a+ Z_b}; \tau)_{\infty}(\tau e^{Z_a + W_b}; \tau)_{\infty}} \right| \leq \exp \left( -2 A(\tau) \cdot \Re \left[ \sum_{1 \leq a < b \leq k} (W_a - Z_a)(W_b - Z_b)  \right]  \right) \\
& \times \exp \left(  \frac{k \cdot \tau B}{1- \tau} \cdot \left( \sum_{a = 1}^k |W_a|^3 + |Z_a|^3 \right) \right),
\end{split}
\end{equation*}
where $A(\tau)$ is as in (\ref{AtauDef}). The last inequality implies (\ref{CTUB1}) with $\cc = B/ (1 - e^{-8\pi})$ once we use
$$2\sum_{1 \leq a < b \leq k} (W_a - Z_a)(W_b - Z_b)  +  \sum_{a = 1}^k (W_a-Z_a)^2 =  \left( \sum_{a = 1}^k (W_a  - Z_a)  \right)^2.$$
\end{proof}

\begin{proof}[Proof of Lemma \ref{S61L1}] Notice that as $u$ varies over $\{ x + \i y \in \mathbb{C}: x = 1, y \in [-\pi, \pi]\}$, the variable $w = e^u$ covers $C_w$. Similarly, as $v$ varies over $\{x + \i y \in \mathbb{C}: x = (3/4)\log \tau, y \in [-\pi, \pi]\}$, the variable $z = e^v$ covers $C_z$. In particular, we see that to prove (\ref{S61E1}) it suffices to show that 
\begin{equation}\label{IO1}
 \Re[F(u)] - \Re[F(v)]  \leq (-1/4) \log \tau,
\end{equation}
where $F$ is as in (\ref{S4DefF}). From (\ref{S4Decay}) we conclude that 
$$ \Re[F(u)] - \Re[F(v)] \leq F(1) - F((3/4) \log \tau) = \frac{1}{1 + e} + \frac{1}{4} - \frac{1}{1 + \tau^{3/4}} - (3/16) \log \tau,$$
which implies (\ref{IO1}) since $\tau \leq e^{-8\pi}$. 
\end{proof}

\begin{proof}[Proof of Lemma \ref{S61L2}] 
Notice that for $w \in C_w$ and $z \in C_z$ we have $\log w = 1 + \i x$, and $\log z= (3/4) \log \tau + \i y$ for some $x,y \in [-\pi, \pi]$. The latter, the fact that $|e^{\i h}| = 1$ for $h \in \mathbb{R}$, and \cite[Equation (6.14)]{ED18} together imply 
\begin{equation*}
\begin{split}
&\left| \frac{ [-\log \tau]^{-1} \pi \cdot e^{ 2m \pi  \i  [(1/4)t +(1/2)( \lfloor t^{2/3}\alpha \rfloor - 1 ) -  t^{1/3} \tilde{r} ]}}{\sin(-\pi  [\log \tau ]^{-1} [ \log z - \log w  -  2m \pi  \i  ])} \right|  =  \frac{ [-\log \tau]^{-1} \pi }{| \sin(-\pi  [\log \tau ]^{-1} [ \log z - \log w  -  2m \pi  \i  ]) |}  \\
& \leq \frac{ [-\log \tau]^{-1} \pi \exp( - \pi [- \log \tau]^{-1} |y - x + 2\pi m| )}{|\sin( 3\pi/2 - 2\pi [-\log \tau]^{-1})|} \leq 2 \pi [-\log \tau]^{-1}\tau^{-2\pi^2} \cdot \tau^{2\pi^2 |m|}.
\end{split}
\end{equation*}
We mention that in the last inequality we used that $\tau \leq e^{-8\pi}$ and so  $|\sin(3\pi/2 - 2\pi [-\log \tau]^{-1})| \geq 1/2$, and also that $x,y \in [-\pi, \pi]$. Summing over $m \in \mathbb{Z}$, we see that (\ref{S61E2}) holds with $B_{\tau,2} = 2 \pi [-\log \tau]^{-1}\tau^{-2\pi^2} \cdot \sum_{m \in \mathbb{Z}} \tau^{2\pi^2 |m|}.$
\end{proof}

\begin{proof}[Proof of Lemma \ref{S61L3}] 
Using that $\tau \in (0, e^{-8\pi}],$ $|z| = \tau^{3/4}$ and $|w| = e$, we see that 
$$ 1 \leq \frac{(1 - \tau^{3/4})\tau^{-3/8} e^{1/2}  }{e + 1} \leq  \frac{|1 + z| \tau^{-3/8} e^{1/2} }{|1 + w|} = \left| \frac{(1 + z )e^{-\log z /2} }{(1 + w) e^{-\log w/2} } \right| \leq \frac{(1 + \tau^{3/4})\tau^{-3/8} e^{1/2}  }{e - 1} \leq \tau^{-1/2}.$$
Combining the latter with the fact that for $t \geq 1$ we have $|\lfloor t^{2/3} \alpha \rfloor -1| \leq 2 t^{2/3} (1 + |\alpha|)$ , we conclude (\ref{S61E3}) with $b_{\tau,3} = - \log \tau$.
\end{proof}

\begin{proof}[Proof of Lemma \ref{S61L4}]
Using that $\tau \in (0, e^{-8\pi}],$ $|z| = \tau^{3/4}$ and $|w| = e$, we see that 
\begin{equation}\label{VS1}
\left| \exp \left( t^{1/3} \tilde{r} [ \log z  -  \log w] \right)  \right| = \exp \left( t^{1/3} \tilde{r} [ (3/4) \log \tau  -  1] \right) \leq \exp \left( - t^{1/3} |\tilde{r}| \log \tau \right).
\end{equation}
In addition, from (\ref{DR1}) and $\tau \in (0, e^{-8\pi}],$ we have
\begin{equation}\label{VS2}
\left|\frac{(-w ;\tau)_{\infty} (z^2  ;\tau)_{\infty} }{(-z ;\tau)_{\infty} (z w  ;\tau)_{\infty} \cdot z } \right| \leq \frac{(-e ;\tau)_{\infty}(- \tau^{3/2}  ;\tau)_{\infty} }{(\tau^{3/4} ;\tau)_{\infty} (e \tau^{3/4} ;\tau)_{\infty}  \cdot \tau^{3/4}} =: B_{\tau,4}.
\end{equation}
Equations (\ref{VS1}), (\ref{VS2}), and the fact that $t \geq 1$ imply (\ref{S61E4}) with $b_{\tau,4} = - \log \tau + \log B_{\tau,4}$. 
\end{proof}

\begin{proof}[Proof of Lemma \ref{S61L5}]
From \cite[Lemma 8.5]{ED2020} with $R = e$, $r = \tau^{3/4}$, $\alpha = e^{-1}\tau^{3/8}$ we have
$$\left| \det \left[ \frac{1}{w_a - z_b} \right]_{a,b = 1}^k \right| \leq  \frac{k^{k/2}}{\tau^{3k/8}(e - 1)^k} \cdot \left( \frac{\tau^{3/8} (1 + e)}{e + \tau^{3/4} }\right)^{k^2}.$$
The latter inequality implies (\ref{S61E5}).
\end{proof}

\begin{proof}[Proof of Lemma \ref{S61L6}]
Using that $|z_a| = \tau^{3/4}$, $|w_a| = e$ for $a = 1, \dots, k$, $\tau \in (0, e^{-8\pi}]$  and (\ref{DR1}), we see that for each $1 \leq a < b \leq k$ we have
\begin{equation*}
\begin{split}
&\left|\frac{(w_aw_b;\tau)_\infty (z_az_b;\tau)_\infty }{(z_aw_b;\tau)_\infty (w_az_b;\tau)_\infty}  \right|  \leq \frac{(-e^{2};\tau)_\infty (- \tau^{3/2};\tau)_\infty }{(e \tau^{3/4};\tau)_\infty (e \tau^{3/4};\tau)_\infty}  \leq \frac{(-e^{2};e^{-8\pi})_\infty (- e^{-12\pi};e^{-8\pi})_\infty }{(e^{1-6\pi};e^{-8\pi})_\infty (e^{1-6\pi};e^{-8\pi})_\infty} .
\end{split}
\end{equation*}
The latter inequality implies (\ref{S61E6}) with $B_6 = \frac{(-e^{2};e^{-8\pi})_\infty (- e^{-12\pi};e^{-8\pi})_\infty }{(e^{1-6\pi};e^{-8\pi})_\infty (e^{1-6\pi};e^{-8\pi})_\infty}$.
\end{proof}

%-------------------------------------------------------------------------------------------------------------------------------------------------------------------------------------------------
% Section 6
%
%-------------------------------------------------------------------------------------------------------------------------------------------------------------------------------------------------
\section{Appendix: Proof of Lemma \ref{S1LWD}}\label{Section6}  In the proof we use the same notation as in Definition \ref{Airy21}. For clarity, we split the proof into four steps.\\

{\bf \raggedleft Step 1.} For $y_1, t_1 \in \mathbb{R}$ we define 
\begin{equation}\label{OR1}
G_{t_1}^{2 \rightarrow 1}(y_1) = \det \left( I - \chi_y K_{\infty} \chi_{y} \right)_{L^2(\{t_1\}\times \mathbb{R})}.
\end{equation}
In this step we show that $G_{t_1}^{2 \rightarrow 1}$ is a distribution function (DF) on $\mathbb{R}$. 

From the displayed equation preceding Remark 1.1 in \cite{QR13}, we see that $G_{t_1}^{2 \rightarrow 1}$ as in (\ref{OR1}) is the same as $G_{t_1}^{2 \rightarrow 1}$ from \cite[(1.7)]{QR13}. In particular, from \cite[Theorem 1]{QR13} we have for $y_1,t_1 \in \mathbb{R}$
\begin{equation}\label{OR2}
G_{t_1}^{2 \rightarrow 1}(y_1) = \mathbb{P} \left( \sup_{x \leq t_1}[ \mathcal{A}_2(x) - x^2] \leq y_1 - t_1^2 \cdot {\bf 1}\{t_1 \leq 0\}  \right),
\end{equation}
where $\mathcal{A}_2$ is the Airy$_2$ process. From the almost sure continuity of the Airy$_2$ process, see \cite[Theorem 1.2]{J03}, we conclude that $G_{t_1}^{2 \rightarrow 1}$ is the DF of an extended real-valued random variable taking values in $(-\infty, \infty]$, and so it suffices to show that 
\begin{equation}\label{OR3}
\lim_{y \rightarrow \infty} G_{t_1}^{2 \rightarrow 1}(y) = 1.
\end{equation}
To prove (\ref{OR3}), we note from (\ref{OR2}) that for each $y \in \mathbb{R}$ and $\tilde{y} = y - t_1^2 \cdot {\bf 1}\{t_1 \leq 0\}$ 
\begin{equation}\label{OR4}
\begin{split}
&F_{1}(4^{1/3} \tilde{y}) = \mathbb{P} \left( \sup_{x \in \mathbb{R}} [ \mathcal{A}_2(x) - x^2] \leq \tilde{y} \right) \leq G_{t_1}^{2 \rightarrow 1}(y),
\end{split}
\end{equation}
where $F_{1}$ is the GOE Tracy-Widom distribution \cite{TW96}. We mention that the equality in (\ref{OR4}) was derived by Johansson \cite{J03} (the factor $4^{1/3}$ is omitted in that paper, see \cite[Theorem 1]{CQR13} for the correct statement). Using that $F_1$ is a DF on $\mathbb{R}$, we have $\lim_{y \rightarrow \infty} F_{1}(y)= 1$. The latter statement and (\ref{OR4}) imply (\ref{OR3}), and so $G_{t_1}^{2 \rightarrow 1}$ is a DF on $\mathbb{R}$. We mention that (\ref{OR3}) also follows from \cite[Proposition 1.2]{QR13}.\\

{\bf \raggedleft Step 2.} In this step we prove the third part of the lemma. By using the Fredholm determinant expansion, we have for all $y_1, t_1 \in \mathbb{R}$
\begin{equation}\label{SF1}
\begin{split}
&G_{t_1}^{2 \rightarrow 1}(y_1) = 1 + \sum_{k = 1}^{\infty} \frac{(-1)^k}{k!} \int_{(y_1,\infty)^k} d \vec{\lambda} \det \left[K_{\infty}(t_1, \lambda_a; t_1, \lambda_b) \right]_{a,b = 1}^k  = 1 + \sum_{k = 1}^\infty H_k, \mbox{ where }\\
&  H_k = \frac{(-1)^k}{(2\pi \i)^{2k} k!}  \int_{(y_1,\infty)^k}d \vec{\lambda}  \int_{C_{1, \pi/4}^k}d\vec{w}  \int_{C_{0, 3\pi/4}^k} d\vec{z}  \det \left[ \frac{-2w_a}{z_b^2 - w_a^2} \right]_{a,b = 1}^k \prod_{a = 1}^k \frac{e^{w_a^3/3 + t_1 w_a^2 - \tilde{\lambda}_a w_a}}{e^{z_a^3/3 + t_1 z_a^2 - \tilde{\lambda}_a z_a }}  ,
\end{split}
\end{equation}
and $\tilde{\lambda}_a = \lambda_a - t_1^2 \cdot {\bf 1}\{ t_1 \leq 0\}$. We mention that in deriving the formula for $H_k$ we used the definition of $K_{\infty}(x,s; y,t) $ from (\ref{QW2}) and the multilinearity of the determinant function.

We next seek to exchange the order of the integrals in $H_k $. We observe that we can find a constant $A \in (0, \infty)$, depending on $y_1, t_1$, such that 
\begin{equation}\label{SF2}
\begin{split}
& \int_{C_{1, \pi/4}}|dw| \int_{C_{0, 3\pi/4}} |dz|  |2w| \cdot \left|\frac{e^{w^3/3 + t_1 w^2 }}{e^{z^3/3 + t_1 z^2  }} \right| \cdot e^{(t_1^2 + |y_1|)(|z| + |w|)} \leq A,
\end{split}
\end{equation}
where $|dz|$, $|dw|$ denote integration with respect to arc length. The latter follows from the cubic terms in the exponential functions.

Note that for $z \in C_{0,3\pi/4}$, $w \in C_{1,\pi/4}$ we have $|z \pm w| \geq 1$. The latter and Hadamard's inequality from Lemma \ref{DetBounds} imply
\begin{equation}\label{SF3}
\left| \det \left[ \frac{-2w_a}{z_b^2 - w_a^2} \right]_{a,b = 1}^k \right| \leq k^{k/2} \cdot \prod_{a = 1}^k |2w_a|.
\end{equation}
In addition, since $\Re(w_a - z_a) \geq 1$, we have 
\begin{equation}\label{SF4}
\int_{y_1}^\infty \left| e^{\tilde{\lambda}_a(z_a - w_a)} \right| d\lambda_a = \frac{\exp\left( t_1^2 \cdot {\bf 1}\{ t_1 \leq 0\} \cdot \Re(w_a - z_a) + y_1 \cdot \Re(z_a - w_a) \right) }{\Re(w_a - z_a)}.
\end{equation}
Combining (\ref{SF3}) and (\ref{SF4}) with the fact that $\Re(w_a - z_a) \geq 1$, we conclude that 
\begin{equation}\label{SF5}
\begin{split}
&\int_{(y_1,\infty)^k} d \vec{\lambda}  \int_{C_{1, \pi/4}^k} \left|d\vec{w} \right|   \int_{C_{0, 3\pi/4}^k}  \left| d\vec{z} \right| \left| \det \left[ \frac{-2w_a}{z_b^2 - w_a^2} \right]_{a,b = 1}^k \prod_{a = 1}^k \frac{e^{w_a^3/3 + t_1 w_a^2 - \tilde{\lambda}_a w_a}}{e^{z_a^3/3 + t_1 z_a^2 - \tilde{\lambda}_a z_a }} \right| \\
& \leq k^{k/2} \cdot \int_{C_{1, \pi/4}^k} \left|d\vec{w} \right|  \int_{C_{0, 3\pi/4}^k}\left| d\vec{z} \right| \prod_{a = 1}^k |2w_a| \cdot \left|\frac{e^{w_a^3/3 + t_1 w_a^2 }}{e^{z_a^3/3 + t_1 z_a^2  }} \right| \cdot e^{(t_1^2 + |y_1|)(|z_a| + |w_a|)}  \leq A^k < \infty,
\end{split}
\end{equation}
where $|d\vec{z}| = |dz_1| \cdots |dz_k|$, $|d\vec{w}| = |dw_1| \cdots |dw_k|$. 

One consequence of (\ref{SF5}) is that 
\begin{equation}\label{SF6}
\begin{split}
|H_k| \leq \frac{A^k k^{k/2}}{(2\pi)^{2k}k!},
\end{split}
\end{equation}
which is summable over $k$. Another consequence of (\ref{SF5}) is that we can exchange the order of the integrals in $H_k $ without affecting the value of the integral by Fubini's theorem. The result is 
\begin{equation}\label{SF7}
\begin{split}
&H_k  = \frac{(-1)^k}{(2\pi \i)^{2k} k!}    \int_{C_{1, \pi/4}^k}  d\vec{w} \int_{C_{0, 3\pi/4}^k} d\vec{z}  \int_{(y_1,\infty)^k} d \vec{\lambda} \det \left[ \frac{-2w_a}{z_b^2 - w_a^2} \right]_{a,b = 1}^k \prod_{a = 1}^k \frac{e^{w_a^3/3 + t_1 w_a^2 - \tilde{\lambda}_a w_a}}{e^{z_a^3/3 + t_1 z_a^2 - \tilde{\lambda}_a z_a }}   \\
& = \frac{1}{(2\pi \i)^{2k} k!}  \int_{C_{1, \pi/4}^k}d\vec{w} \int_{C_{0, 3\pi/4}^k}d\vec{z} \det \left[ \frac{2w_a e^{w_a^3/3 + t_1 w_a^2 +{\bf 1}\{t_1 \leq 0\}  t_1^2 w_a -y_1 w_a  }}{( z_a^2 - w_b^2)(w_a - z_a) e^{z_a^3/3 + t_1 z_a^2 +{\bf 1}\{t_1 \leq 0\}  t_1^2 z_a  - y_1 z_a}} \right]_{a,b = 1}^k  ,
\end{split}
\end{equation}
where we used the multilinearity of the determinant, the fact that the determinant of a matrix is equal to that of its transpose, and that for $\Re(w_a - z_a) \geq 1$ we have
$$\int_{y_1}^\infty e^{\tilde{\lambda}_a (z_a - w_a) } d\lambda_a = \exp \left( t_1^2 {\bf 1}\{ t_1 \leq 0\} (w_a - z_a) + y_1 (z_a - w_a) \right) \cdot \frac{1}{w_a -z_a}.$$
Equations (\ref{SF1}) and (\ref{SF7}) establish (\ref{CrossFD2}), and so we conclude the third part of the lemma.\\

{\bf \raggedleft Step 3.} In this step we prove the second part of the lemma. We first show that for each  $t_1 \in \mathbb{R}$ the function $G_{t_1}^{2 \rightarrow 1}$ from (\ref{OR1}) is continuous. Let $y_N \in \mathbb{R}$ for $N \in \mathbb{N} \cup\{\infty\}$ be such that $\lim_{N \rightarrow \infty} y_N = y_\infty$. We seek to show that 
\begin{equation}\label{BB1}
\lim_{N \rightarrow \infty} G_{t_1}^{2 \rightarrow 1}(y_N) = G_{t_1}^{2 \rightarrow 1}(y_{\infty}).
\end{equation}
We define for $w_1, w_2 \in C_{1, \pi/4}$, $z \in C_{0, 3\pi/4}$ and $N \in \mathbb{N} \cup\{\infty\}$
\begin{equation*}
g_{w_1,w_2}^N(z) = \frac{2w_1 e^{w_1^3/3 + t_1 w_1^2 +{\bf 1}\{t_1 \leq 0\}  t_1^2 w_1 -y_N w_1  }}{( z^2 -  w_2^2)(w_1 - z) e^{z^3/3 + t_1 z^2 +{\bf 1}\{t_1 \leq 0\}  t_1^2 z  - y_N z}}.
\end{equation*}
Let $R > 0$ be sufficiently large so that $|y_N| \leq R$ for all $N \in \mathbb{N} \cup\{\infty\}$. We observe that the functions $g_{w_1,w_2}^N$ satisfy the conditions of \cite[Lemma 2.3]{ED18} with $\Gamma_1 = C_{1, \pi/4}$, $\Gamma_2 = C_{0, 3\pi/4}$ and functions
$$F_1(w) = \left|2w e^{w^3/3 + t_1 w^2 +{\bf 1}\{t_1 \leq 0\}  t_1^2 w  } \right| \cdot e^{|R| |w|}, \hspace{2mm } F_2(z) = \left| e^{-z^3/3 - t_1 z^2 - {\bf 1}\{t_1 \leq 0\}  t_1^2 z} \right| \cdot e^{|R||z|}.$$
From \cite[Lemmas 2.2 and 2.3]{ED18} we conclude that 
\begin{equation}\label{BB2}
\lim_{N \rightarrow \infty} \det (I + K^N)_{L^2(C_{1, \pi/4})} = \det (I + K^\infty)_{L^2(C_{1, \pi/4})},\mbox{ where } K^N(w_1,w_2) = \int_{C_{0, 3\pi/4}}\hspace{-5mm} g_{w_1,w_2}^N(z)  dz.
\end{equation}
From (\ref{CrossFD2}), which we proved in Step 2 above, and (\ref{OR1}) we see that 
$$\det (I + K^N)_{L^2(C_{1, \pi/4})} = G_{t_1}^{2 \rightarrow 1}(y_N) \mbox{ for all $N \in\mathbb{N} \cup\{\infty\}$,}$$
which in view of (\ref{BB2}) proves (\ref{BB1}).\\

Let us fix $m \in \mathbb{N}$, $t_1 < \cdots < t_m$ and define for $\vec{y} = (y_1, \dots, y_m) \in \mathbb{R}^m$ the function
\begin{equation}\label{BB3}
f_m(\vec{y}; \vec{t}) = f_m(y_1, \dots, y_m; t_1, \dots, t_m) :=\det \left( I - \chi_y K_{\infty} \chi_{y} \right)_{L^2(\{t_1, \dots, t_m\}\times \mathbb{R})}.
\end{equation}
In the remainder of this step we prove that, for fixed $\vec{t}$, $f_m(\cdot ;  \vec{t})$ is a continuous function on $\mathbb{R}^m$. Let $\vec{y}^N \in \mathbb{R}^m$ for $N \in \mathbb{N} \cup\{\infty\}$ be such that $\lim_{N \rightarrow \infty} \vec{y}^N = \vec{y}^\infty$. We seek to show that 
\begin{equation}\label{BB4}
\lim_{N \rightarrow \infty} f_m(\vec{y}^N; \vec{t}) =   f_m(\vec{y}^\infty;  \vec{t}).
\end{equation}
In principle, it is possible to use the Fredholm determinant expansion formula in the right side of (\ref{BB3}) and repeat our work from Step 2 and the first part of Step 3 to prove (\ref{BB4}). As the computations are quite involved, we instead deduce the continuity of $f_m$ from the continuity of $f_1$ (i.e. $G_{t_1}^{2 \rightarrow 1}$), which we showed in (\ref{BB1}) and the fact that the $f_m$ arise as limits of distribution functions of the process $X_t$ in \cite{BFS08}. We provide the details below.\\

Let $X_t$ be the process, defined in \cite[Equation (2.5)]{BFS08}. From the work in \cite[Section 4]{BFS08} we have that for each $m \in \mathbb{N}$, $t_1 < t_2 < \cdots < t_m$  and $y_1, \dots, y_m \in \mathbb{R}$
\begin{equation}\label{BB5}
\lim_{n \rightarrow \infty} \mathbb{P} \left( \cap_{k = 1}^m \{ X_n(t_k) \leq y_k \} \right) = f_m(\vec{y};  \vec{t}).
\end{equation}
Equation (\ref{BB5}) shows that $f_m(\vec{y}; \vec{t}) \in [0,1]$ for each $\vec{y} \in \mathbb{R}^m$, and for fixed $\vec{t}$ is non-decreasing in the $y$ variables. Let us define the sequences $\vec{x}^N, \vec{X}^N \in \mathbb{R}^m$ through
$$\vec{x}^N = \left(\min(y_1^N, y_1^\infty), \dots, \min(y_m^N, y_m^\infty) \right), \mbox{ and }\vec{X}^N = \left(\max(y_1^N, y_1^\infty), \dots, \max(y_m^N, y_m^\infty) \right).$$
As $f_m$ is non-decreasing, we see that for each $N \in \mathbb{N}$
\begin{equation}\label{BB6}
\left|f_m(\vec{y}^N; \vec{t})  - f_m(\vec{y}^\infty; \vec{t}) \right| \leq f(\vec{X}^N; \vec{t}) - f(\vec{x}^N; \vec{t}).
\end{equation}
In addition, by subadditivity we have for each $N \in \mathbb{N} \cup \{\infty\}$
$$ \mathbb{P} \left( \cap_{k = 1}^m \{  X_n(t_k) \leq X^N_k \} \right) - \mathbb{P} \left( \cap_{k = 1}^m \{  X_n(t_k) \leq x^N_k \} \right) \leq \sum_{k = 1}^m \mathbb{P} \left(   x^N_k< X_n(t_k) \leq X^N_k \right).$$
Taking the limit $n \rightarrow \infty$ in the last inequality, using (\ref{BB5}), we get
\begin{equation}\label{BB7}
 f(\vec{X}^N; \vec{t}) - f(\vec{x}^N; \vec{t}) \leq \sum_{k = 1}^m [f_1(X_k^N; t_k) - f_1(x_k^N; t_k)].
\end{equation}
Since $f_1$ is continuous from (\ref{BB1}), and $\lim_{N \rightarrow \infty} \vec{X}^N = \lim_{N \rightarrow \infty} \vec{x}^N = \vec{y}^\infty$, we may conclude (\ref{BB4}) by combining (\ref{BB6}) and (\ref{BB7}).\\

{\bf \raggedleft Step 4.} In this final step we prove the first part of the lemma. From our work in Steps 1 and 3, we know that for each $t_1 \in \mathbb{R}$, the function $f_1(\cdot; t_1)$ as in (\ref{BB3}) is a continuous DF on $\mathbb{R}$. This and (\ref{BB5}) show that $X_n(t_1)$ weakly converge to a random variable with DF $f_1(\cdot; t_1)$ as $n \rightarrow \infty$. In particular, we conclude that $X_n(t_1)$ is a tight sequence of random variables. The latter implies that for all $m \in \mathbb{N}$ and $t_1 < t_2 < \cdots < t_m$, we have $(X_n(t_1), \dots, X_n(t_m))$ is a tight sequence of random vectors in $\mathbb{R}^m$. The tightness of $(X_n(t_1), \dots, X_n(t_m))$, equation (\ref{BB5}) and the continuity of $f_m(\cdot; \vec{t})$ together imply that $f_m(\cdot; \vec{t})$ is a continuous DF on $\mathbb{R}^m$ and $(X_n(t_1), \dots, X_n(t_m))$ converge weakly to a random vector in $\mathbb{R}^m$ with DF $f_m(\cdot; \vec{t})$ as $n \rightarrow \infty$.

For $x \in \mathbb{R}$, we let $\mathbb{R}_x$ denote a copy of $\mathbb{R}$ and endow it with the Borel $\sigma$-algebra $\mathcal{B}(\mathbb{R}_x)$. If $I \subset \mathbb{R}$ is a finite set, and $I = \{t_1, \dots, t_m\}$, with $t_1 < \cdots < t_m$, we let $S_I = \times_{x \in I} \mathbb{R}_i$ and endow $S_I$ with the product $\sigma$-algebra $\otimes_{x \in I} \mathcal{B}(\mathbb{R}_x)$. We also let $\mu_I$ denote the unique measure on $\left(\times_{x \in I} \mathbb{R}_x, \otimes_{x \in I} \mathcal{B}(\mathbb{R}_x) \right)$, whose DF is given by 
$$\mu_I\left( \prod_{i = 1}^m (-\infty, y_i]_{t_i} \right) = f_m(y_1, \dots, y_m; \vec{t}).$$
Here, we have placed the subscript $t_i$ to indicate that $(-\infty, y_i]_{t_i} \subset \mathbb{R}_{t_i}$. Since (\ref{BB5}) holds for each $m \in \mathbb{N}$, $t_1 < t_2 < \cdots < t_m$  and $y_1, \dots, y_m \in \mathbb{R}$, we conclude that for finite sets $I \subseteq J \subset \mathbb{R}$, we have
$$\mu_J(\cdot \times S_{J \setminus I}) = \mu_I.$$
From the Kolmogorov existence theorem, see \cite[Theorem 5.16]{Ka}, we conclude that there is a probability space $(\Omega, \mathcal{F}, \mathbb{P})$ and a real-valued process $\{ \mathcal{A}_{2 \rightarrow 1}(t): t \in \mathbb{R}\}$ on that space, whose finite-dimensional distribution functions are given by $f_m$. This proves the first part of the lemma.

\bibliographystyle{alpha}
\bibliography{PD}

\end{document}